\pdfoutput=1
\documentclass[12pt]{amsart}
\usepackage{lmodern} 
\usepackage[scale=0.89]{tgheros} 
\usepackage[osf]{Baskervaldx} 
\usepackage[baskervaldx,cmintegrals,bigdelims,vvarbb]{newtxmath} 
\usepackage[cal=boondoxo]{mathalfa} 
       
  \usepackage{amsmath}          
  \usepackage{amsfonts}           
  \usepackage{amscd}
  \usepackage{enumerate}
  \usepackage[active]{srcltx}
  \usepackage[percent]{overpic}
  \usepackage[mathscr]{eucal}
  \usepackage{graphicx,epsfig} 
  \usepackage{tikz-cd}        
  \usepackage{color}
  \usepackage{xcolor}
  \usepackage{marginnote}
  \usepackage{hyperref}
  \usepackage{bm,bbm}
  \usepackage{enumitem} 
  \usepackage{tikz-cd}
  \usepackage{thmtools}

\tikzset{>=latex,auto}
\usetikzlibrary{positioning,shapes.geometric}

\newtheorem{theorem}{Theorem}[section]
\newtheorem{lemma}[theorem]{Lemma}
\newtheorem{corollary}[theorem]{Corollary}

\theoremstyle{definition}
\newtheorem{definition}[theorem]{Definition}
\newtheorem{example}[theorem]{Example}

\theoremstyle{remark}
\newtheorem{remark}[theorem]{Remark}

\theoremstyle{thmx}
\newtheorem{thmx}{Theorem}

\declaretheoremstyle[numbered=no,bodyfont = \itshape, spaceabove=7pt,spacebelow=7pt]{theorem}
\declaretheorem[name=Basic Structure Lemma, style=theorem]{redlm}

\numberwithin{equation}{section}

\newcommand{\thmref}[1]{Theorem~\ref{#1}}

\newcommand{\remref}[1]{Remark~\ref{#1}}

\newcommand{\lemref}[1]{Lemma~\ref{#1}}

\newcommand{\corref}[1]{Corollary~\ref{#1}}
\newcommand{\figref}[1]{Fig.~\ref{#1}}
\newcommand{\exref}[1]{Example~\ref{#1}}

\newcommand{\vs}{\vspace{6pt}}

\newcommand{\CC}{{\mathbb C}}

\newcommand{\RR}{{\mathbb R}}
\newcommand{\HH}{{\mathbb H}}
\newcommand{\TT}{{\mathbb T}}
\newcommand{\ZZ}{{\mathbb Z}}

\newcommand{\DD}{{\mathbb D}}

\newcommand{\sS}{{\mathscr S}}
\newcommand{\sC}{{\mathscr C}}
\newcommand{\sK}{{\mathscr K}}
\newcommand{\sP}{{\mathscr P}}

\newcommand{\sU}{{\mathscr U}}

\newcommand{\sL}{{\mathscr L}}

\newcommand{\Chat}{\hat{\CC}}
\newcommand{\ga}{\gamma}
\newcommand{\modd}{\operatorname{mod}}

\newcommand{\id}{\operatorname{id}}

\newcommand{\ve}{\varepsilon}
\newcommand{\de}{\delta}
\newcommand{\es}{\emptyset}
\newcommand{\sm}{\smallsetminus}
\newcommand{\bd}{\partial}
\newcommand{\ov}{\overline}

\newcommand{\resit}{\operatorname{r\'esit}}
\newcommand{\res}{\operatorname{res}}

\newcommand{\myre}{\operatorname{Re}}
\newcommand{\myim}{\operatorname{Im}}

\newcommand{\cara}{Carath\'eodory}
\newcommand{\Om}{\Omega}
\newcommand{\La}{\Lambda}

\newcommand{\e}{\mathrm{e}}
\newcommand{\ii}{\mathrm{i}}

\makeatletter
\newcommand{\oset}[3][0ex]{%
  \mathrel{\mathop{#3}\limits^{
    \vbox to#1{\kern-3\ex@
    \hbox{$\scriptstyle#2$}\vss}}}}
\makeatother

\DeclareRobustCommand\longtwoheadrightarrow
     {\relbar\joinrel\twoheadrightarrow}

\newcommand{\ct}{\longtwoheadrightarrow}

\newcommand{\aeq}{\oset[-0.2ex]{\circ}{=}}

\newcommand{\dist}{\operatorname{dist}}

\newcommand{\Aut}{\operatorname{Aut}}

\DeclareMathAlphabet{\mathlib}{OT1}{LinuxLibertineT-OsF}{m}{it}

\definecolor{myblue}{rgb}{0.06, 0.75, 0.99}
\definecolor{mygreen}{rgb}{0.14, 0.75, 0.30}
\definecolor{mypurple}{rgb}{0.74, 0.12, 0.99}

\makeatletter
\@namedef{subjclassname@2020}{%
  \textup{2020} Mathematics Subject Classification}
\makeatother

\newcommand{\bit}{\it \bfseries}

\begin{document}

\title[Hausdorff limits of external rays]{Hausdorff limits of external rays: the topological picture}

\author[C. L. Petersen and S. Zakeri]{Carsten Lunde Petersen and Saeed Zakeri}

\address{Department of Mathematical Sciences, University of Copenhagen, Universitetsparken 5, DK-2100 Copenhagen {{\O}}, Denmark}

\email{lunde@math.ku.dk}

\address{Department of Mathematics, Queens College of CUNY, 65-30 Kissena Blvd., Queens, New York 11367, USA} 
\address{The Graduate Center of CUNY, 365 Fifth Ave., New York, NY 10016, USA}

\email{saeed.zakeri@qc.cuny.edu}


\date{\today}

\begin{abstract}
We study Hausdorff limits of the external rays of a given periodic angle along a convergent sequence of polynomials of degree $d \geq 2$ with connected Julia sets.     
\end{abstract}

\maketitle

\setcounter{tocdepth}{1}

\tableofcontents

\section{Introduction}

This paper investigates Hausdorff limits of the external rays of a given periodic angle along a convergent sequence of polynomials of degree $\geq 2$ with connected Julia sets. This is a basic question in the context of {\it geometric limits} of conformal dynamical systems, but it is particularly motivated by our work in \cite{PZ2} and its higher degree analogs where the limbs of connectedness loci are defined by patterns of co-landing rays and questions about such geometric limits arise naturally. \vs

Let $\sC(d)$ be the connectedness locus of all monic polynomial maps $\CC \to \CC$ of degree $d \geq 2$. We denote the Julia set and filled Julia set of $P \in \sC(d)$ by $J_P$ and $K_P$. The external ray of $P$ at angle $\theta \in \RR/\ZZ$ is denoted by $R_{P,\theta}$. \vs

Consider a convergent sequence $P_n \to P$ in $\sC(d)$. Fix an angle $\theta$ which has period $q$ under the endomorphism $t \mapsto d t \, (\modd \ZZ)$ of the circle. Let $\zeta_n$ and $\zeta$ be the landing points of the external rays $R_n:=R_{P_n,\theta}$ and $R:=R_{P,\theta}$. After passing to a subsequence, we may assume $\zeta_n \to \zeta_\infty$ 
and $\ov{R_n}:=R_n \cup \{ \zeta_n, \infty \} \to \sL$ in the Hausdorff metric on compact subsets of the Riemann sphere $\Chat$. Our primary goal is to analyze the possible structures for the continuum $\sL$. It is well-known that if the landing point $\zeta$ of $R$ is repelling, then $\sL=\ov{R}:=R \cup \{ \zeta, \infty \}$. However, if $\zeta$ is parabolic, then $\sL$ can be strictly larger than $\ov{R}$, depending on the choice of perturbations $P_n$. This phenomenon was observed as early as the 1990's, for example in the works of Goldberg and Milnor on the fixed point portraits of polynomials \cite{GM}, Oudkerk on the gate structure of near-parabolic points \cite{O}, and Douady and Lavaurs on parabolic implosion \cite{D,La} (compare \figref{pert}). Facets of the phenomenon appears in the work of Pilgrim and Tan Lei on spinning deformations of rational maps \cite{PT}. The problem of Hausdorff limits of $\ov{R_n}$ when the perturbations $P_n$ are postcritically finite is also tackled in the paper of Gao and Tiozzo on the core entropy of polynomials \cite[\S 9]{GT}.\vs

\begin{figure}[t]
\centering 
\begin{overpic}[width=\textwidth]{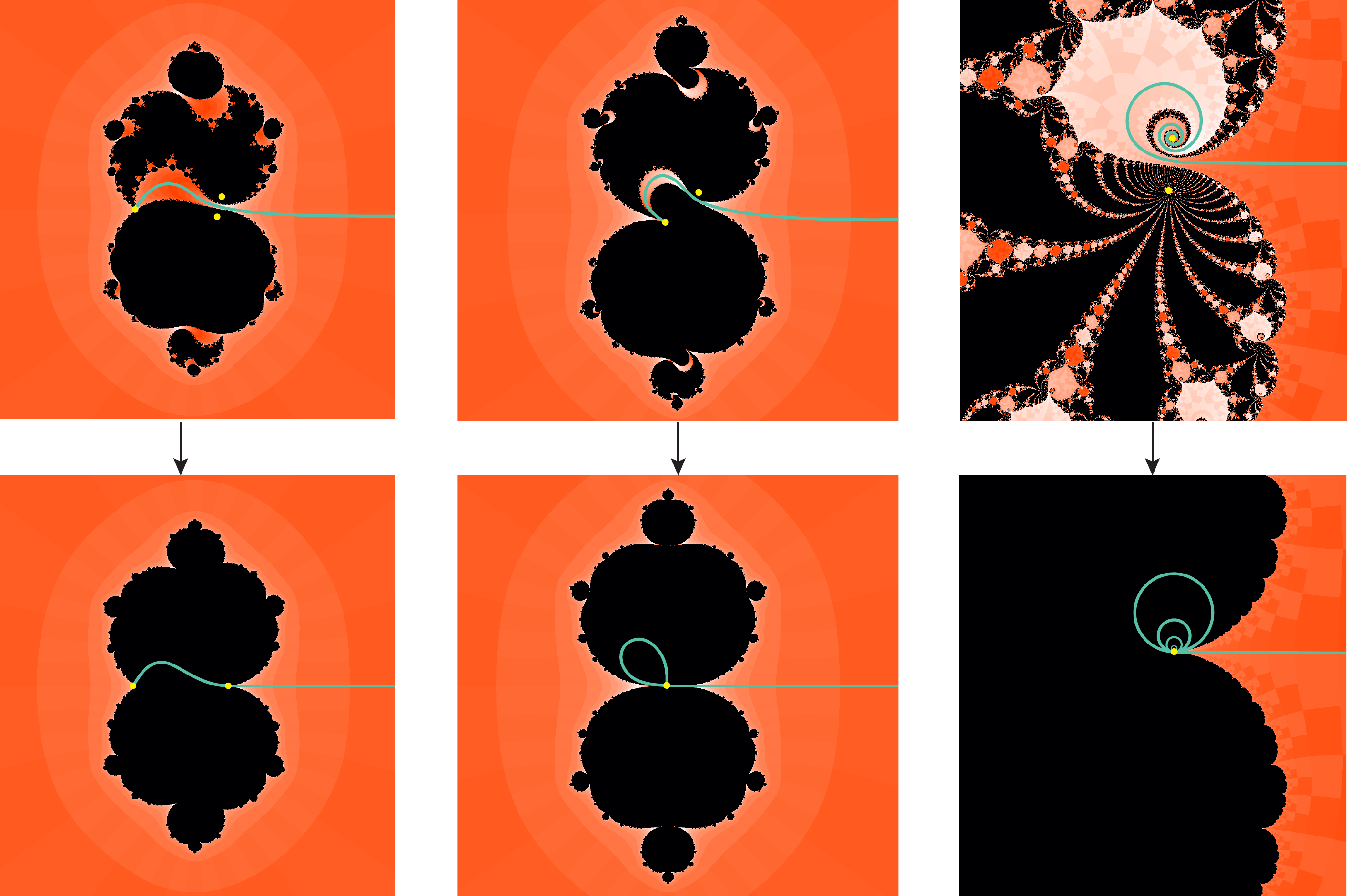}
\put (26,51.2) {\footnotesize{$R_n$}}
\put (63.4,51.1) {\footnotesize{$R_n$}}
\put (97,55.3) {\footnotesize{$R_n$}}
\put (24,16.5) {\footnotesize{$\sL$}}
\put (60,16.5) {\footnotesize{$\sL$}}
\put (94.5,19) {\footnotesize{$\sL$}}
\end{overpic}
\caption{\footnotesize Examples of the Hausdorff limit $\sL$ of the closed external ray $\ov{R_n}$ at angle $\theta=0$. Left: perturbations of a cubic with a non-degenerate parabolic fixed point, where $\sL \cap K_P$ is an embedded arc; Middle: perturbations of the cubic $z+z^3$ with a degenerate parabolic fixed point at $0$, where $\sL \cap K_P$ is a loop; Right: perturbations of the quadratic map $z+z^2$, where $\sL \cap K_P$ is a ``Hawaiian earring.''}
\label{pert}
\end{figure}
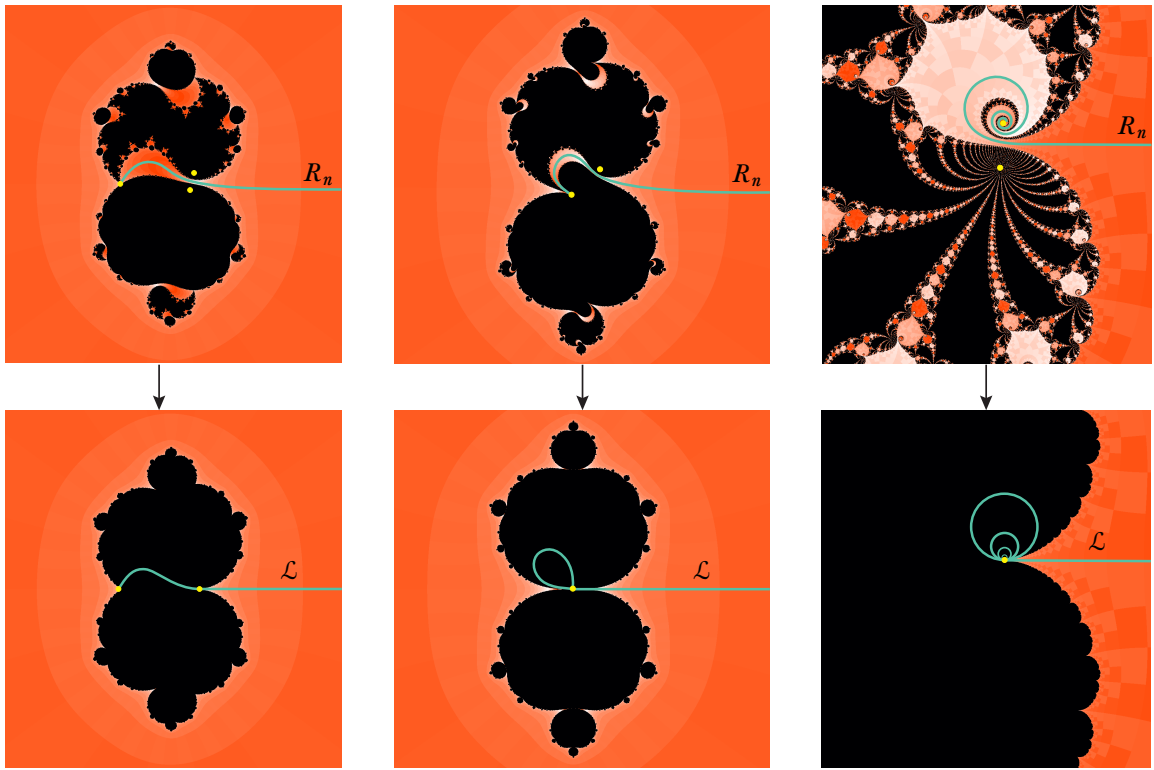   

It is easy to see that $\sL$ is a $P^{\circ q}$-invariant continuum containing $\zeta$ and $\zeta_\infty$, with $\sL \sm K_P = R \cup \{ \infty \}$. The following lemma gives the main reduction for our analysis of $\sL$. It shows that $\sL \cap J_P$ is finite and $\sL \cap \mathring{K}_P$ is a disjoint union of real-analytic arcs. \vs

\noindent
{\bf Convention.} Throughout this paper by a {\bit parabolic basin} we mean a connected component of the immediate basin of attraction of a parabolic periodic point. 

\begin{redlm} 
Let $u \in \sL \cap K_P$. 
\begin{enumerate}[leftmargin=*]
\item[(i)]
If $u \in J_P$, then $P^{\circ q}(u)=u$. \vs 
\item[(ii)]
If $u \in \mathring{K}_P$, then $u$ has a simply connected neighborhood $V=P^{\circ q}(V)$ contained in a parabolic basin $B$ and the action of $P^{\circ q}$ on $V$ is conjugate to a hyperbolic translation. More precisely, there is a biholomorphism $\psi: V \to \{ z \in \CC: \myre(z)>0 \}$, normalized by $\psi(u)=1$, which satisfies 
$$
\psi \circ P^{\circ q} = d^q \psi \qquad \text{in} \ V. 
$$
The real analytic arc $\gamma: \, ]0,+\infty[ \to V$ defined by $\gamma(t)=\psi^{-1}(t)$ is contained in $\sL$ and both limits
$$
w^-(\gamma):=\lim_{t \to 0} \gamma(t) \qquad \text{and} \qquad  w^+(\gamma):=\lim_{t \to +\infty} \gamma(t)
$$ 
exist and are fixed under $P^{\circ q}$, with $w^+(\gamma)$ parabolic of multiplier $1$. If $w^+(\gamma) \neq w^-(\gamma)$, the basin $B$ contains at least two critical points of $P^{\circ q}$.
\end{enumerate}
\end{redlm}
 
\noindent
This is proved in \S \ref{larc}. The main idea is to extract \cara \ limits of the pointed disks $(\Chat \sm K_{P_n}, u_n)$, where $u_n \to u \in \sL$. Statements of a similar nature have appeared in the thesis of A. Deniz \cite[Propositions 4.2.7 and 4.2.8]{De} and, in a different but related context, in the work of Bonifant, Milnor and Sutherland on the relative Green's function \cite[Lemma 3.4]{BMS}. The technique of using \cara \ convergence of pointed disks is also used in Luo's work on degenerations of Blaschke products to study limits of quasi-invariant trees \cite[\S 6]{Lu}. \vs      

The Lemma shows that $\sL \cap \mathring{K}_P$ is partitioned into $P^{\circ q}$-invariant real analytic arcs in parabolic basins that have well-defined initial and end points in $J_P$ and are naturally oriented by the dynamics. For simplicity, each such arc will be called an {\bit $\sL$-arc}. Every $\sL$-arc $\gamma$ is contained in an invariant strip $V$ in which $P^{\circ q}$ acts as $z \mapsto d^q z$. Moreover, $V$ isolates $\gamma$ from all other $\sL$-arcs in the sense that $V \cap \sL = \gamma$ (\lemref{inter}). In particular, $\sL \cap \mathring{K}_P$ is the disjoint union of at most countably many $\sL$-arcs. We call $\gamma$ a {\bit heteroclinic} arc if $w^-(\gamma) \neq w^+(\gamma)$, and a {\bit homoclinic} arc if $w^-(\gamma)=w^+(\gamma)$. A maximal nested chain of homoclinic arcs is called an {\bit earring} (see \figref{band}). It is a finite or countably infinite nested collection of homoclinics, all sharing the same initial and end point $w$. We often say that such an earring, or each of its homoclinic arcs, is {\bit based at} $w$. An earring with infinitely many homoclinic arcs is referred to as a {\bit Hawaiian earring}. 

\begin{figure}[t]
\begin{overpic}[width=\textwidth]{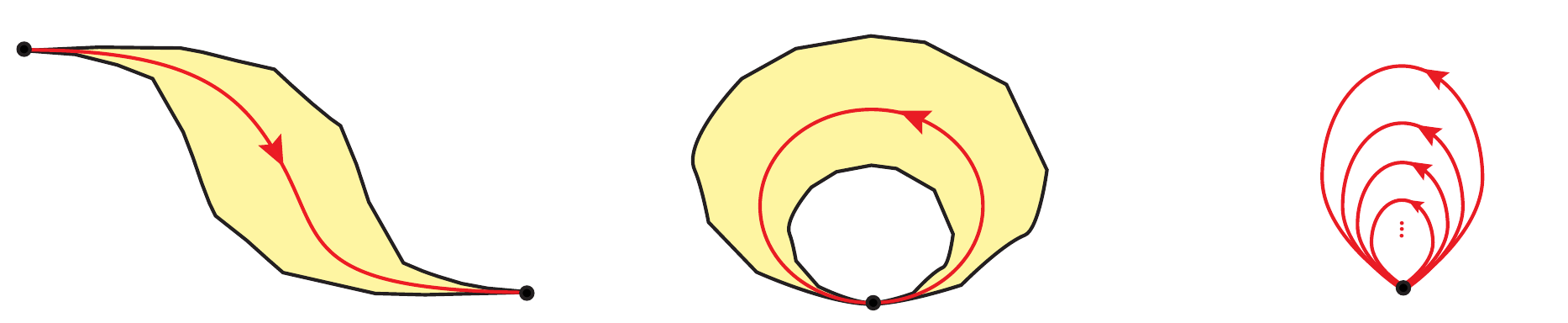}
\put (-2,17) {\footnotesize{$w^-$}}
\put (35,1) {\footnotesize{$w^+$}}
\put (51.4,-1) {\footnotesize{$w^-=w^+$}}
\put (89,0) {\footnotesize{$w$}}
\put (19,10) {\small{\color{red}$\gamma$}}
\put (60,14) {\small{\color{red}$\gamma$}}
\put (15,7) {\small{$V$}}
\put (47,12) {\small{$V$}}
\end{overpic}
\caption{\footnotesize From left to right: a heteroclinic arc, a homoclinic arc, and an earring based at $w$. The arrows indicate the natural dynamical orientation.}
\label{band}
\end{figure}
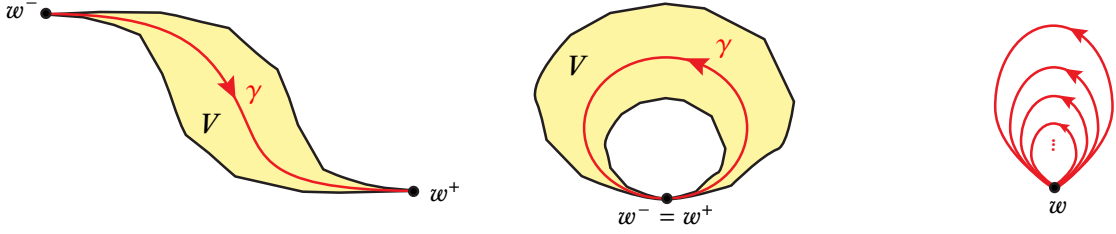

\begin{thmx}[A trichotomy for $\sL$]\label{A}
We have one of the following possibilities: 

\begin{enumerate}[leftmargin=*]
\item[$\bullet$]
{\bit The tame case:} $\sL=\ov{R}:=R \cup \{ \zeta, \infty \}$. Then $\zeta = \zeta_\infty$, and this point can be either repelling or parabolic. \vs 

\item[$\bullet$]
{\bit The semi-wild case:} $\sL \supsetneq \ov{R}$ and $\zeta=\zeta_{\infty}$. Then there are no heteroclinic arcs in $\sL$, but $\zeta=\zeta_{\infty}$ is the endpoint of at least one homoclinic arc. \vs

\item[$\bullet$]
{\bit The wild case:} $\sL \supsetneq \ov{R}$ and $\zeta \neq \zeta_\infty$. Then the set of heteroclinic arcs in $\sL$ is non-empty and finite. Moreover, we can label the heteroclinics as $\gamma_1, \ldots, \gamma_N$ and the points of $\sL \cap J_P$ as $ w_0=\zeta, w_1, \ldots, w_N=\zeta_\infty$ such that    
$$
w^+(\gamma_j) =w_{j-1} \quad \text{and} \quad w^-(\gamma_j) =w_j  \quad \text{for all} \ 1 \leq j \leq N.
$$
\end{enumerate}
\end{thmx}

Define the {\bit spine} $\sL^*$ of $\sL$ as the union of $\ov{R}$ together with the heteroclinic arcs $\gamma_1, \ldots, \gamma_N$ (if any) and their endpoints $w_0=\zeta, w_1, \ldots, w_N=\zeta_\infty$. If there are no heteroclinics (tame or semi-wild cases), then $N=0$ and the spine reduces to $\ov{R}$.  

\begin{thmx}[Anatomy of $\sL$]\label{B}
Either $\sL=\sL^*$ or $\sL \sm \sL^*$ is a union of finitely many earrings of homoclinic arcs based at $w_0, \ldots, w_N$. Each parabolic basin of $w_j$ that meets $\sL$ contains either a unique heteroclinic arc or a unique earring of homoclinic arcs, but not both. Moreover, if $N \geq 1$, any earring based at $w_0, \ldots, w_{N-1}$ must consist of a single homoclinic arc. 
\end{thmx}

Compare \figref{spine}. \vs 

The proof of \thmref{A} is carried out in three stages by verifying the following statements: \vs

$\bullet$ Each parabolic basin contains at most finitely many heteroclinics and at most two earrings of homoclinics (both counts will eventually be sharpened to at most one). The proof uses Fatou coordinates and a modulus argument (\S \ref{modo}). \vs

$\bullet$ $\sL$ and all its connected subsets are arcwise-connected and $\sL^*$ is a finite graph embedded in $\Chat$, with $\sL \cap J_P$ and $\infty$ as its vertices and the heteroclinics and $R$ as its edges (\S \ref{wise}). \vs

\begin{figure}[t]
\centering 
\begin{overpic}[width=0.8\textwidth]{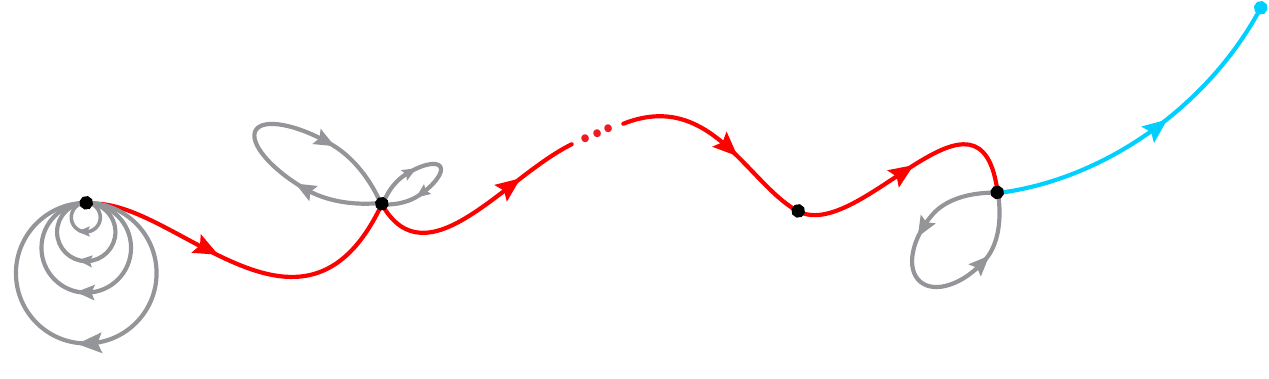}
\put (89,21) {\small{\color{myblue}$R$}}
\put (98.3,30.2) {\small{\color{myblue}$\infty$}}
\put (79.5,12.5) {\footnotesize{$w_0=\zeta$}}
\put (61,10.4) {\footnotesize{$w_1$}}
\put (29.2,9.5) {\footnotesize{$w_{N-1}$}}
\put (0,15) {\footnotesize{$w_N=\zeta_{\infty}$}}
\put (69,17.6) {\small{\color{red}$\gamma_1$}}
\put (56,21) {\small{\color{red}$\gamma_2$}}
\put (40,12) {\small{\color{red}$\gamma_{N-1}$}}
\put (16,6.5) {\small{\color{red}$\gamma_N$}}
\end{overpic}
\caption{\footnotesize Schematic picture of $\sL$, as in \thmref{B}. It consists of the spine $\sL^*$ (union of blue and red) and possible additional earring decorations (in gray). All the earrings based at $w_0, \ldots, w_{N-1}$ consist of a single homoclinic arc. Thus, ``Hawaiian earrings'' with infinitely many loops can only occur at the far left point $w_N=\zeta_{\infty}$.}
\label{spine}
\end{figure}
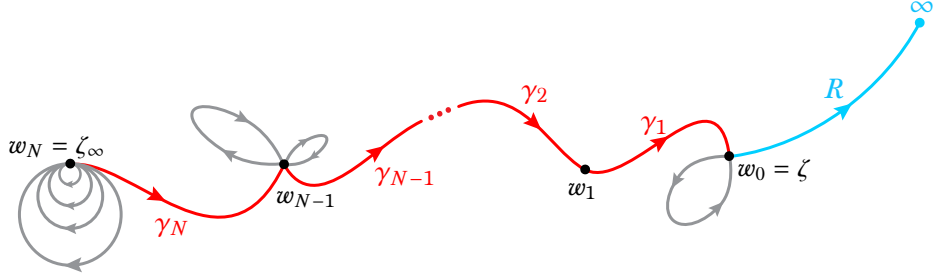

$\bullet$ Every point of $\sL \cap J_P$ is the initial or end point of at most one heteroclinic or $R$. As a result, the spine $\sL^*$ is a finite tree with vertices of degree $1$ or $2$, from which it easily follows that $\sL^*$ is topologically a closed arc (\S \ref{oshet}). The arguments here make use of the notion of {\bit intrinsic potential order} on $\sL \cap \mathring{K}_P$, which can be described as follows: Let $R_n(s)$ denote the point on the external ray $R_n$ at Green's potential $s>0$. If $u, u' \in \sL \cap \mathring{K}_P$, there are sequences of potential $s_n, s'_n \to 0$ such that $R_n(s_n) \to u$ and $R_n(s'_n) \to u'$, and these sequences are well defined up to multiplication by sequences that tend to $1$. We declare $u<u'$ if and only if $\lim_{n \to \infty} s_n/s'_n<1$. This puts a linear order on the union of $\sL$-arcs, where the above limit belongs to $]0,1[$ if $u,u'$ are on the same $\sL$-arc and is $0$ otherwise (\S \ref{ssord}). \vs

The proof of \thmref{B} depends on the more delicate analysis of the order and structure of homoclinics that is carried out in \S \ref{oshom}. \vs

We say that two earrings of homoclinic arcs in $\sL$ are {\bit equivalent} if they belong to the same cycle of parabolic basins. Let 
\begin{align*}
N & := \text{number of heteroclinic arcs in} \ \sL \\
M & := \text{number of earrings in} \ \sL \\
M^\# & := \text{number of equivalence classes of earrings in} \ \sL.
\end{align*}
In \S \ref{BD} we prove  

\begin{thmx}[Bounding the complexity of $\sL$]
\label{C} 
The following inequality holds:
\begin{equation}\label{BB1}
2N + M^\# \leq d-1.
\end{equation}
If $p$ is the period of $\zeta_{\infty}$ and $\nu$ is the degeneracy order of $\zeta_{\infty}$ as a fixed point of $P^{\circ p}$, then  
\begin{equation}\label{BB2}
2N + M \leq d-1 + \Big( \frac{q}{p}-1 \Big) \nu.
\end{equation}
In particular, if $\zeta_{\infty}$ is either repelling so $\nu=0$ or has top period $p=q$, then  
$$
2N + M \leq d-1.
$$
\end{thmx}

Thus, for quadratic polynomials the only possibilities for $(N,M^\#)$ are $(0,0)$ which is tame, and $(0,1)$ which is semi-wild. For cubics, the only possibilities for $(N,M^\#)$ are $(0,0)$ which is tame, $(0,1), (0,2)$ which are semi-wild, and $(1,0)$ which is wild. Compare \figref{double} for an example of the case $(N,M^\#)=(0,2)$. \vs 

If $B=P^{\circ q}(B)$ is a parabolic basin that meets $\sL$, then by classical Fatou-Julia theory the union $B \cup P(B) \cup \cdots \cup P^{\circ q-1}(B)$ contains at least one critical point of $P$. By the Basic Structure Lemma, if $B$ meets $\sL$ along a heteroclinic arc, this union contains at least two critical points of $P$. Thus, to prove the bound \eqref{BB1} it suffices to show that the critical points designated this way are not shared between distinct heteroclinics or between a heteroclinic and an earring. While the former is rather easy to verify, the latter requires a more in-depth investigation, which is carried out in \S \ref{BD}. The proof of the second bound \eqref{BB2} depends on verifying that if $N \geq 1$, the homoclinics based at $w_0, \ldots, w_{N-1}$ are in distinct equivalence classes (\thmref{only1}). \vs 

The bounds in \thmref{C} are indeed optimal. For example, it is not hard  to see that for any $d \geq 3$ there is a sequence of perturbations $P_n(z)=\lambda_n z+z^d$ of $P(z)=z+z^d$ for which the Hausdorff limit of the closed fixed rays $\ov{R_{P_n,0}}$ has $N=0$ and $M^\#=M=d-1$ (compare \figref{double} for the case $d=3$). To this end, consider the map $Q(w)=w(1+w)^{d-1}=w+(d-1)w^2+O(w^3)$ which has a non-degenerate parabolic fixed point at the origin at which the ray $R_{Q,0}$ lands. For the perturbations $Q_n(w)=w(\lambda_n+w)^{d-1}$ with $|\lambda_n-d/(d-1)|<1/(d-1)$ the parabolic fixed point bifurcates into a pair of fixed points: $0$ which is now repelling of multiplier $\lambda_n^{d-1}$, and a nearby fixed point at $w_n=1-\lambda_n$ which is attracting of multiplier $d-(d-1)\lambda_n$. Moreover, the critical point at $-\lambda_n$ (of multiplicity $d-2$) maps to $0$, so the remaining simple critical point at $-\lambda_n/d$ must belong to the attracting basin of $w_n$. It follows that $K_{Q_n}$ is connected. It is well known from the theory of parabolic implosions that for suitable sequences $\lambda_n \to 1$ (say, along the circle $|\lambda_n-(1+a)|=a$ for some $0<a<1/(d-1)$), the closed rays $\ov{R_{Q_n,0}}$ spiral down the fixed point $0$ and converge to a Hausdorff limit which contains a single Hawaiian earring in the unique parabolic basin of $Q$ at $0$. Lifting the picture under the branched covering map $w=z^{d-1}$ which semi-conjugates $P_n$ to $Q_n$ and $P$ to $Q$, it follows that the closed rays $\ov{R_{P_n,0}}$ converge to a Hausdorff limit which has $d-1$ Hawaiian earrings, each residing in one of the $d-1$ invariant parabolic basins of $P$ at $0$. \vs

Another idea, which we present in \S \ref{ex}, is to achieve the optimal bound at the other end of the spectrum by constructing polynomials in any odd degree $d \geq 3$ for which the corresponding $\sL$ has $N=(d-1)/2$ and $M=0$:
   
\begin{thmx}[Existence of maximally wild polynomials]\label{D}
For each $N \geq 1$ there exist real numbers $w_N=0<\cdots<w_1<w_0$ and a real monic polynomial $P$ of degree $d=2N+1$ which has a repelling fixed point at $w_N=0$ and a parabolic fixed point of multiplier $1$ at $w_j$ with \ $\resit(P,w_j)<0$ \ for every $0 \leq j \leq N-1$. For $\ve>0$ sufficiently small, the perturbations $P_\ve:=P+\ve$ are in $\sC(d)$ and the Hausdorff limit $\sL$ of the closed rays $\ov{R_{P_\ve,0}}$ as $\ve \to 0$ consists of $N$ heteroclinics along the real line:
$$
\sL = \sL^* = [0,+\infty] \, = [w_N,w_{N-1}] \cup \cdots \cup [w_1,w_0] \cup [w_0,+\infty]
$$
\end{thmx}

Here $\resit(P,w_j)$ is the {\bit r\'esidu it\'eratif} of $P$ at the parabolic fixed point $w_j$ (see \S \ref{indres} for a brief account). Negativity of this invariant together with real symmetry ensures that $w_j$ bifurcates into a pair of complex conjugate {\it attracting} fixed points for $P_\ve$, which automatically forces the Julia set of $P_\ve$ and therefore $P$ to be connected. If this pair is repelling, the ray $R_{P_\ve}(0)$ may well hit a critical point along the real line, and the Julia set of $P_\ve$ may well be disconnected. \figref{pinch} shows a degree $d=5$ example with $N=2$, before and after perturbation. Examples of this type in degree $d=3$ have appeared in \cite[Appendix B]{GM} and \cite[\S 8]{IK}.

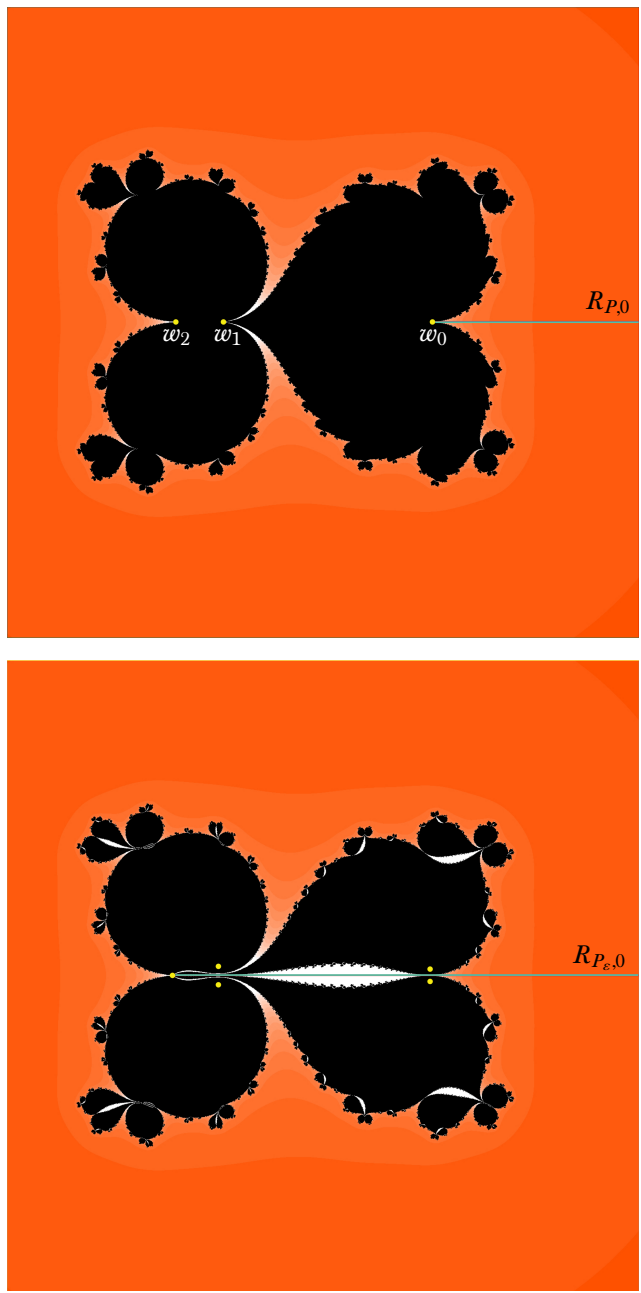
\begin{figure}[]
	\centering
	\begin{overpic}[width=0.55\textwidth]{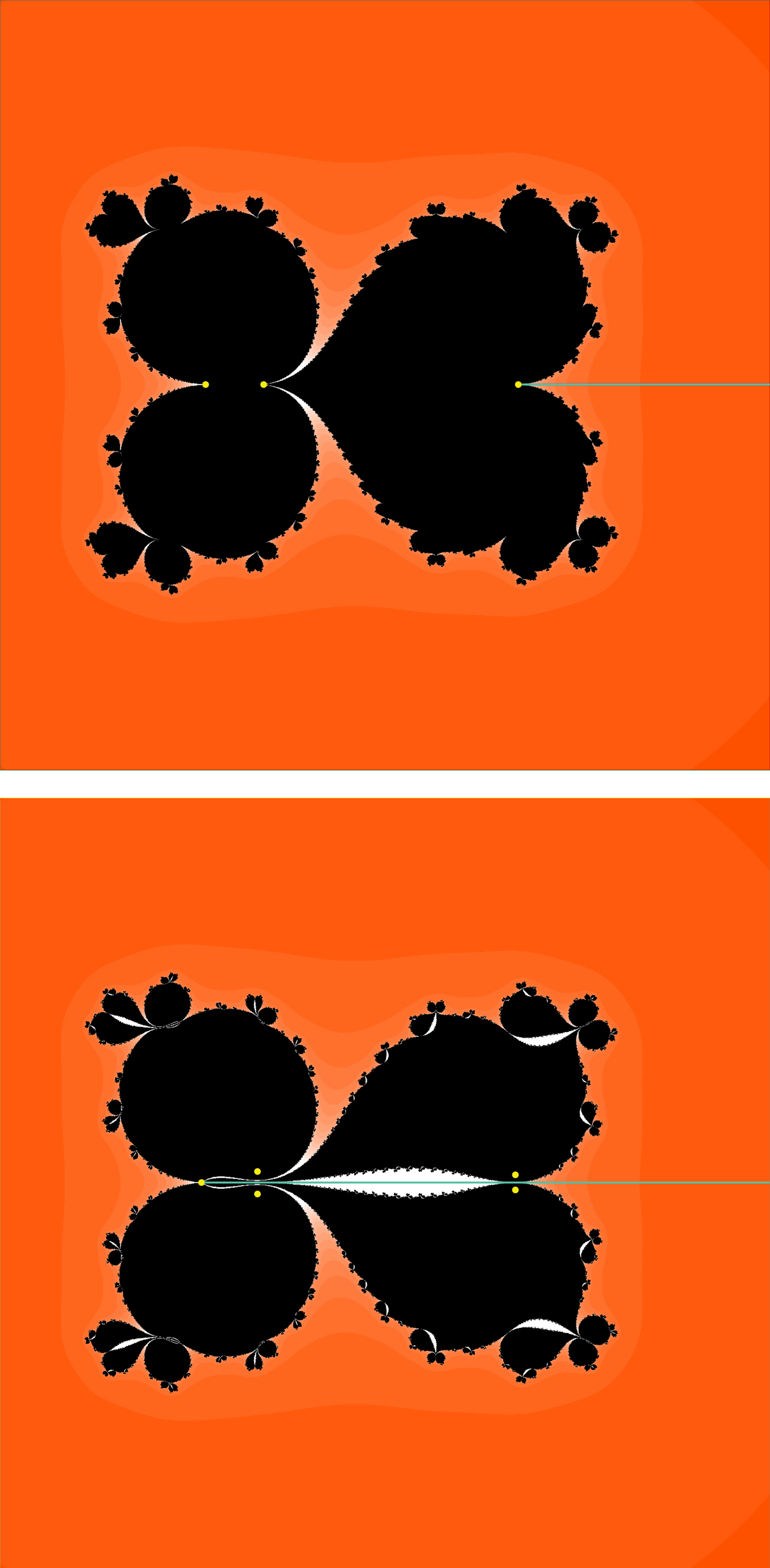}
\put (12,74) {\footnotesize{\color{white} $w_2$}}
\put (16,74) {\footnotesize{\color{white} $w_1$}}
\put (32,74) {\footnotesize{\color{white} $w_0$}}
\put (45,76.5) {\footnotesize{$R_{P,0}$}}
\put (44,25.5) {\footnotesize{$R_{P_\ve,0}$}}
	\end{overpic}
\caption{\footnotesize An example of a real degree $5$ polynomial $P$ (top) and its perturbations $P_\ve=P+\ve$ (bottom) with the properties asserted by \thmref{D}. Under such perturbations the repelling fixed point $w_2$ is stable but the parabolic fixed points $w_0,w_1$ bifurcate into pairs of complex conjugate attracting fixed points. The resulting four immediate attracting basins contain the four critical point of $P_\ve$ and share the repelling fixed point near $w_2$ on their boundary.}
\label{pinch}
\end{figure}

\vs

This work represents an approach to the problem of Hausdorff limits of external rays that entirely avoids parabolic implosions and the gate structure of near-parabolic points. There are related analytic questions, however, for which invoking these tools seems inevitable. This is the subject of a joint project of ours (in progress) that aims at a more quantitative understanding of how external rays behave as they approach a near-parabolic point. For example, in \thmref{D} the multipliers of the attracting fixed points bifurcating off of $w_0, \ldots, w_{N-1}$ tend to $1$ tangentially as $\ve \to 0$. We provide an explanation for this phenomenon by proving that a non-tangential multiplier approach always produces tame convergence of external rays (compare \cite[\S 8]{Mc2} for other manifestations of tameness when parabolic points are perturbed non-tangentially). As a simple formulation, suppose $f(z)=\lambda z + O(z^2) \in \sC(d)$ has a non-degenerate parabolic fixed point at $0$ whose multiplier $\lambda$ is a primitive $q$-th root of unity, and consider a sequence $f_n(z)=\lambda_n z+O(z^2) \in \sC(d)$ of perturbations of $f$ where $\lambda^q_n \to 1$ non-tangentially. If the ray $R_{f,\theta}$ lands at $0$, then $R_{f_n,\theta}$ lands at $0$ or at a nearby period $q$ point according as $|\lambda_n|>1$ or $|\lambda_n|<1$, and in either case $\ov{R_{f_n,\theta}} \to \ov{R_{f,\theta}}$ in the Hausdorff metric. The analytic tools of the theory of near-parabolic germs are perfectly suited for proving such statements, even in more general contexts that go beyond polynomial maps and external rays. \vs

We conclude by noting that the results of this paper essentially cover the more general case where the external angle $\theta$ is preperiodic under $t \mapsto d t \, (\modd \ZZ)$. In fact, it is not hard to see that the Hausdorff limits of $\ov{R_{P_n,\theta}}$ in this case are at worst branched coverings of those described by Theorems \ref{A} and \ref{B}, with finitely many possible branch points within the parabolic basins of $P$. \vs   

\noindent
{\bf Acknowledgments.} We thank H. Inou for his comments, especially on the example illustrated in \figref{double}, and the anonymous referee for providing useful comments and additional references. C.~L.~P.~ would like to thank the Danish Council for Independent Research | Natural Sciences for support via grant DFF-1026-00267B. S.~Z.~ acknowledges the partial support of the Research Foundation of The City University of New York via grant TRADB-54-375.  

\section{Background material}

Throughout the paper we adopt the following notation: \vs 

\begin{enumerate}
\item[$\bullet$]
$\CC$ and $\Chat:=\CC \cup \{ \infty \}$: the complex plane and Riemann sphere \vs  
\item[$\bullet$]
$\DD(p,r):=\{ z \in \CC: |z-p|<r \}$, $\DD:=\DD(0,1)$ \vs
\item[$\bullet$]
$\HH^r:=\{ z \in \CC: \myre(z)>0 \}$: the right half-plane \vs
\item[$\bullet$]
$\ov{X}$ and $\mathring{X}$: the closure and interior of $X$ \vs
\item[$\bullet$]
$\dist$: the Euclidean distance in $\CC$ \vs
\item[$\bullet$]
$\dist_X$: the hyperbolic distance in a domain $X \subset \Chat$ \vs
\item[$\bullet$]
For $a,b>0$, the symbol $a \asymp b$ means $C^{-1} \leq a/b \leq C$ for a constant $C>1$ independent of the choice of $a,b$. \vs    
\end{enumerate}

\subsection{Polynomial maps and external rays}\label{poly}

We assume a working knowledge of basic complex dynamics, as in \cite{DH} or \cite{M}. Let $\sP(d)$ be the space of all monic polynomial maps $\CC \to \CC$ of degree $d$. For $P \in \sP(d)$ we denote by $K=K_P$ and $J=J_P$ the filled Julia set and Julia set of $P$, respectively. The complement $\Om=\Om_P:=\Chat \sm K$ is the basin of attraction of $\infty$. The Green's function of $P$ is the continuous subharmonic function $G=G_P: \CC \to [0,+\infty[$ defined by 
$$
G(z):=\lim_{n \to \infty} \frac{1}{d^n} \log^+ |P^{\circ n}(z)|,
$$
where $\log^+ t = \max \{ \log t, 0 \}$. It satisfies the relation
$$
G(P(z))=d \, G(z) \qquad \text{for all} \ z \in \CC,
$$
with $G(z)=0$ if and only if $z \in K$. We often refer to $G(z)$ as the {\bit potential} of $z$. It is well known that the Green's function also depends continuously on the polynomial; in fact, the function $\sP(d) \times \CC \to [0,+\infty[$ defined by $(P,z) \mapsto G_P(z)$ is continuous (see e.g. \cite[Proposition 8.1]{DH}). \vs

The {\bit B\"{o}ttcher coordinate} of $P$ is the unique conformal isomorphism $\beta=\beta_P$, defined in some neighborhood of $\infty$, which is tangent to the identity at $\infty$ and satisfies 
\begin{equation}\label{bfe}
\beta(P(z))=(\beta(z))^d \qquad \text{for large} \ |z|.  
\end{equation}
The modulus of $\beta$ is related to the Green's function by the relation 
$$
\log |\beta(z)|=G(z) \qquad \text{for large} \ |z|. 
$$
For $\theta \in \RR/\ZZ$, we denote by $R_\theta=R_{P,\theta}$ the maximally extended smooth field line of $\nabla G$ which maps under $\beta$ to the radial line $s \mapsto \e^{s+2\pi \ii \theta}$ for large $s$. We can parametrize $R_\theta$ by the potential $s$: For each $\theta$ there is an $s_\theta = s_{P,\theta} \geq 0$ such that $G(R_\theta(s)) = s$ for all $s > s_\theta$. The field line $R_\theta$ either extends all the way to the Julia set $J$ in which case $s_\theta=0$, or it crashes into a critical point $\omega$ of $G$ at potential $s_\theta>0$ in the sense that $\lim_{s \to s^+_\theta} R_{\theta}(s)=\omega$. The function $\theta \mapsto s_\theta$ is upper semicontinuous \cite{PZ1}, so the set 
$$
V=V_P := \{ \e^{s+2\pi \ii \theta}: s>s_\theta \} \cup \{ \infty \}   
$$
is open. The inverse $\beta^{-1}$ extends to $V$, mapping it conformally to its image $\Om'=\Om'_P \subset \Om$. If the filled Julia set $K$ is connected, then $s_\theta=0$ for all $\theta$, $V=\CC \sm \ov{\DD}$, and $\Om'=\Om$. In this case $\beta: \Om \to \Chat \sm \ov{\DD}$ is a conformal isomorphism. More generally, it is not hard to verify that the unions $\bigcup_{P\in \sP(d)} \, (\{P\} \times \Om'_P)$ and $\bigcup_{P\in \sP(d)} \, (\{P\} \times V_P)$ are open and the global B\"{o}ttcher coordinate $(P,z) \mapsto \beta_P(z)$ is a biholomorphism between them (see for example \cite[\S 1]{BH}). 

\begin{corollary}\label{unif}
Let $P \in \sP(d)$, $\theta \in \RR/\ZZ$, and $s_0>s_{P,\theta}$. If $P_n \to P$ in $\sP(d)$, then $R_{P_n,\theta}(s) \to R_{P,\theta}(s)$ uniformly for $s \in [s_0,+\infty[$. 
\end{corollary}

The following result on the stability of rays landing on repelling points can be found in \cite[Proposition 8.5]{DH} or \cite[Lemma B.1]{GM}:

\begin{theorem}\label{stab}
Suppose $P_0 \in \sP(d)$ and the ray $R_{P_0,\theta}$ lands at a repelling periodic point $\zeta_0$. Then for every neighborhood $V \subset \Chat$ of $\zeta_0$ there is a neighborhood $\sU \subset \sP(d)$ of $P_0$ such that if $P \in \sU$ the ray $R_{P,\theta}$ lands at a repelling periodic point of $P$ in $V$.     
\end{theorem}

\subsection{Local invariants of parabolic points} \label{indres}

The following brief presentation will be useful in \S \ref{ex}. For more details, see \cite{M}, \cite{B}, and \cite{BE}. Let $z_0$ be an isolated fixed point of a holomorphic map $f$, so $f(z)=z_0 + \lambda (z-z_0) + O((z-z_0)^2)$ in some neighborhood of $z_0$. Here $\lambda=f'(z_0)$ is the {\bit multiplier} of the fixed point $z_0$. The {\bit index} $\iota(f,z_0)$ is defined as the residue of the meromorphic $1$-form $dz/(z-f(z))$ at $z_0$:   
$$
\iota(f,z_0) := \res \left(\frac{1}{z-f(z)} \, dz,z_0 \right).   
$$ 
The index is invariant under analytic change of coordinates. Moreover,  
\begin{equation}\label{eq:index}
\iota(f,z_0) = \frac{1}{1-\lambda} \qquad \text{if} \ \lambda \neq 1.
\end{equation}
There is a variant of the notion of index which is well behaved under iteration and (despite its more complicated definition) is often easier to work with. Define the {\bit r\'esidu it\'eratif} of $f$ at $z_0$ by
$$
\resit(f,z_0) := -\frac{1}{2} \res \left( \frac{1+f'(z)}{z-f(z)} \, dz,z_0 \right). 
$$
By the argument principle, 
\begin{align*}
\resit(f,z_0) & = \frac{1}{2} \res \left( \frac{1-f'(z)}{z-f(z)} \, dz,z_0 \right) - \res \left(\frac{1}{z-f(z)} \, dz,z_0 \right) \\
& = \frac{m}{2} - \iota(f,z_0),  
\end{align*} 
where $m \geq 1$ is the fixed point multiplicity of $z_0$, i.e., the multiplicity of $z_0$ as a root of $z-f(z)=0$. Evidently,
\begin{equation}\label{eq:resit}
\resit(f,z_0)=\frac{1}{2} - \frac{1}{1-\lambda} \qquad \text{if} \ \lambda \neq 1.  
\end{equation}
This shows that a multiplicity $1$ fixed point $z_0$ is attracting or repelling according as $\myre(\resit(f,z_0))$ is negative or positive. \vs

When $\lambda=1$ the formulas \eqref{eq:index} and \eqref{eq:resit} for the index and r\'esidu it\'eratif break down, but we can still calculate these invariants with the help of a suitable expansion of $f$. Assuming $z_0$ has multiplier $1$ and fixed point multiplicity $m \geq 2$, there is an analytic change of coordinates in which $f$ assumes the local normal form
$$
f(z)=z+a(z-z_0)^m+b(z-z_0)^{2m-1}+O((z-z_0)^{2m})
$$ 
with $a,b \in \CC$ and $a \neq 0$. An easy computation then shows that 
\begin{equation}\label{resitcomp}
\iota(f,z_0)=\frac{b}{a^2} \qquad \text{so} \qquad \resit(f,z_0)=\frac{m}{2}-\frac{b}{a^2}. 
\end{equation}
For a generic perturbation $f_\ve$ of $f$ the parabolic fixed point $z_0$ splits into $m$ simple fixed points $z_1(\ve), \ldots z_m(\ve)$ of multipliers $\lambda_1(\ve), \ldots, \lambda_m(\ve)$. Continuity of the index then gives 
$\lim_{\ve \to 0} \sum_{j=1}^m 1/(1-\lambda_j(\ve)) = \iota(f,z_0)$, or, in terms of the r\'esidu it\'eratif, 
$$
\lim_{\ve \to 0} \sum_{j=1}^m \resit(f_\ve,z_j(\ve)) = \resit(f,z_0).
$$

When $z_0$ is parabolic with multiplier $\lambda$ a primitive $q$-th root of unity, we can apply the preceding remarks to the iterate $f^{\circ q}$. In this case the multiplicity of $z_0$ as a fixed point of $f^{\circ q}$ is necessarily of the form $m=\nu q+1$ for some integer $\nu \geq 1$ called the {\bit degeneracy order} of $z_0$, the case $\nu=1$ being considered a non-degenerate parabolic. Geometrically, there are $m-1=\nu q$ parabolic basins of $f$ attached to $z_0$, and these fall into $\nu$ disjoint cycles of length $q$. {\it We adopt the convention that the degeneracy order of a repelling fixed point is $\nu=0$.}   

\subsection{The Hausdorff metric}

Let $\sK$ be the space of all non-empty compact subsets of the Riemann sphere $\Chat$. The {\bit Hausdorff metric} on $\sK$ is defined by 
$$
{\bf d}(K,H):= \inf \{ \ve>0: K \subset N_\ve(H) \ \text{and} \ H \subset N_\ve(K) \}, 
$$
where $N_\ve(\cdot)$ denotes the $\ve$-neighborhood in the spherical metric. It is well known that $(\sK,{\bf d})$ is a compact metric space. \vs

Let $\{ K_t \}_{t \in T}$ be a family of non-empty compact sets in $\Chat$ parametrized by a topological space $T$. By the definition of $\bf d$, continuity of the map $t \mapsto K_t$ at $t_0 \in T$ means that for every $\ve>0$ there is a neighborhood $V$ of $t_0$ such that $K_t \subset N_\ve(K_{t_0})$ and $K_{t_0} \subset N_\ve(K_t)$ for all $t \in V$. This can be viewed as a combination of two semi-continuity conditions defined as follows. We say $t \mapsto K_t$ is {\bit upper semicontinuous} at $t_0$ if for every $\ve>0$ there is a neighborhood $V$ of $t_0$ such that  
$$
K_t \subset N_\ve(K_{t_0}) \quad \text{for all} \ t \in V, 
$$
and is {\bit lower semicontinuous} at $t_0$ if for every $\ve>0$ there is a neighborhood $V$ of $t_0$ such that  
$$
K_{t_0} \subset N_\ve(K_t) \quad \text{for all} \ t \in V. 
$$

The following result can be found in \cite{D}:

\begin{theorem}[Douady]\label{semi}
For every $d \geq 2$ the maps $\sP(d) \to \sK$ defined by $P \mapsto K_P$ and $P \mapsto J_P$ are upper semicontinuous and lower semicontinuous, respectively. 
\end{theorem}

\subsection{\cara \ limits of pointed disks}

By a {\bit disk} in the plane is meant a simply connected domain $U \subset \CC$ other than $\CC$ itself. A {\bit pointed disk} $(U,u)$ consists of a disk $U$ and the choice of a base point $u \in U$. By the Riemann mapping theorem there is a unique conformal isomorphism $f:(\DD,0) \oset[-0.2ex]{\cong}{\longrightarrow} (U,u)$ normalized so that $f(0)=u$ and $f'(0)>0$. An easy exercise, based on the Schwarz lemma and the Koebe $1/4$-theorem, shows that 
\begin{equation}\label{1/4}
1 \leq \frac{f'(0)}{\dist(u,\bd U)} \leq 4. 
\end{equation}  

We say that a sequence of pointed disks $(U_n,u_n)$ converges to $(U,u)$ {\bit in the sense of \cara}, and write $(U_n,u_n) \ct (U,u)$, if the sequence of normalized Riemann maps $f_n:(\DD,0) \to (U_n,u_n)$ converges locally uniformly to the normalized Riemann map $f:(\DD,0) \to (U,u)$. Equivalently, if the sequence $f_n^{-1} \circ f$ converges locally uniformly in $\DD$ to the identity map. Notice that in this case every compact subset of $U$ must be contained in $U_n$ for all sufficiently large $n$.\vs   

The \cara \ convergence can be formulated in purely topological terms without any reference to Riemann maps as follows: $(U_n,u_n) \ct (U,u)$ if and only if $u_n \to u$ and for every subsequential Hausdorff limit $K$ of $\Chat \sm U_n$, the disk $U$ is the connected component of $\Chat \sm K$ containing $u$. \vs

The next two lemmas are easy consequences of the definition of \cara \ convergence:

\begin{lemma}\label{ct-easy}
Suppose $(U_n,u_n) \ct (U,u)$. Then $u'_n \to u$ if and only if \ $\dist_{U_n}(u_n,u'_n) \to 0$. 
\end{lemma}

\begin{proof}
Take the normalized Riemann maps $f_n:(\DD,0) \to (U_n,u_n)$ and $f:(\DD,0) \to (U,u)$, so $g_n:=f_n^{-1} \circ f \to \id$ locally uniformly in $\DD$. First suppose $u'_n \to u$. Set $z_n:=f^{-1}(u'_n), w_n:=f_n^{-1}(u'_n)$. Then $z_n \to 0$ and $g_n(z_n)=w_n$, so $w_n \to 0$. It follows that $\dist_{U_n}(u_n,u'_n) = \dist_\DD(0,w_n) \to 0$. Conversely, suppose $\dist_{U_n}(u_n,u'_n) \to 0$. Then, with $z_n, w_n$ defined as above, we have $w_n \to 0$. Since $g_n(z_n)=w_n$ and $g_n \to \id$, we must have $z_n \to 0$, so $u'_n \to u$. 
\end{proof}

\begin{lemma}\label{car-bih}
Suppose $(U_n,u_n) \ct (U,u)$ and $(U'_n,u'_n) \ct (U',u')$. For each $n$ take a conformal isomorphism $\psi_n: (U_n,u_n) \to (U'_n,u'_n)$. Then some subsequence of $\{ \psi_n \}$ converges locally uniformly in $U$ to a conformal isomorphism $\psi: (U,u) \to (U',u')$. 
\end{lemma}

\begin{proof}
Take the normalized Riemann maps $f_n:(\DD,0) \to (U_n,u_n)$, $f:(\DD,0) \to (U,u)$, $g_n:(\DD,0) \to (U'_n,u'_n)$, and $g:(\DD,0) \to (U',u')$, so $f_n \to f$ and $g_n \to g$ locally uniformly in $\DD$. The disk automorphisms $\sigma_n:=g_n^{-1} \circ \psi_n \circ f_n$ fix the origin, so after passing to a subsequence, $\sigma_n \to \sigma \in \Aut(\DD)$. It follows that the corresponding subsequence of $\{ \psi_n \}$ converges to $\psi:=g \circ \sigma \circ f^{-1}: (U,u) \to (U',u')$.     
\end{proof}

The space of all pointed disks has the following form of compactness: 

\begin{lemma}\label{compactness}
Every sequence $(U_n,u_n)$ with $u_n \to u$ and $\dist(u_n,\bd U_n) \asymp 1$ has a subsequence which converges to some $(U,u)$ in the sense of \cara. 
\end{lemma}

\begin{proof}
Consider the normalized Riemann maps $f_n:(\DD,0) \to (U_n,u_n)$. Since $\alpha_n:=f'_n(0) \asymp 1$ by \eqref{1/4}, we may assume after passing to a subsequence that $\alpha_n \to \alpha \in \, ]0,+\infty[$. The sequence $\{ \alpha_n^{-1}(f_n-u_n) \}$ of schlicht functions has a subsequence that converge locally uniformly to a schlicht function $g: (\DD,0) \to (V,0)$. It follows that the corresponding subsequence of $\{ f_n \}$ converges locally uniformly to the normalized Riemann map $\alpha g + u :(\DD,0) \to (U,u)$, where $U=\alpha V +u$. 
\end{proof}

The following useful result describes how changing the base point can affect \cara \ convergence: 

\begin{theorem}\label{two}
Suppose $(U_n,u_n) \ct (U,u)$, and take any sequence $u'_n \in U_n$ with $u'_n \to u'$.%
\vs
\begin{enumerate}
\item[(i)]
If $\dist_{U_n}(u_n,u'_n)$ is bounded, after passing to a subsequence we have $(U_n,u'_n) \ct (U,u')$ (in particular, $u' \in U$). \vs
\item[(ii)]
If $\dist_{U_n}(u_n,u'_n) \to +\infty$ and $\dist(u'_n,\bd U_n) \asymp 1$, after passing to a subsequence we have $(U_n,u'_n) \ct (V,u')$, where $V \cap U =\es$ (in particular, $u' \notin U$).
\end{enumerate} 
\end{theorem} 

\figref{carat} illustrates the two cases.

\begin{figure}[]
	\centering
	\begin{overpic}[width=0.7\textwidth]{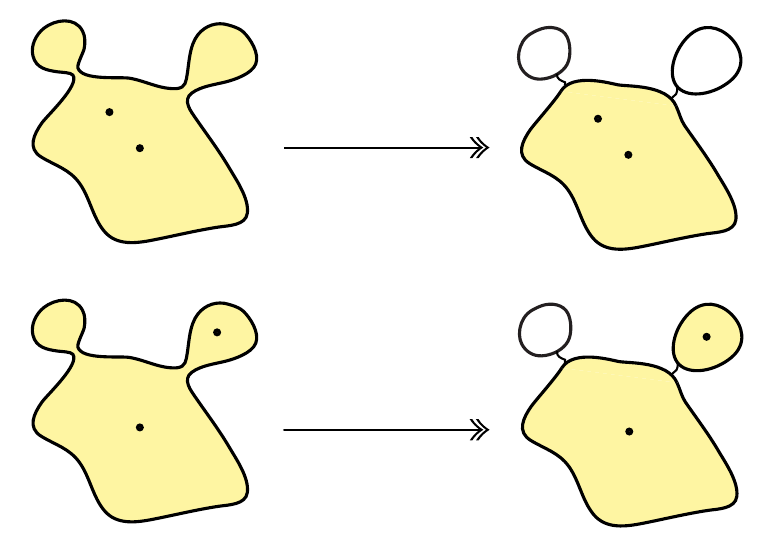}
\put (19,51) {\footnotesize{$u_n$}}
\put (19,15) {\footnotesize{$u_n$}}
\put (10,55) {\footnotesize{$u'_n$}}
\put (28.8,27.4) {\footnotesize{$u'_n$}}
\put (6,45) {\footnotesize{$U_n$}}
\put (6,9) {\footnotesize{$U_n$}}
\put (82,50) {\footnotesize{$u$}}
\put (82,14) {\footnotesize{$u$}}
\put (73.4,55) {\footnotesize{$u'$}}
\put (91.6,26) {\footnotesize{$u'$}}
\put (69,44) {\footnotesize{$U$}}
\put (69,8) {\footnotesize{$U$}}
\put (85,29) {\footnotesize{$V$}}
	\end{overpic}
\caption{\footnotesize Illustration of the two cases of \cara \ convergence in \thmref{two}.}
\label{carat}
\end{figure}
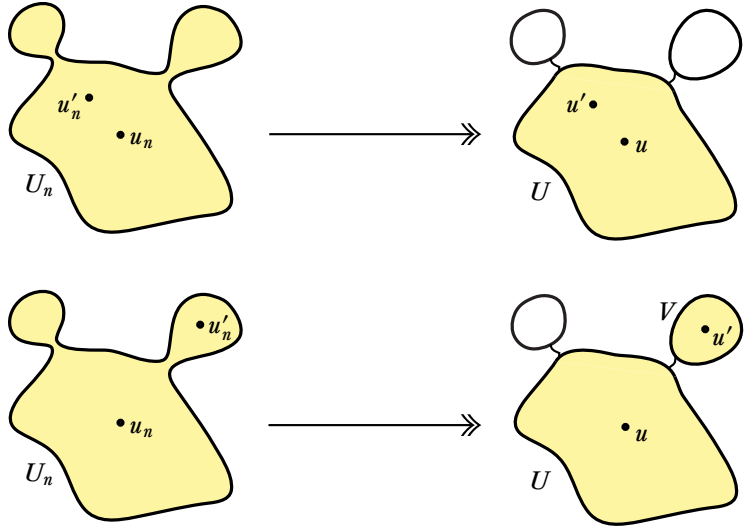

\begin{proof}
As usual, consider the normalized Riemann maps $f_n:(\DD,0) \to (U_n,u_n)$ and $f:(\DD,0) \to (U,u)$. \vs

(i) The normalized Riemann map $g_n:(\DD,0) \to (U_n,u'_n)$ is of the form $f_n \circ \psi_n$, where $\psi_n \in \Aut(\DD)$ sends $0$ to $w_n:=f_n^{-1}(u'_n)$. We have $\sup_n |w_n|<1$ since by the hypothesis $\dist_{\DD}(0,w_n)=\dist_{U_n}(u_n,u'_n)$ is bounded. It follows that the $\psi_n$ lie in a compact subset of $\Aut(\DD)$. So, after passing to a subsequence, $\psi_n$ converges to some $\psi \in \Aut(\DD)$. Thus, $g_n \to f \circ \psi$ and therefore $(U_n,u'_n) \ct (U,u')$, with $u'=f(\psi(0))$. \vs

(ii) By \lemref{compactness}, the assumption $\dist(u'_n,\bd U_n) \asymp 1$ guarantees that after passing to a subsequence $(U_n,u'_n) \ct (V,u')$ for some disk $V$. Suppose $V \cap U \neq \es$ and pick some $\zeta \in V \cap U$. Then $\zeta \in U_n$ for all large $n$ and we have $f_n^{-1}(\zeta) \to f^{-1}(\zeta)$. Hence $\dist_{U_n}(u_n,\zeta)=\dist_\DD(0,f_n^{-1}(\zeta))$ is bounded. Similarly, if $g_n:(\DD,0) \to (U_n,u'_n)$ and $g:(\DD,0) \to (V,u')$ denote the corresponding normalized Riemann maps, then  $g_n^{-1}(\zeta) \to g^{-1}(\zeta)$, so $\dist_{U_n}(u'_n,\zeta) =\dist_\DD(0,g_n^{-1}(\zeta))$ is bounded. The two bounds together imply $\dist_{U_n}(u_n,u'_n)$ being bounded, which is a contradiction.     
\end{proof}

\begin{lemma}\label{samelim}
Suppose $(U_n,u_n) \ct (U,u)$. Let $V_n$ be a proper subdisk of $U_n$ containing $u_n$ and $r_n$ be the radius of the largest hyperbolic ball in $U_n$ centered at $u_n$ that is contained in $V_n$. If $r_n \to +\infty$, then $(V_n,u_n) \ct (U,u)$.  
\end{lemma}

\begin{proof}
Take the normalized Riemann maps $f_n:(\DD,0) \to (U_n,u_n)$ and $g_n:(\DD,0) \to (V_n,u_n)$. By the assumption $f_n$ converges locally uniformly to the normalized Riemann map $f:(\DD,0) \to (U,u)$. The domain $V'_n :=f_n^{-1}(V_n)$ contains the round disk centered at $0$ of radius $(\e^{r_n}-1)/(\e^{r_n}+1) \to 1$. It follows from the Schwarz lemma that $f_n^{-1} \circ g_n: (\DD,0) \to (V'_n,0)$ tends to the identity map, and we conclude that $g_n \to f$ locally uniformly in $\DD$.       
\end{proof}

It is well known that in a given disk $U \subset \CC$ the Euclidean diameter of the hyperbolic ball of fixed radius $r$ tends to $0$ as the center of the ball converges to $\bd U$. The next corollary gives a uniform version of this statement, and is an easy consequence of the following bounds (see \cite[Corollary 1.5]{Po} for a slightly modified form): If $U \subset \CC$ is a disk, $u,u' \in U$, and $r:=\dist_U(u,u')$, then
$$
\frac{1}{4} \tanh\left(\frac{r}{2}\right) \leq \frac{|u-u'|}{\dist(u,\bd U)} \leq 4\exp(2r).
$$

\begin{corollary}\label{shrink}
Fix $r>0$. Take any sequence of pointed disks $(U_n,u_n)$ and let $\de_n$ be the Euclidean diameter of the hyperbolic $r$-ball in $U_n$ centered at $u_n$. Then, $\de_n \to 0$ if and only if $\dist(u_n,\bd U_n) \to 0$ as $n \to \infty$.  
\end{corollary}

\cara \ convergence can of course be defined for pointed disks in the Riemann sphere $\Chat$ in the same manner. For our purposes the main examples of such disks are basins of infinity of polynomials with connected Julia sets:

\begin{example}\label{omconv}
Suppose $P_n \to P$ in $\sC(d)$. Then $\beta^{-1}_{P_n} \to \beta^{-1}_P$ locally uniformly in $\Chat \sm \ov{\DD}$, hence $(\Omega_{P_n},\infty) \ct (\Omega_P,\infty)$. 
\end{example}

All the lemmas in this subsection remain valid for the \cara \ convergence of disks containing $\infty$ except that the Euclidean condition $\dist(u_n,\bd U_n) \asymp 1$ needed for compactness must be modified in terms of the spherical metric. For example, we could require that both the spherical distance between $u_n$ and $\bd U_n$ and the spherical diameter of $\bd U_n$ be bounded away from $0$. However, in our main examples where $U_n$ is the basin of infinity of a polynomial $P_n \in \sC(d)$ and $u_n \in U_n$ converges in $\CC$, this spherical condition is actually equivalent to $\dist(u_n,U_n) \asymp 1$, so we will use the simpler Euclidean formulation without further warning.  
 
\section{$\sL$-arcs and their basic properties}

Recall that we are considering a convergent sequence $P_n \to P$ in $\sC(d)$, $\theta \in \RR/\ZZ$ is a given angle with $d^q \theta = \theta \ (\modd \ZZ)$, $\zeta_n$ and $\zeta$ are the landing points of the external rays $R_n:=R_{P_n,\theta}$ and $R:=R_{P,\theta}$, $\zeta_n \to \zeta_\infty$ and $\ov{R}_n:=R_n \cup \{ \zeta_n, \infty \} \to \sL$ in the Hausdorff metric. \vs

\noindent
{\bf Convention.} In what follows we will use unsubscripted symbols for objects associated with $P$ and symbols with a subscript $n$ for those associated with $P_n$. \vs 

It is clear that $\sL$ is a compact connected subset of $\Chat$ that satisfies $P^{\circ q}(\sL)=\sL$. Moreover, by \corref{unif} for every potential $s_0>0$ we have $R_n(s) \to R(s)$ uniformly on $[s_0,+\infty[$, which shows the ray segment $\{ R(s) : s\geq s_0\}$ is contained in $\sL$. Since this is true for every $s_0$, it follows that $\sL$ contains $R \cup \{ \zeta, \zeta_\infty, \infty \}$. Thus, if $\sL=\ov{R}=R \cup \{ \zeta, \infty \}$, then $\zeta_\infty$ must coincide with $\zeta$ (this is the ``tame case'' in \thmref{A}). Take any $u \in \sL \sm K$ in $\CC$ and any sequence $u_n \in R_n$ such that $u_n \to u$. By the joint continuity of the Green's function (see \S \ref{poly}) we have $G_n(u_n) \to G(u)>0$. By \corref{unif} $u_n=R_n(G_n(u_n)) \to R(G(u))$, which shows $u=R(G(u))$. It follows that $\sL \sm K = R \cup \{ \infty \}$. \vs 

\subsection{Proof of the Basic Structure Lemma}\label{larc}

First suppose $u \in \sL \cap J$. Take any sequence $u_n \in R_n$ such that $u_n \to u$. Then $s_n:=G_n(u_n) \to 0$. By the lower semicontinuity of $P \mapsto J_P$ (\thmref{semi}), $\dist(u_n,J_n) \to 0$. Since 
$$
\dist_{\Om_n}(u_n,P_n^{\circ q}(u_n)) \leq \dist_{\CC \sm \ov{\DD}} (\e^{s_n+2\pi \ii \theta}, \e^{d^q s_n+2\pi \ii \theta}) = q \log d,  
$$    
it follows from \corref{shrink} that $|u_n - P_n^{\circ q}(u_n)| \to 0$, implying $P^{\circ q}(u)=u$. This proves part (i) of the Basic Structure Lemma. \vs

Now suppose $u \in \sL \cap \mathring{K}$. It suffices to prove the claims in part (ii) of the Lemma under the additional hypothesis $P^{\circ q}(u) \neq u$. Once this is accomplished, this additional hypothesis can be removed easily. In fact, by connectivity of $\sL$ we can always find some $u' \in \sL \cap \mathring{K}$ arbitrarily close to $u$ for which $P^{\circ q}(u') \neq u'$ and conclude that $u'$ belongs to a parabolic basin. This proves that $u$ must be in a parabolic basin and therefore $P^{\circ q}(u) \neq u$. \vs

So let us assume $u \in \sL \cap \mathring{K}$ and $P^{\circ q}(u) \neq u$. By an argument similar to above, no subsequence of $\dist(u_n,J_n)$ can tend to $0$. In other words, $\dist(u_n,J_n) \asymp 1$. On the other hand, $(\Om_n,\infty) \ct (\Om,\infty)$ (\exref{omconv}), and 
$$
\dist_{\Om_n}(u_n,\infty)= \log \left( \frac{\e^{s_n}+1}{\e^{s_n}-1} \right) \to +\infty
$$
since $s_n \to G(u)=0$. Hence, by \thmref{two}(ii), after passing to a subsequence we may assume $(\Om_n,u_n) \ct (V,u)$ for some disk $V \subset K$. It is not hard to see that $P^{\circ q}|_V : V \to V$ is a proper map (see for example \cite[Theorem 5.6]{Mc1}). Since $P_n \in \sC(d)$, the basin of infinity $\Om_n$ does not contain any critical point of $P_n$, so $V$ does not contain any critical point of $P$. Thus, $P^{\circ q}|_V: V \to V$ is a conformal isomorphism. \vs  

Let ``$\log$'' denote the branch of logarithm which maps the slit-plane $\CC \sm ]-\infty,0]$ conformally onto the strip $\{ z \in \CC: |\myim(z)|<\pi \}$. The map
$$
\psi_n(z):= \frac{1}{s_n} \log \big( \e^{-2\pi \ii \theta} \beta_n(z) \big)
$$ 
carries the slit-basin $V_n:=\Om_n \sm R_n(\theta+1/2)$ conformally onto the horizontal half-strip 
$$
S_n:= \left\{ z \in \CC: \myre(z)>0, |\myim(z)|<\frac{\pi}{s_n} \right\}, 
$$
with $\psi_n(u_n)=1$. The subdomain $V'_n \subset V_n$ defined by  
$$
V'_n := \beta_n^{-1} \left\{ r\e^{2\pi \ii t} : r>1, |t-\theta|<\frac{1}{2d^q} \right\}   
$$
maps conformally under $\psi_n$ onto the half-strip 
$$
S'_n:= \left\{ z: \myre(z)>0, |\myim(z)|<\frac{\pi}{d^q s_n} \right\}. 
$$
Moreover, if $z \in V'_n$, 
\begin{align*}
\psi_n(P_n^{\circ q}(z)) & = \frac{1}{s_n} \log \big( \e^{-2\pi \ii \theta} \beta_n(P_n^{\circ q}(z)) \big) = \frac{1}{s_n} \log \big( \e^{-2\pi \ii \theta} (\beta_n(z))^{d^q} \big) \\
& = \frac{1}{s_n} \log \big( (\e^{-2\pi \ii \theta} \beta_n(z))^{d^q} \big) \qquad \qquad (\text{since} \ d^q \theta = \theta \ (\modd \ZZ)) \\
& = \frac{d^q}{s_n} \log \big( \e^{-2\pi \ii \theta} \beta_n(z) \big) = d^q \psi_n(z). 
\end{align*}
This gives the commutative diagram 
\begin{equation}\label{cd1}
\begin{tikzcd}[column sep=small]
V'_n \arrow[d,swap,"P_n^{\circ q}"] \arrow[rr,"\psi_n"] & & S'_n \arrow[d,"\cdot d^q"] \\
V_n \arrow[rr,"\psi_n"] & & S_n 
\end{tikzcd} 
\end{equation}

\begin{lemma}\label{conj}
There is a subsequence of $\{ \psi_n \}$ which converges locally uniformly to a conformal isomorphism $\psi: V \to \HH^r$ normalized by $\psi(u)=1$, and the following diagram commutes: 
\begin{equation}\label{cd2}
\begin{tikzcd}[column sep=small]
V \arrow[d,swap,"P^{\circ q}"] \arrow[rr,"\psi"] & & \HH^r \arrow[d,"\cdot d^q"] \\
V \arrow[rr,"\psi"] & & \HH^r 
\end{tikzcd} 
\end{equation}
\end{lemma}

\begin{proof}
It is easy to see that $(S_n,1) \ct (\HH^r,1)$ and $(S'_n,1) \ct (\HH^r,1)$. By \lemref{samelim}, $(V_n, u_n) \ct (V,u)$ and $(V'_n,u_n) \ct (V,u)$. It follows from \lemref{car-bih} that some subsequence of $\{ \psi_n \}$ converges locally uniformly in $V$ to a conformal isomorphism $\psi: (V,u) \to (\HH^r,1)$. Taking the limit of \eqref{cd1} then shows that $\psi$ satisfies \eqref{cd2}.      
\end{proof} 

Define $\gamma: \, ]0,+\infty[ \to V$ by $\gamma(t):=\psi^{-1}(t)$. Evidently $\gamma$ satisfies $P^{\circ q}(\gamma(t))=\gamma(d^q t)$. For simplicity we use $\gamma$ both for the map and the arc $\gamma( ]0,+\infty[ ) \subset \CC$.  

\begin{lemma}\label{inter}
$\gamma = \sL \cap V$.
\end{lemma}

\begin{proof}
Let $\psi_n: V_n \to S_n$ and $\psi: V \to \HH^r$ be as in \lemref{conj}. First suppose $z \in \sL \cap V$. Take a compact neighborhood $E$ of $z$ such that $E \subset V$ and therefore $E \subset V_n$ for all large $n$. Take a sequence $z_n \in R_n$ such that $z_n \to z$, so $z_n \in E$ for all large $n$. The uniform convergence $\psi_n \to \psi$ on $E$ implies $\psi_n(z_n) \to \psi(z)$. Since $\psi_n(z_n) \in \RR$ for all $n$, we conclude that $\psi(z) \in \RR$, so $z \in \gamma$. \vs 

For the reverse inclusion, take $z=\gamma(t)$ for a given $t>0$. Since $z \in V$ we have $z \in V_n$ for all large $n$ and $\zeta_n:=\psi_n(z) \to \psi(z)=t$. It follows that $t_n:=\myre(\zeta_n) \to t$. The point $z_n:=\psi^{-1}_n(t_n)$ is in $V_n \cap R_n$ and  
$$
\dist_{V_n}(z_n,z)=\dist_{S_n}(t_n,\zeta_n) \to 0 
$$ 
It follows that $z_n \to z$ in the Euclidean metric. Since $z_n \in R_n$, we conclude that $z \in \sL$. 
\end{proof}

A standard hyperbolic geometry argument (see e.g. \cite[Lemma 5.5]{M}) shows that both limits $w^+ :=\lim_{t \to +\infty} \gamma(t)$ and $w^- :=\lim_{t \to 0} \gamma(t)$ exist and are fixed under $P^{\circ q}$. By \lemref{inter}, $w^\pm \in \sL$. By the Snail Lemma (\cite[Lemma 16.2]{M}) $w^+$ is either attracting or parabolic with multiplier $1$. The first case cannot happen: If $w^+$ were attracting, a full neighborhood of $w^+$ would be contained in an attracting basin of $P_n$ for all large $n$ \cite[Lemma 6.3]{D}. Clearly this is impossible since $w^+ \in \sL$ must be accumulated by the rays $R_n$. Thus, $w^+$ is parabolic with multiplier $1$. This proves that $V$ and therefore $\gamma \big( ]0,+\infty[ \big)$ is contained in a parabolic basin $B$ of $w^+$ and $P^{\circ q}: B \to B$ is a proper map of some degree $k$. If $\phi: B \to \DD$ is any conformal isomorphism, the induced map $f:=\phi \circ P^{\circ q} \circ \phi^{-1}: \DD \to \DD$ is a Blaschke product of degree $k$. If $w^+ \neq w^-$, then $\varphi(w^+)$ and $\varphi(w^-)$ are distinct fixed points of $f$, with the former necessarily parabolic with multiplier $1$ and multiplicity $3$ since by symmetry it has two attracting basins (namely $\DD$ and $\Chat \sm \ov{\DD}$). It follows that $f$ has at least $4$ fixed points on $\bd \DD$ counting multiplicities. Since the total number of fixed points of $f$ is $k+1$, we obtain $k+1\geq 4$, or $k \geq 3$. It follows that $B$ contains $k-1 \geq 2$ critical points of $P^{\circ q}$. This completes the proof of the Basic Structure Lemma.

\begin{remark}\label{2sides}
The statement that $P^{\circ q}: B \to B$ has at least two critical points when $w^+ \neq w^-$ will be sharpened in \S \ref{modo}, where we show that each of the two components of $B \sm \gamma$ contains at least one critical point. By contrast, when $w^+=w^-$ the Jordan domain $\Delta$ bounded by $\ov{\gamma}$ is a component of $B \sm \gamma$ free from  critical points and the restriction $P^{\circ q}: \Delta \to \Delta$ is a conformal isomorphism (see \lemref{5eq}).   
\end{remark} 

\begin{remark}\label{topol}
Here is a byproduct of the proof of \lemref{conj}. For each $n$ the conformal map $\psi_n$ sends $R_n(s)$, the point on the ray $R_n$ at potential $s$, to the point $s/s_n$ on the real line. The convergence of $\psi_n : V_n \to S_n$ to $\psi: V \to \HH^r$ shows that the sequence of inverse maps $\psi_n^{-1}: S_n \to V_n$ converges to $\psi^{-1}: \HH^r \to V$ uniformly on compact subsets of $\HH^r$. It follows that for each $t>0$, $R_n(ts_n)=\psi_n^{-1}(t) \to \psi^{-1}(t)=\gamma(t)$, and this convergence is uniform on compact subsets of $]0,+\infty[$. Thus, {\it every compact subarc of $\gamma$ can be approximated in $C^\infty$-topology by a suitable sequence of compact subarcs of the rays $R_n$.} 
\end{remark}   

\subsection{Heteroclinic and homoclinic arcs in $\sL$.}\label{hhintro}

We have shown that $\sL \cap \mathring{K}$ is a disjoint union of $P^{\circ q}$-invariant real-analytic open arcs contained in parabolic basins. Each such arc comes equipped with a natural parametrization $\gamma: \, ]0,+\infty[ \to \sL \cap \mathring{K}$ which satisfies $\gamma(d^q t)=P^{\circ q}(\gamma(t))$. For simplicity each such $\gamma$ will be called an {\bit $\sL$-arc}. The initial point $w^-=w^-(\gamma):= \lim_{t \to 0} \gamma(t)$ and the end point $w^+=w^+(\gamma):= \lim_{t \to +\infty} \gamma(t)$ are fixed under $P^{\circ q}$, with $w^+$ always parabolic of multiplier $1$ under $P^{\circ q}$. Note that $\gamma$ has a well-defined tangent direction at $w^+$, namely the attracting direction of $w^+$ corresponding to the basin that contains $\gamma$. More precisely, $(w^+-\gamma(t))/|w^+-\gamma(t)| \to v$ as $t \to +\infty$, where $v$ is the unit vector in the given attracting direction. Similarly, if $w^-$ is parabolic, then $\gamma$ has a well-defined tangent direction at $w^-$, namely a repelling direction of $w^-$. More precisely, $(\gamma(t)-w^-)/|\gamma(t)-w^-| \to v$ as $t \to 0$, where $v$ is the unit vector in the given repelling direction. We note however that {\it typically} $\gamma$ does not have a $C^1$ extension at its extremities, i.e., $\gamma'(t)/|\gamma'(t)|$ fails to have a limit as $t \to 0$ or $+\infty$ (see the Appendix and compare \figref{gooddisk} where the schematic picture of a homoclinic arc is drawn to suggest this behavior.) \vs  

It follows from the above remarks that if $w^-=w^+=w$, the Jordan curve $\ov{\gamma}=\gamma \cup \{ w \}$ has the well-defined angle $\pi p/(\nu q)$ at $w$, where $p$ is the period of $w$ and $\nu \geq 1$ is the degeneracy order of $w$ as a fixed point of $P^{\circ p}$ (see \S \ref{indres}). In particular, this angle is $\pi$ if $w$ is non-degenerate of period $q$. \vs 

An $\sL$-arc $\gamma$ is called {\bit heteroclinic} if $w^-(\gamma) \neq w^+(\gamma)$ and {\bit homoclinic} if $w^-(\gamma)=w^+(\gamma)$. Any pair of heteroclinic arcs $\gamma, \eta$ with common initial and end points must be contained in the same parabolic basin since by the maximum principle the topological disk bounded by $\ov{\gamma} \cup \ov{\eta}$ is contained in $\mathring{K}$. On the other hand, two homoclinic arcs that join the same $w$ to itself can be contained in different parabolic basins of $w$. \vs

The Basic Structure Lemma shows that every $\sL$-arc $\gamma$ is contained in a topological ``strip'' $V_\gamma$ in which the action of $P^{\circ q}$ is conformally conjugate to $z \mapsto d^q z$ in $\HH^r$, or equivalently, to the translation $z \mapsto z+q \log d$ in the Euclidean strip $\{ z: |\myim(z)| < \pi/2 \}$. We remark that if $\gamma, \eta$ are distinct $\sL$-arcs, then $V_\gamma \cap V_\eta = \es$. To see this, take sequences $u_n \in R_n$ converging to $u \in \gamma$ and $v_n \in R_n$ converging to $v \in \eta$. Then, after passing to subsequences, $(\Om_n,u_n) \ct (V_\gamma, u)$ and $(\Om_n,v_n) \to (V_\eta, v)$. If $\dist_{\Om_n}(u_n,v_n)$ has a bounded subsequence, then $V_\gamma = V_\eta$ by \thmref{two}, which is impossible by \lemref{inter}. Thus $\dist_{\Om_n}(u_n,v_n) \to +\infty$ and another application of \thmref{two} shows that $V_\gamma \cap V_\eta = \es$. If follows from this disjointness property of strips that {\it there are at most countably many $\sL$-arcs}. \vs       

\noindent
{\bf Convention.} It will be convenient to also regard the external ray $R$ as an $\sL$-arc with $w^-(R)=\zeta$ and $w^+(R)=\infty$. We consider $R$ as neither a heteroclinic nor a homoclinic arc. Observe that the existence of special parametrizations of $\sL$-arcs guaranteed by the Basic Structure Lemma is trivially true for $R$ as well, namely there is a parametrization $\gamma: \, ]0,+\infty[ \to R$ which sends $t=1$ to any designated point at Green's potential $s>0$ and satisfies $P^{\circ q}(\gamma(t))=\gamma(d^q t)$. To see this, simply take $\gamma(t):=\beta^{-1}(\e^{st+2\pi \ii \theta})$, where $\beta$ is the B\"{o}ttcher coordinate of $P$.    

\subsection{$\sL$-arcs in a given parabolic basin}\label{modo}

Let $B$ be a parabolic basin that is invariant under $P^{\circ q}$. Take a Fatou coordinate $\Phi: B \to \CC$ which satisfies $\Phi \circ P^{\circ q} = T \circ \Phi$, where $T:z \mapsto z+1$ is the unit translation. We normalize $\Phi$ so that it maps a maximal attracting petal $W \subset B$ biholomorphically onto the right half-plane $\HH^r$, sending some critical point $c \in \bd W \cap B$ of $P^{\circ q}$ to $\Phi(c)=0$. One can check that $\Phi$ maps the closure $\ov{W}$ homeomorphically onto $\ov{\HH^r}=\{ z: \myre(z) \geq 0 \}$. The quotient $W/P^{\circ q}$ is conformally isomorphic to the cylinder $\HH^r/T = \CC/T$. The critical points of $\Phi$ are the critical points of $P^{\circ q}$ and their preimages in $B$, so the critical values of $\Phi$ form finitely many backward orbits of $T$ in $\CC$. It is not hard to show that $\Phi$ is an infinite-degree ramified covering from $B$ onto $\CC$ and as such it has no finite asymptotic value. It follows from the monodromy theorem that any simply connected domain in $\CC$ which avoids the critical values of $\Phi$ can be lifted univalently under $\Phi$. 

\begin{lemma}\label{propan}
Let $\gamma$ be an $\sL$-arc in a parabolic basin $B=P^{\circ q}(B)$. \vs
\begin{enumerate}
\item[(i)]
The Fatou coordinate $\Phi:B \to \CC$ normalized as above maps $V_\gamma$ biholomorphically onto a $T$-invariant topological strip $\tilde{V}_\gamma$ which avoids the critical value $\Phi(c)=0$. The image $\tilde{\gamma}:=\Phi(\gamma)$ is a $T$-invariant arc in $\tilde{V}_\gamma$. \vs  

\item[(ii)]
The annulus $A_{\gamma}:= \tilde{V}_\gamma/T$ is essentially embedded in the cylinder $\CC/T$ and has the projection $\tilde{\gamma}/T$ as its core geodesic. Moreover, $\operatorname{mod}(A_{\gamma})=\pi/(q \log d)$. \vs

\item[(iii)]
If $\eta$ is another $\sL$-arc in $B$ distinct from $\gamma$, then $\tilde{V}_\gamma \cap \tilde{V}_\eta= \es$, hence $A_\gamma \cap A_\eta = \es$.
\end{enumerate}  
\end{lemma}

\begin{proof}
For (i), consider the maximal attracting petal $W \subset B$ having $c$ on its boundary. Then $\Phi$ maps $V_\gamma \cap W$ biholomorphically onto an open set in $\HH^r$ which is forward invariant under $T$. The relation $\Phi=T^{-n} \circ \Phi \circ P^{\circ nq}$ together with the fact that $P^{\circ q}: V_\gamma \to V_\gamma$ is a conformal isomorphism shows that $\Phi$ is a biholomorphism between $V_\gamma$ and a topological strip $\tilde{V}_\gamma$ in $\CC$ that is fully invariant under $T$. Moreover, $\tilde{V}_\gamma$ avoids the critical value $\Phi(c)=0$ since $V_\gamma$ avoids $c$ and $\Phi: \ov{W} \to \ov{\HH^r}$ is a homeomorphism (see \figref{quo}). \vs 

Part (ii) is an easy exercise since a combination of (i) and the Basic Structure Lemma shows that $A_\gamma \cong V_\gamma/P^{\circ q}$ is isomorphic to the quotient of $\HH^r$ by the action of the automorphism $z \mapsto d^q z$. Statement (iii) follows from the disjointness $V_\gamma \cap V_\eta = \es$ proved at the end of \S \ref{hhintro} since $\tilde{V}_\gamma \cap \tilde{V}_\eta \cap \HH^r = \Phi(V_\gamma \cap V_\eta \cap W) = \es$, so $\tilde{V}_\gamma \cap \tilde{V}_\eta = \es$ by $T$-invariance.    
\end{proof}

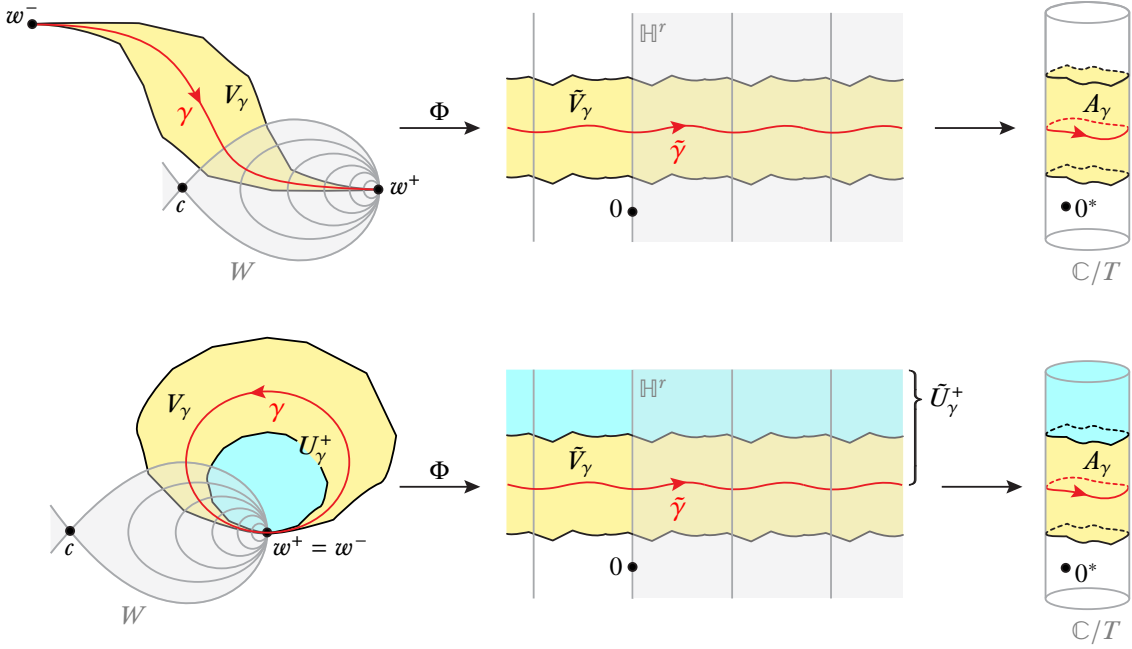
\begin{figure}[]
	\centering
	\begin{overpic}[width=\textwidth]{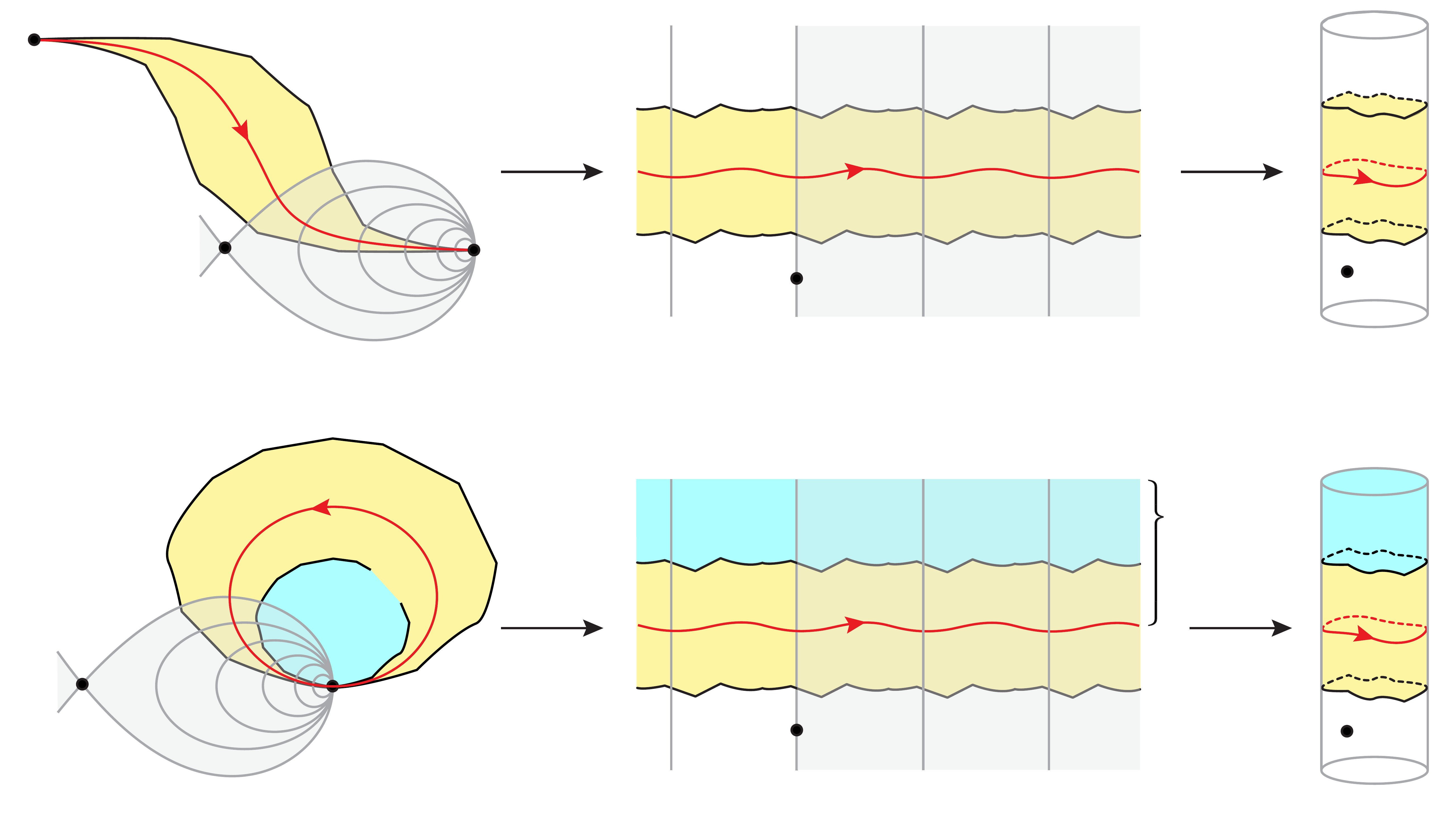}
\put (15,46.5) {\small{\color{red}{$\gamma$}}}
\put (23,20.7) {\small{\color{red}{$\gamma$}}}
\put (19,48) {\footnotesize{$V_\gamma$}}
\put (14,21) {\footnotesize{$V_\gamma$}}
\put (58,43.2) {\small{\color{red}{$\tilde{\gamma}$}}}
\put (58,12) {\small{\color{red}{$\tilde{\gamma}$}}}
\put (33.5,40) {\footnotesize{$w^+$}}
\put (0,55) {\footnotesize{$w^-$}}
\put (23.3,8.5) {\footnotesize{$w^+=w^-$}}
\put (52.7,38) {\footnotesize{$0$}}
\put (52.7,6.8) {\footnotesize{$0$}}
\put (93.3,38) {\footnotesize{$0^*$}}
\put (93.3,6.2) {\footnotesize{$0^*$}}
\put (37,46.5) {\footnotesize{$\Phi$}}
\put (37,15.2) {\footnotesize{$\Phi$}}
\put (14.8,38.5) {\footnotesize{$c$}}
\put (5.1,8.5) {\footnotesize{$c$}}
\put (19.5,32.5) {\footnotesize{\color{gray}$W$}}
\put (10,2.5) {\footnotesize{\color{gray}$W$}}
\put (55.3,53.3) {\footnotesize{\color{gray}$\HH^r$}}
\put (55.3,22.3) {\footnotesize{\color{gray}$\HH^r$}}
\put (94,47.2) {\footnotesize{$A_\gamma$}}
\put (94,16) {\footnotesize{$A_\gamma$}}
\put (49,47.2) {\footnotesize{$\tilde{V}_\gamma$}}
\put (49,16) {\footnotesize{$\tilde{V}_\gamma$}}
\put (80.6,21.5) {\footnotesize{$\tilde{U}^+_\gamma$}}
\put (25.5,17.3) {\footnotesize{$U^+_\gamma$}}
\put (93,32.5) {\footnotesize{\color{gray}$\CC/T$}}
\put (93,1.2) {\footnotesize{\color{gray}$\CC/T$}}
	\end{overpic}
\caption{\footnotesize Top: A heteroclinic arc $\gamma$ with its invariant strip $V_\gamma$ and their images $\tilde{\gamma}$ and $\tilde{V}_\gamma$ under the normalized Fatou coordinate $\Phi$. Bottom: A positively oriented homoclinic arc $\gamma$ with its invariant disk $U^+_\gamma=\Delta_\gamma$ and their images $\tilde{\gamma}$ and $\tilde{U}^+_\gamma$ under $\Phi$. Here $W$ is the maximal attracting petal whose $\Phi$-image is the right half-plane $\HH^r$.}
\label{quo}
\end{figure}

Take an $\sL$-arc $\gamma$ in $B$ and its $T$-invariant image $\tilde{\gamma}=\Phi(\gamma)$ as above. Denote the upper and lower components of $\CC \sm \tilde{\gamma}$ by $\tilde{U}^+_\gamma$ and $\tilde{U}^-_\gamma$, respectively. Let $U^\pm_\gamma$ be the unique component of $\Phi^{-1}(\tilde{U}^\pm_\gamma)$ having $\gamma$ on its boundary. Notice that each of the two components of $B \sm \gamma$ contains one of $U^\pm_\gamma$. Every point of $\bd U^\pm_\gamma$ either belongs to the basin boundary $\bd B$ at which $\Phi$ is undefined, or it belongs to an iterated $P^{\circ q}$-preimage of $\gamma$ in $B$ which maps under $\Phi$ to a point of $\tilde{\gamma}$. It is easy to check that $U^\pm_\gamma$ is simply connected, the map $P^{\circ q}: U^\pm_\gamma \to U^\pm_\gamma$ is proper, and the following diagram commutes: 
\begin{equation}\label{cd3}
\begin{tikzcd}[column sep=small]
U^\pm_\gamma \arrow[d,swap,"P^{\circ q}"] \arrow[rr,"\Phi"] & & \tilde{U}^\pm_\gamma \arrow[d,"T"] \\
U^\pm_\gamma \arrow[rr,"\Phi"] & & \tilde{U}^\pm_\gamma 
\end{tikzcd} 
\end{equation}

Recall that each $\sL$-arc has a natural dynamical orientation. We call a homoclinic arc $\gamma$ positively or negatively oriented according as the dynamical orientation of the Jordan curve $\ov{\gamma}$ is counterclockwise or clockwise. The Jordan domain bounded by $\ov{\gamma}$ will be denoted by $\Delta_\gamma$. 

\begin{lemma}[Characterization of homoclinic arcs]\label{5eq}
The following conditions on an $\sL$-arc $\gamma \subset B$ are equivalent: \vs
\begin{enumerate}
\item[(i)]
$\gamma$ is a positively oriented homoclinic arc. \vs
\item[(ii)]
$U^+_\gamma$ is one of the two components of $B \sm \gamma$. \vs 
\item[(iii)]
$P^{\circ q}: U^+_\gamma \to U^+_\gamma$ is a conformal isomorphism. \vs
\item[(iv)]
$\Phi: U^+_\gamma \to \tilde{U}^+_\gamma$ is a conformal isomorphism. \vs
\item[(v)]
$U^+_\gamma$ contains no critical point of $P^{\circ q}$. \vs
\end{enumerate}  
Under these conditions $U^+_\gamma=\Delta_\gamma$. A similar statement is true if we change $\gamma$ in (i) to negatively oriented and $U^+_\gamma, \tilde{U}^+_\gamma$ everywhere to $U^-_\gamma, \tilde{U}^-_\gamma$.
\end{lemma}

\begin{proof} 
For simplicity we will drop the subscript $\gamma$ from our notation. \vs

(i) $\Longrightarrow$ (ii): The Jordan domain $\Delta$ bounded by $\ov{\gamma}$ is a component of $B \sm \gamma$. We claim that $U^+=\Delta$. In fact, $U^+ \subset \Delta$ since $\gamma$ is positively oriented. If this inclusion were strict, we would have $\bd U^+ \cap \Delta \neq \es$ and any point in this intersection would eventually map to $\gamma$ under the iterations of $P^{\circ q}$. This is impossible since $P^{\circ q}(\ov{\gamma})=\ov{\gamma}$ together with the maximum principle implies $P^{\circ q}(\Delta)=\Delta$. \vs

(ii) $\Longrightarrow$ (iii): The restriction $P^{\circ q}: U^+ \to U^+$ is proper of some degree $k \geq 1$. The assumption on $U^+$ implies $P^{-q}(\gamma) \cap \bd U^+ = \gamma$. Since $P^{\circ q}$ acts homeomorphically on $\gamma$, every point of $\gamma$ has a unique $P^{\circ q}$-preimage on $\bd U^+$, so $k=1$. \vs

(iii) $\Longrightarrow$ (v): Trivial. \vs

(v) $\Longrightarrow$ (iv): By the hypothesis the restriction $\Phi: U^+ \to \tilde{U}^+$ is a ramified covering without critical points, so it is a regular covering map. As $\tilde{U}^+$ is simply connected, this covering map must be a conformal isomorphism. \vs

(iv) $\Longrightarrow$ (i): By \eqref{cd3}, $P^{\circ q}: U^+ \to U^+$ is a conformal isomorphism. Since $P^{\circ q}$ acts homeomorphically on $\gamma$, it follows that $P^{-q}(\gamma) \cap \bd U^+ = \gamma$. This, in turn, implies $\Phi^{-1}(\tilde{\gamma}) \cap \bd U^+ = \bigcup_{n \geq 0} P^{-nq}(\gamma) \cap \bd U^+ = \gamma$. Now for each $z_0 \in U^+$ the arc $\tilde{\gamma}$ has full harmonic measure in $\bd \tilde{U}^+$ as seen from $\Phi(z_0) \in \tilde{U}^+$. Since $\Phi: U^+ \to \tilde{U}^+$ is a conformal isomorphism and $\Phi^{-1}(\tilde{\gamma}) \cap \bd U^+ = \gamma$, it follows that $\gamma$ has full harmonic measure in $\bd U^+$ as seen from $z_0$. By elementary conformal mapping theory, this implies $\gamma$ being homoclinic. In fact, if $\gamma$ were heteroclinic, its end points on $\bd U^+$ would be distinct so we could find distinct accessible points $\alpha, \beta \in \bd U^+ \sm \gamma$. Under any conformal isomorphism $(U^+,z_0) \to (\DD,0)$ these points would correspond to distinct points $\alpha',\beta' \in \bd \DD$ and the image of $\gamma$ would be contained in one of the two components of $\bd \DD \sm \{ \alpha', \beta' \}$, forcing the harmonic measure of $\gamma$ to be $<1$. 
\end{proof}

Given two homoclinic arcs $\gamma, \eta$ based at the same parabolic point, we say that $\gamma$ is {\bit inside} $\eta$, or $\eta$ is {\bit outside} $\gamma$, if $\Delta_\gamma \subset \Delta_\eta$. Any maximal linearly ordered set of homoclinics with respect to this order will be called an {\bit earring}. Notice that all homoclinic arcs in the same earring must have the same (positive or negative) dynamical orientation. 

\begin{theorem}\label{hetfin}
Suppose $B=P^{\circ q}(B)$ is a parabolic basin. \vs 
\begin{enumerate}
\item[(i)]
$B$ contains at most finitely many heteroclinic $\sL$-arcs. \vs
\item[(ii)]
$B$ contains at most two earrings of homoclinic $\sL$-arcs, and each earring has an outermost element.
\end{enumerate}   
\end{theorem}

Later we will sharpen this result by replacing ``at most finitely many'' in (i) and ``at most two'' in (ii) with ``at most one'' (see Corollaries \ref{hetuni} and \ref{homuni}).

\begin{proof}
(i) Suppose there are infinitely many heteroclinics $\gamma_0, \gamma_1, \gamma_2, \ldots$ in $B$. After relabeling we may assume $\tilde{U}^+_{\gamma_{j+1}} \subset \tilde{U}^+_{\gamma_j}$ for all $j \geq 0$ (the case where $\tilde{U}^+_{\gamma_{j+1}} \supset \tilde{U}^+_{\gamma_j}$ for all $j$ is similar). By \lemref{propan} the annuli $A_{\gamma_j}=\tilde{V}_{\gamma_j} /T$ are mutually disjoint and essentially embedded in $\CC/T$, all having the same modulus $\pi/(q \log d)$. It follows from the Gr\"{o}tzsch inequality (\cite[Corollary B.6]{M}) that the annulus $X_j$ bounded by the core geodesics of $A_{\gamma_1}$ and $A_{\gamma_j}$ has modulus $\geq (j-1)\pi/(q \log d)$, so $\lim_{j \to \infty} \operatorname{mod}(X_j)= +\infty$. On the other hand, \lemref{5eq} shows that both topological half-planes $\tilde{U}^\pm_{\gamma_j}$ contain critical values of $\Phi$. Since the critical values of $\Phi$ lie in finitely many backward $T$-orbits, there should be distinct critical values $a,b$ such that $a \in \tilde{U}^+_{\gamma_j}$ and $b \in \tilde{U}^-_{\gamma_j}$ {\it for every} $j \geq 0$. As $\tilde{V}_{\gamma_j} \subset \tilde{U}^+_{\gamma_{j-1}} \cap \tilde{U}^-_{\gamma_{j+1}}$, the points $a$ and $b$ belong to different components of $\CC \sm \tilde{V}_{\gamma_j}$ for all $j \geq 1$. It follows that $A_{\gamma_j}$ and therefore $X_j$ separates the images of $a$ and $b$ in the quotient cylinder $\CC/T$ for all $j \geq 1$. This is a contradiction since there is a bound on the moduli of essentially embedded annuli in $\CC/T$ that separate two given points.  \vs   

(ii) First we show that every earring in $B$ has an outermost homoclinic. Suppose to the contrary that there is an infinite sequence $\{ \gamma_j \}$ of distinct homoclinics in $B$ such that $\gamma_j$ is inside $\gamma_{j+1}$ for all $j$. Without loss of generality take every $\gamma_j$ to be positively oriented. By \lemref{5eq}, $U^+_{\gamma_j}=\Delta_{\gamma_j}$, hence $U^+_{\gamma_j} \subset U^+_{\gamma_{j+1}}$ for all $j$. It follows that $\tilde{U}^+_{\gamma_j} \subset \tilde{U}^+_{\gamma_{j+1}}$ and in particular $\tilde{V}_{\gamma_j} \subset \tilde{U}^+_{\gamma_{j+1}} \cap \tilde{U}^-_{\gamma_{j-1}}$ for all $j$. Moreover, every $\tilde{U}^+_{\gamma_j}$ avoids the critical value $\Phi(c)=0$ since $U^+_{\gamma_j}$ avoids $c$ and $\Phi: \ov{W} \to \ov{\HH^r}$ is a homeomorphism. Fixing some point $a \in \tilde{U}^+_{\gamma_1} \sm \tilde{V}_{\gamma_1}$, we see that $0$ and $a$ belong to different components of $\CC \sm \tilde{V}_{\gamma_j}$ and therefore every $A_{\gamma_j}$ separates the images of $0$ and $a$ in the quotient $\CC/T$. This leads to a contradiction by applying Gr\"{o}tzsch inequality as in (i). \vs

We have shown that each earring of homoclinics in $B$ has an outermost element $\gamma$. The image $\Phi(\Delta_\gamma)$ is one of the topological half-planes $\tilde{U}^\pm_\gamma$ depending on the orientation of $\gamma$, so the quotient $\Phi(\Delta_\gamma)/T$ is a punctured neighborhood of one end of the cylinder $\CC/T$. If $\eta$ is the outermost element of another earring in $B$, then $\Phi(\Delta_\eta)$ is disjoint from $\Phi(\Delta_\gamma)$ since 
$$
\Phi(\Delta_\gamma) \cap \Phi(\Delta_\eta) \cap \HH^r = \Phi (\Delta_\gamma \cap \Delta_\eta \cap W) = \es. 
$$
It follows that $\Phi(\Delta_\eta)/T$ is another punctured neighborhood of an end of $\CC/T$ disjoint from $\Phi(\Delta_\gamma)/T$. As this cylinder has only two ends, we conclude that there are at most two earrings of homoclinics in $B$.  
\end{proof} 

One trivial consequence of the above proof: {\it There are at most finitely many homoclinic arcs between a given pair of homoclinics in an earring}. In fact, if $\gamma, \xi, \eta$ are distinct homoclinics with $\Delta_\gamma \subset \Delta_\xi \subset \Delta_\eta$, then the annulus $A_\xi$ is essentially embedded in the annulus bounded by the core geodesics of $A_{\gamma}$ and $A_{\eta}$, and there can be at most finitely many such annuli since they are pairwise disjoint and have the same modulus.

\begin{remark}
We believe it is possible to prove using the theory of parabolic implosions (specifically, the existence and properties of Lavaurs maps) that every earring consists of either one or infinitely many homoclinic arcs. We shall not attempt to present the necessary setup and technical details of a possible argument, which would be incompatible with the approach and intended scope of this paper.       
\end{remark}        

\subsection{Arcwise-connectivity in $\sL$}\label{wise}

For any subset $E \subset \sL$ let $E^*$ denote the set of points in $E$ that do not lie on any homoclinic arc: 
$$
E^* := E \sm (\text{union of all homoclinic arcs}).
$$

\begin{theorem}\label{pc}
For every connected set $E \subset \sL$, both $E$ and $E^*$ are arcwise-connected. The arc in $E^*$ joining a given pair of points in $E \cap J$ is unique up to homotopy in $K$ rel $E \cap J$.   
\end{theorem}

We will eventually see that $\sL^*$ is homeomorphic to the interval $[0,+\infty]$ so the arc joining any pair in $E^*$ is in fact unique (see \S \ref{oshet}).    

\begin{proof}
By \thmref{hetfin} $\sL$ contains at most finitely many heteroclinics and finitely many earrings of homoclinics. The outermost homoclinic $\gamma$ in any earring has the property that $\Delta_\gamma \cup V_\gamma$ contains all homoclinics in that earring and $\sL \sm (\Delta_\gamma \cup V_\gamma)$ is compact and connected. Removing all the finitely many such $\Delta_\gamma \cup V_\gamma$ from $\sL$, we conclude that $\sL^*$ is a compact connected subset of $\Chat$. It follows that $\sL^*$ is a finite connected graph which has the points of $\sL \cap J$ and $\infty$ as its vertices and the heteroclinic arcs and $R$ as its edges. \vs

Suppose now that $E \subset \sL$ is connected. For any homoclinic $\gamma$ that meets $E$, either $E \subset \gamma$ (in which case $E$ is trivially arcwise-connected and $E^*=\es$), or $E \cap \ov{\gamma}$ is an arc (open, closed, half-open) containing $w^-(\gamma)=w^+(\gamma)$. It is now easy to see that $E^*$ is connected, and that arcwise-connectivity of $E$ is equivalent to that of $E^*$ (provided that $E^* \neq \es$). The latter is trivial since $E^*$ is a connected subset of the finite graph $\sL^*$.%
\vs 

Finally, suppose $\xi \neq \xi'$ are two arcs in $E^*$ that join a given pair in $E \cap J$. Evidently each of these arcs is the closure of a finite union of heteroclinics. Let $U$ be a bounded component of $\CC \sm (\xi \cup \xi')$. Then $U$ is a Jordan domain with $\bd U \subset K$. By the maximum principle $U \subset \mathring{K}$, from which it follows that $U$ is contained in a parabolic basin which also contains all heteroclinics in $\bd U$. This implies that $U$ is bounded by exactly two heteroclinics with the same initial and end points $E \cap J$, which are clearly homotopic in $K$ rel $E \cap J$. Repeating this for the finitely many bounded components of $\CC \sm (\xi \cup \xi')$, we conclude that $\xi,\xi'$ must be homotopic in $K$ rel $E \cap J$.  
\end{proof}

\section{Proof of \thmref{A}}

\subsection{The intrinsic potential order}\label{ssord}

Every $u \in \sL$ is the limit of a sequence $u_n = R_n(s_n)$. In this case the full sequence of potentials $s_n=G_n(u_n)$ must have a well-defined limit. In fact, $s_n \to s>0$ if and only if $u=R(s) \in \sL \sm K$, and $s_n \to 0$ if and only if $u \in \sL \cap K$. \vs

Our goal is to show that the union of $\sL$-arcs inherits a natural {\it linear} order from the potentials of all possible sequences of approximating points on the rays $R_n$. We begin with the following   

\begin{lemma}\label{unipot}
Suppose $u_n=R_n(s_n) \to u \in \sL \sm (J \cup \{ \infty \})$. Take a sequence $\{ s'_n \}$ of potentials and set $u'_n=R_n(s'_n)$. Then, the following conditions on a sequence $n_i \to \infty$ are equivalent: \vs
\begin{enumerate}
\item[(i)]
$u'_{n_i} \to u$ as $i \to \infty$. \vs 

\item[(ii)]
$s_{n_i}/s'_{n_i} \to 1$ as $i \to \infty$. \vs
\end{enumerate}
In particular, $u'_n \to u$ if and only if $s_n/s'_n \to 1$. 
\end{lemma}

\begin{proof}
The result follows from the joint continuity of the Green's function if $u \in R$, so let us assume $u \in \sL \cap \mathring{K}$. We will make use of the computation 
$$
\dist_{\Om_n}(u_n,u'_n)=\dist_{\Chat \sm \ov{\DD}}(\e^{s_n+2\pi \ii \theta}, \e^{s'_n+2\pi \ii \theta}) = \left| \log \left( \frac{s_n}{s'_n} \right) \right| 
$$
which shows that the conditions $s_n/s'_n \to 1$ and $\dist_{\Om_n}(u_n,u'_n) \to 0$ along any given subsequence are equivalent. \vs 

First suppose $u'_{n_i} \to u$ but $\dist_{\Om_{n_i}}(u_{n_i},u'_{n_i}) \not \to 0$. Take a subsequence of $\{ n_i \}$ along which $\dist_{\Om_{n_i}}(u_{n_i},u'_{n_i})$ remains bounded away from $0$. As in the proof of the Basic Structure Lemma, there is a sub-subsequence of $\{ n_i \}$ along which $(\Om_{n_i},u_{n_i}) \ct (V,u)$. By \lemref{ct-easy}, $\dist_{\Om_{n_i}}(u_{n_i},u'_{n_i}) \to 0$ along this sub-subsequence of $\{ n_i \}$, which is a contradiction. \vs

Now suppose $\dist_{\Om_{n_i}}(u_{n_i},u'_{n_i}) \to 0$ but $u'_{n_i} \not \to u$. Take a subsequence of $\{ n_i \}$ along which $u'_{n_i}$ remains bounded away from $u$. Take a sub-subsequence of $\{ n_i \}$ along which $(\Om_{n_i},u_{n_i}) \ct (V,u)$. By \lemref{ct-easy}, $u'_{n_i} \to u$ along this sub-subsequence of $\{ n_i \}$, which is a contradiction.  
\end{proof}

Let us call a sequence $\{ s_n \}$ of potentials {\bit admissible} if $\lim_{n \to \infty} R_n(s_n)$ exists and belongs to $\sL \sm (J \cup \{ \infty \})$. Two admissible sequences $\{ s_n \}, \{ s'_n \}$ are {\bit equivalent} if $s_n/s'_n \to 1$. The equivalence class of $\{ s_n \}$ is denoted by $\langle s_n \rangle$. We denote by $\sS$ the space of all equivalence classes of admissible sequences. By \lemref{unipot} there is a well-defined bijection $\Pi: \sS \to \sL \sm (J \cup \{ \infty \})$ given by $\Pi(\langle s_n \rangle):=\lim_{n \to \infty} R_n(s_n)$. We topologize $\sS$ so that $\Pi$ is continuous, in which case $\sS$ is homeomorphic to a disjoint union of at most countably many open intervals. \vs

For each $u=\Pi(\langle s_n \rangle)$ take the $\sL$-arc $\gamma$ through $u$ and the special parametrization $\gamma: \, ]0,+\infty[ \to \sL$ which satisfies $\gamma(1)=u$ and $\gamma(d^q t) = P^{\circ q}( \gamma(t))$, given by the Basic Structure Lemma if $\gamma \neq R$ or the B\"{o}ttcher coordinate if $\gamma=R$. Then $\gamma(t)= \Pi(\langle ts_n \rangle)$ for every $t>0$ (see \remref{topol}). Thus, we can think of $\Pi^{-1}(\gamma)$ as the interval $\{ \langle ts_n \rangle : t>0 \}$ in $\sS$ homeomorphic to $]0,+\infty[$. 

\begin{lemma}\label{limex}
For any pair $\langle s_n \rangle, \langle s'_n \rangle$ in $\sS$ the limit $c:=\lim_{n \to \infty} s_n/s'_n \in [0,+\infty]$ exists. More precisely, if $u=\Pi(\langle s_n \rangle)$ and $u'=\Pi(\langle s'_n \rangle)$ belong to the same $\sL$-arc, then $0<c<+\infty$, while if $u,u'$ belong to different $\sL$-arcs, then $c=0$ or $+\infty$.   
\end{lemma}

\begin{proof}
We have already seen that if $u,u'$ belong to the same $\sL$-arc, then $u=\Pi(\langle cs'_n \rangle)$ for some $0<c<+\infty$. This gives $\langle cs'_n \rangle =  \langle s_n \rangle$, which shows $\lim s_n/s'_n =c$, as required. We claim that if $u,u'$ do not belong to the same $\sL$-arc, then one of the relations $\lim s_n/s'_n=0$ or $\lim s_n/s'_n=+\infty$ must hold. If not, take a subsequence $s_{n_i}/s'_{n_i}$ which tends to some limit $c \in \, ]0,+\infty[$ as $i \to \infty$. By \lemref{unipot}, $R_{n_i}(cs'_{n_i}) \to u$. On the other hand, we know that the whole sequence $\{ R_n(cs'_n) \}$ tends to a point on the same $\sL$-arc as $u'$. This contradicts our assumption that $u,u'$ are not on the same $\sL$-arc.
\end{proof}

\lemref{limex} shows that we can define a linear order on $\sS$ by declaring 
$$
\langle s_n \rangle < \langle s'_n \rangle \qquad \text{if and only if} \qquad \lim_{n \to \infty} \frac{s_n}{s'_n} <1.
$$       
Pushing forward this order by the bijection $\Pi$, we obtain a corresponding linear order on $\sL \sm (J \cup \{ \infty \})$ called the {\bit intrinsic potential order}, that is, we define $u<u'$ if and only if $\Pi^{-1}(u)<\Pi^{-1}(u')$. Observe that the restriction of this order to each $\sL$-arc is compatible with the dynamical orientation: if $u,u' \in \gamma$ with $u<u'$, then in going from $w^-(\gamma)$ to $w^+(\gamma)$ we visit $u$ before $u'$. \vs

We write $u \leq u'$ in the usual sense that $u<u'$ or $u=u'$. \vs
 
Evidently if $\gamma, \eta$ are distinct $\sL$-arcs such that $u<u'$ for some $u \in \gamma, u'\in \eta$, then $u<u'$ for every $u \in \gamma, u'\in \eta$. In this case we write $\gamma<\eta$. This defines a linear order on the collection of $\sL$-arcs. Explicitly, $\gamma<\eta$ if and only if for every $\Pi(\langle s_n \rangle) \in \gamma$ and $\Pi(\langle s'_n \rangle) \in \eta$ we have $\lim_{n \to \infty} s_n/s'_n=0$. Note that in this order the ray $R$ is the largest $\sL$-arc.

\begin{remark}
There is a completion $\hat{\sS}$ homeomorphic to $[0,+\infty]$ such that $\Pi$ extends to a continuous surjection $\hat{\Pi}: \hat{\sS} \to \sL$. Moreover, for each $w \in \sL \cap J$ the cardinality of $\hat{\Pi}^{-1}(w)$ is one more than the number of homoclinics based at $w$ (possibly infinite). The proof is based on \thmref{B} and will not be given.     
\end{remark}

The following lemma will be used frequently in the next sections:

\begin{lemma}\label{EE}
Consider positive sequences $\{ s_n \}, \{ s'_n \}$ such that $s_n<s'_n$ for all $n$. Let $E \subset \sL$ be any subsequential Hausdorff limit of the ray segments $R_n([s_n,s'_n]):=\{ R_n(s):  s_n \leq s \leq s'_n \}$, and let $z \in E \sm (J \cup \{ \infty \})$. \vs 
\begin{enumerate}
\item[(i)]
If $\langle s_n \rangle \in \sS$ and $u=\Pi(\langle s_n \rangle)$, then $u \leq z$. \vs

\item[(ii)]
If $\langle s'_n \rangle \in \sS$ and $u'=\Pi(\langle s'_n \rangle)$, then $z \leq u'$. 
\end{enumerate} 
\end{lemma} 

\begin{proof}
We prove (i), the proof of (ii) being similar. Take an increasing sequence $\{ n_i \}$ of integers such that $R_{n_i}([s_{n_i},s'_{n_i}]) \to E$ in the Hausdorff metric and choose $r_i \in [s_{n_i}, s'_{n_i}]$ such that $R_{n_i}(r_i) \to z$. Let $\langle t_n \rangle := \Pi^{-1}(z)$, so $R_n(t_n) \to z$. By \lemref{unipot}, $r_i/t_{n_i} \to 1$. It follows from   
$$
\frac{s_{n_i}}{t_{n_i}} = \frac{s_{n_i}}{r_i} \cdot \frac{r_i}{t_{n_i}}
$$
that $\limsup_{i \to \infty} s_{n_i}/t_{n_i} \leq 1$ and therefore $\lim_{n \to \infty} s_n/t_n \leq 1$ by \lemref{limex}. This implies $u \leq z$, as required. 
\end{proof}

\subsection{The order and structure of heteroclinic arcs}\label{oshet} 

\begin{lemma}\label{third}
Let $\gamma$ be a heteroclinic arc and $\eta$ be any $\sL$-arc. \vs
\begin{enumerate}
\item[(i)]
Suppose $w^+(\gamma)=w^+(\eta)=w$. If $\eta<\gamma$, there is a heteroclinic arc $\xi$ with $w^+(\xi)=w$ such that $\eta<\xi<\gamma$. \vs
\item[(ii)]
Suppose $w^-(\gamma)=w^-(\eta)=w$. If $\gamma<\eta$, there is a heteroclinic arc $\xi$ with $w^-(\xi)=w$ such that $\gamma<\xi<\eta$.
\end{enumerate}   
\end{lemma}

Since by \thmref{hetfin} there are only finitely many heteroclinic arcs in $\sL$, we immediately obtain the following 

\begin{corollary}\label{hetuni}
\mbox{}
\begin{enumerate}
\item[(i)]
For every $w \in \sL \cap J$ there is at most one heteroclinic arc $\gamma$ with $w=w^-(\gamma)$ or with $w=w^+(\gamma)$. In particular, every parabolic basin $B=P^{\circ q}(B)$ contains at most one heteroclinic arc. \vs
\item[(ii)]
Let $\gamma$ be a heteroclinic and $\eta$ be a homoclinic arc. If $w^+(\gamma)=w^+(\eta)$, then $\gamma<\eta$. If $w^-(\gamma)=w^-(\eta)$, then $\eta<\gamma$. 
\end{enumerate}    
\end{corollary}  

\begin{proof}[Proof of \lemref{third}]
We only prove (i), as the proof of (ii) is similar. Fix $u = \Pi(\langle s_n \rangle) \in \eta$ and $u' = \Pi(\langle s'_n \rangle) \in \gamma$, so $\lim_{n \to \infty} s_n/s'_n=0$. Let $E$ be a subsequential Hausdorff limit of the ray segments $R_n([s_n,s'_n])$. Then $E$ is a compact connected subset of $\sL$ containing $u,u'$. In fact, it is easy to see that $E$ contains the segment of $\eta$ between $u$ and $w$ and the segment of $\gamma$ between $w^-(\gamma)$ and $u'$. Observe that by \lemref{EE}, every $z \in E \sm J$ other than $u,u'$ satisfies $u<z<u'$. \vs

By \thmref{pc} we can find an arc $\xi$ in $E^*$ joining $w$ and $w^-(\gamma)$. By the above observation, this arc cannot meet the open segment of $\eta$ between $w^-(\eta)$ and $u$ or the open segment of $\gamma$ between $u'$ and $w$. Hence it must be altogether disjoint from $\gamma, \eta$. Another application of \thmref{pc} then shows that $\xi$ is homotopic to $\gamma$ rel $E \cap J$ and therefore must be a heteroclinic arc in the same basin as $\gamma$. Evidently $w^+(\xi)=w^+(\gamma)=w$ and $\eta<\xi<\gamma$.    
\end{proof}

The next lemma shows that the dynamical orientation of adjacent heteroclinic arcs is compatible with their intrinsic potential order. 

\begin{lemma}\label{ord}
If $\gamma, \eta$ are heteroclinic arcs with $w^+(\gamma)=w^-(\eta)=w$, then $\gamma < \eta$.
\end{lemma}   

\begin{proof}
(i) First note that $w^+(\eta) \neq w^-(\gamma)$; otherwise the union $\gamma \cup \eta \cup \{ w^{\pm}(\gamma) \}$ would bound a topological disk in $\mathring{K}$, which would imply $\gamma, \eta$ are contained in the same parabolic basin $B$. This leads to a contradiction, either by invoking \corref{hetuni}(i) or by simply observing that all orbits of $P^{\circ q}$ in $B$ must converge to a unique boundary point. \vs

Assume by way of contradiction that $\eta<\gamma$. As in the proof of \lemref{third} fix $u = \Pi(\langle s_n \rangle) \in \eta$ and $u' = \Pi(\langle s'_n \rangle) \in \gamma$, consider a subsequential Hausdorff limit $E \subset \sL$ of the ray segments $R_n([s_n,s'_n])$, and find an arc $\xi$ in $E^\ast$ joining $w^+(\eta)$ and $w^-(\gamma)$ which must be disjoint from $\gamma, \eta$. By \thmref{pc}, $\xi$ is homotopic to the arc $\gamma$ followed by $\eta$ rel $E \cap J$. In particular, $w \in \xi$. It follows that the segment of $\xi$ between $w^+(\eta)$ and $w$ is a heteroclinic in the same basin as $\eta$ and the segment of $\xi$ between $w$ and $w^-(\gamma)$ is a heteroclinic in the same basin as $\gamma$. This contradicts \corref{hetuni}(i).
\end{proof}

\begin{lemma}\label{orddd}
Suppose $\gamma$ is a heteroclinic arc and $u=\Pi(\langle s_n \rangle) \in \gamma$. Take any sequence of potentials $\{ t_n \}$ with $R_n(t_n) \to w \in \sL \cap J$. \vs
\begin{enumerate}
\item[(i)]
If $w=w^+(\gamma)$, then $s_n/t_n \to 0$. \vs
\item[(ii)]
If $w=w^-(\gamma)$, then $s_n/t_n \to +\infty$.
\end{enumerate}
\end{lemma}

\begin{proof}
We prove (i), the proof of (ii) being similar. Assume by way of contradiction that $\limsup_{n \to \infty} s_n/t_n >0$. If there were a subsequence $s_{n_i}/t_{n_i} \to c \in ]0,+\infty[$, then $\lim_{i \to \infty} R_{n_i}(t_{n_i}) = \lim_{n \to \infty} R_n(c^{-1}s_n) \in \gamma$ by \lemref{unipot}, which would lead to the conclusion $w \in \gamma$. Thus, we must have $\lim_{n \to \infty} s_n/t_n = +\infty$. Let $E$ be a subsequential Hausdorff limit of the segments $R_n([t_n, s_n])$. Then $E$ is a compact connected subset of $\sL$ containing $w,u$. Moreover, by \lemref{EE}, if $z \in E \sm J$ and $z \neq u$, then $z<u$. In particular, $E$ is disjoint from the open subarc $\xi \subset \gamma$ between $u$ and $w$. Now \thmref{pc} shows that there is an arc in $E$ joining $w$ and $u$. But by \corref{hetuni}(i) the only arc in $\sL$ joining $w$ and $u$ is $\xi$. The contradiction proves $\lim_{n \to \infty} s_n/t_n=0$. 
\end{proof}

We now have all the ingredients we need to prove \thmref{A}: 

\begin{proof}[Proof of \thmref{A}]
Most of the claims of the theorem have already been verified.  \thmref{pc} showed that the spine $\sL^*$ is a finite connected graph embedded in $\Chat$ which has $(\sL \cap J) \cup \{ \infty \}$ as its vertices and all non-homoclinic $\sL$-arcs as its edges. By \corref{hetuni}(i), every vertex of this graph has degree $1$ or $2$. Since $\infty$ is surely a vertex of degree $1$, it follows that this graph is a tree with exactly two vertices of degree $1$ and the remaining vertices (if any) of degree $2$. In particular, this tree is homeomorphic to a closed arc. If there are no vertices of degree $2$, then $\sL^* = \ov{R}$ and therefore $\sL \cap J = \{ \zeta \} = \{ \zeta_{\infty} \}$, and depending on whether $\sL$ contains any homoclinics or not, we are in the semi-wild or tame case, respectively. On the other hand, if $\sL^*$ does have vertices of degree $2$, \lemref{ord} shows that we can sort the heteroclinic arcs in $\sL$ as $\gamma_N < \cdots < \gamma_1 <R$ so that $w^+(\gamma_j)=w_{j-1}$ and $w^-(\gamma_j)=w_j$ for all $1 \leq j \leq N$, and $w^-(R)=w_0$ (compare \figref{lstar}). \vs

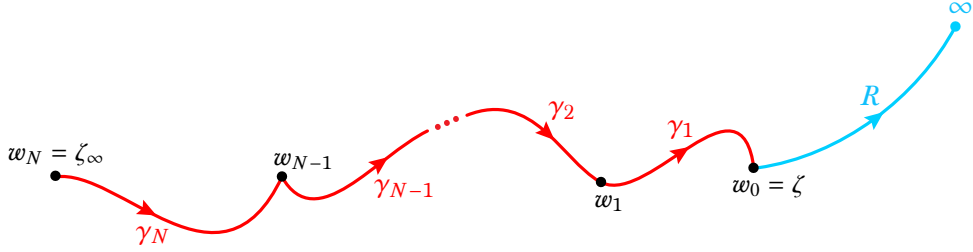
\begin{figure}[t]
	\centering
\begin{overpic}[width=0.8\textwidth]{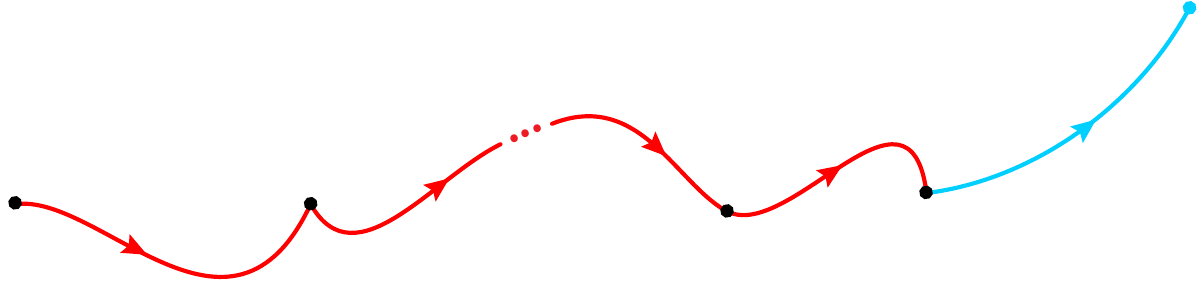}
\put (89,16) {\small{\color{myblue}$R$}}
\put (98.6,25.8) {\small{\color{myblue}$\infty$}}
\put (75,6.5) {\footnotesize{$w_0=\zeta$}}
\put (60,4.9) {\footnotesize{$w_1$}}
\put (25,9.5) {\footnotesize{$w_{N-1}$}}
\put (-4,10) {\footnotesize{$w_N=\zeta_{\infty}$}}
\put (68,13) {\small{\color{red}$\gamma_1$}}
\put (55,15) {\small{\color{red}$\gamma_2$}}
\put (36,6.5) {\small{\color{red}$\gamma_{N-1}$}}
\put (10,1.4) {\small{\color{red}$\gamma_N$}}
\end{overpic}
\caption{\footnotesize The spine $\sL^*$ is the Hausdorff limit $\sL = \lim_{n \to \infty} \ov{R_n}$ with all possible homoclinic arcs removed. It is a piecewise analytic arc homeomorphic to $[0,+\infty]$.}
\label{lstar}
\end{figure}

To finish the proof, it remains to show that the limit $\zeta_\infty=\lim_{n \to \infty} \zeta_n$ of the landing points of the rays $R_n$ is $w_N$. Suppose $\zeta_\infty=w_j$ for some $0 \leq j \leq N-1$. Fix some $u= \Pi(\langle s_n \rangle) \in \gamma_{j+1}$. On the one hand, since $\zeta_n \to \zeta_{\infty}$, we can find a sequence $\{ t_n \}$ of potentials converging to $0$ so fast that $0<t_n<s_n$ and $R_n(t_n) \to \zeta_{\infty}$. On the other hand, \lemref{orddd}(i) implies $s_n/t_n \to 0$. This is a contradiction.      
\end{proof}

\section{Proof of \thmref{B}}

So far we have shown that $\sL$ is the union of the spine $\sL^*$ together with a finite collection (possibly empty) of earrings attached to the points of $\sL \cap J$. For the proof of \thmref{B}, we need to show that these earrings cannot occur in the basins that meet the spine, that the same basin cannot contain two distinct earrings, and that every earring with more than one homoclinic arc must be based at the point $w_N=\zeta_\infty$. These statements require a better understanding of the structure of homoclinic arcs and will be addressed in \S \ref{oshom}.      

\subsection{Good disks}\label{gd}

Suppose $\gamma$ is an $\sL$-arc with $w^+(\gamma)=w$. By real analyticity, there are at most countably many radii $\ve>0$ for which the boundary of the disk $D:=\DD(w,\ve)$ meets $\gamma$ tangentially. In other words, for all but countably many choices of $\ve>0$ the circle $\bd D$ meets $\gamma$ transversally at finitely many points. Of course a similar description holds when $w^-(\gamma)=w$. 
It follows that for any finite collection $\sC$ of $\sL$-arcs with either the initial or end point at $w$, there are arbitrarily small $\ve>0$ for which $\bd D$ meets every $\gamma \in \sC$ transversally at finitely many points. We call such $D$ a {\bit good disk} centered at $w$ for the collection $\sC$. \vs

Suppose $D=\DD(w,\ve)$ is a good disk for a finite collection $\sC$, where $w \in \sL \cap J$ is parabolic. Let $\gamma \in \sC$ be a heteroclinic with $w^-(\gamma)=w$, so $\gamma$ is asymptotic to a repelling direction at $w$. If $a^-$ is the point where $\bd D$ meets the radial line at $w$ in this repelling direction, then every $z \in \gamma \cap \bd D$ satisfies $|z-a^-|=o(\ve)$ as $\ve \to 0$ (see \S \ref{hhintro}). The greatest point on $\gamma \cap \bd D$ (in the intrinsic potential order) is denoted by $w^-(\gamma,D)$. Thus, $w^-(\gamma,D)$ is characterized as the point on $\gamma \cap \bd D$ such that $z \in \gamma$ and $w^-(\gamma,D)<z$ imply $z \notin \ov{D}$. Similarly, suppose $\gamma \in \sC$ is a heteroclinic with $w^+(\gamma)=w$, so $\gamma$ is asymptotic to an attracting direction at $w$. If $a^+$ is the point where $\bd D$ meets the radial line at $w$ in this attracting direction, then every $z \in \gamma \cap \bd D$ satisfies $|z-a^+|=o(\ve)$ as $\ve \to 0$. The least point on $\gamma \cap \bd D$ (in the intrinsic potential order) is denoted by $w^+(\gamma,D)$. Thus, $w^+(\gamma,D)$ is characterized as the point on $\gamma \cap \bd D$ such that $z \in \gamma$ and $z<w^+(\gamma,D)$ imply $z \notin \ov{D}$. \vs

Now suppose $\gamma$ is a homoclinic in $\sC$ so $\gamma$ is asymptotic to a pair of repelling and attracting directions at $w$. If $a^-, a^+$ are the points where $\bd D$ meets the radial line at $w$ in these repelling and attracting directions, then $|a^--a^+|\asymp \ve$ but every $z \in \gamma \cap \bd D$ satisfies $|z-a^-|=o(\ve)$ or $|z-a^+|=o(\ve)$ depending on which end of $\gamma$ the point $z$ is close to. In other words, the finite set $\gamma \cap \bd D$ is partitioned into two subsets, one near $a^-$ and the other near $a^+$, unambiguously separated for $\ve$ sufficiently small. By definition, the greatest point of $ \gamma \cap \bd D$ in the first subset is denoted by $w^-(\gamma,D)$ and the least point of $\gamma \cap \bd D$ in the second subset is denoted by $w^+(\gamma,D)$. Notice that by this definition $w^-(\gamma,D)<w^+(\gamma,D)$ and the segment of $\gamma$ between $w^\pm(\gamma,D)$ is outside $\ov{D}$. Recall that $\Delta_\gamma$ denotes the Jordan domain bounded by $\ov{\gamma}=\gamma \cup \{ w \}$. We define $I_\gamma$ to be the closed arc of the circle $\bd D$ bounded by $w^\pm(\gamma,D)$ which is nearly contained in $\Delta_\gamma$ in the sense that the length of $I_\gamma \sm \Delta_\gamma$ is $o(\ve)$ (see \figref{gooddisk}). \vs

\begin{figure}[]
	\centering
	\begin{overpic}[width=0.5\textwidth]{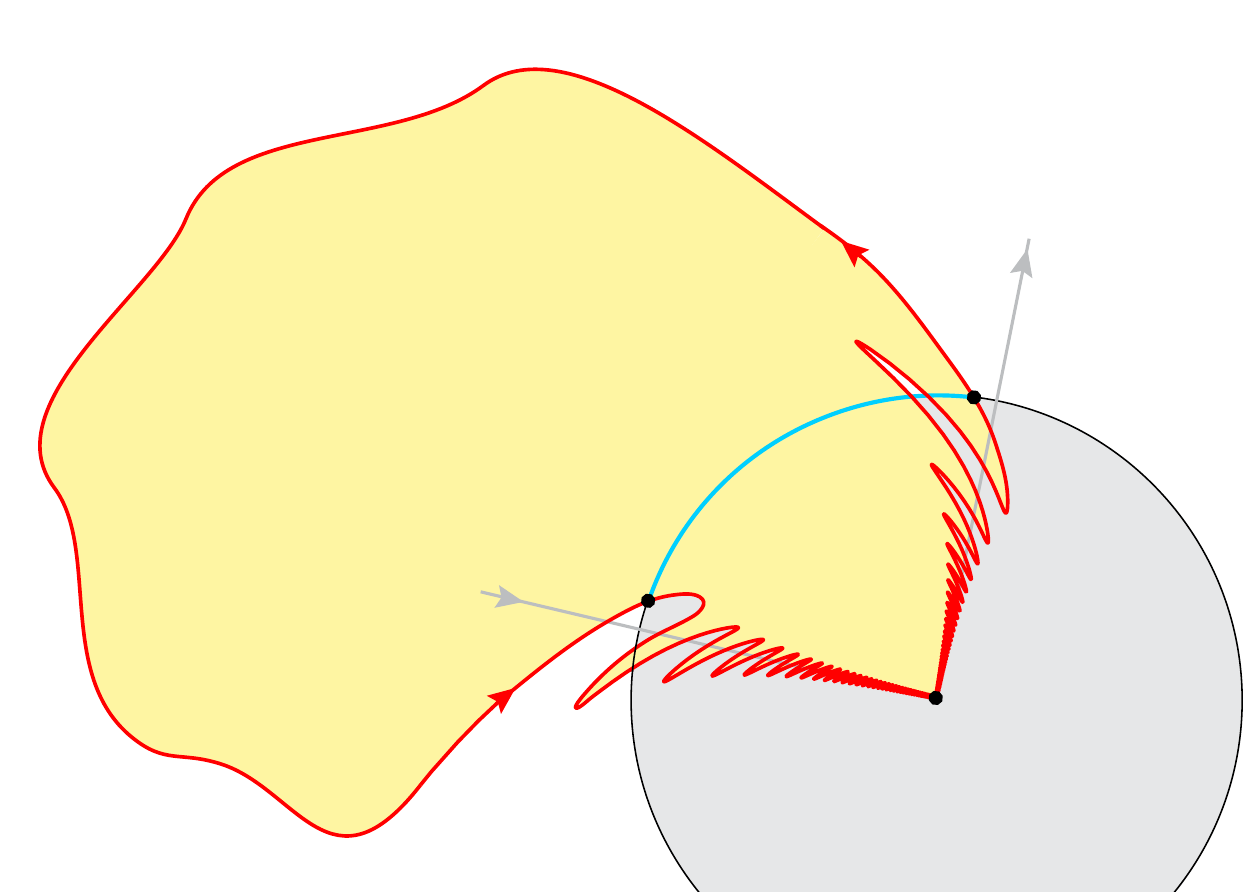}
\put (54,35) {\color{myblue}$I_\gamma$}
\put (70,51) {\color{red}$\gamma$}
\put (30,40) {$\Delta_\gamma$}
\put (73,12.3) {\footnotesize{$w$}}
\put (77,40.4) {\footnotesize{$w^-(\gamma,D)$}}
\put (48.8,24.6) {\footnotesize{$w^+(\gamma,D)$}}
\put (90,10) {\color{gray}$D$}
	\end{overpic}
\caption{\footnotesize A good disk $D$ centered at $w$, a homoclinic arc $\gamma$, and the points $w^-(\gamma,D)$ and $w^+(\gamma,D)$. The closed arc $I_\gamma$ bounded by $w^\pm(\gamma,D)$ and nearly contained in $\Delta_\gamma$ is highlighted in blue.}
\label{gooddisk}
\end{figure}
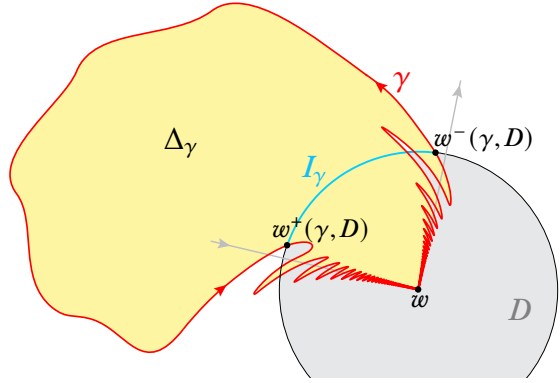

The following basic properties will be used in the next section and can be easily verified. Suppose, as above, that $\gamma \in \sC$ is a homoclinic. 

\begin{enumerate}
\item[(P1)] 
If $\eta \in \sC$ is a homoclinic in the same earring as $\gamma$, then $I_\eta \subset I_\gamma$ or $I_\gamma \subset I_\eta$. \vs

\item[(P2)] 
If $\eta \in \sC$ is {\it not} a homoclinic in the same earring as $\gamma$, then neither of $w^\pm(\eta,D)$ (when defined) can be in $I_\gamma$. 
\end{enumerate} 

\subsection{Good transversals}\label{gt}

We now turn to another construction that will be useful for our purposes. Let us work with the compact set $\sL_{\leq 1} := \sL \sm (R (]1,+\infty[) \cup \{ \infty \})$, i.e., the result of truncating $\sL$ beyond Green's potential $1$. It will be convenient to use the term {\bit chain} to describe a finite sequence of adjacent homoclinics in the same earring starting with the outermost. In other words, the homoclinics $\eta_1, \ldots, \eta_n$ form a chain if $\eta_1$ is the outermost in its earring and $\ov{\Delta}_{\eta_j} \supsetneq \ov{\Delta}_{\eta_{j+1}}$ and $(\Delta_{\eta_j} \sm \ov{\Delta}_{\eta_{j+1}}) \cap \sL =\es$ for all $1 \leq j \leq n-1$. We refer to a chain of length $n$ as an {\bit $n$-chain}.   

\begin{definition}
A smooth embedded arc $\Sigma: [0,1[ \to \CC$ with $\lim_{t \to 1} \Sigma(t)=\infty$ is called a {\bit good transversal} for $\sL$ if $\Sigma$ intersects $\sL_{\leq 1}$ transversally at finitely many points $z_1, \ldots, z_n$ such that \vs
\begin{enumerate}
\item[$\bullet$] 
either $n=1$ and $z_1$ belongs to a heteroclinic arc or $R$, \vs

\item[$\bullet$]
or $z_1, \ldots, z_n$ belong to an $n$-chain of homoclinics $\eta_1, \ldots, \eta_n$, respectively.    
\end{enumerate}
\end{definition}

\begin{lemma}[Existence of good transversals]\label{egt}
\mbox{} 
\begin{enumerate}
\item[(i)] 
Suppose $z$ belongs to a heteroclinic or $R(]0,1[)$. Then there is a good transversal $\Sigma$ with $\Sigma \cap \sL_{\leq 1}= \{ z \}$. \vs 
\item[(ii)]
Suppose $z_1, \ldots, z_n$ belong to an $n$-chain of homoclinics $\eta_1, \ldots, \eta_n$, respectively. Then there is a good transversal $\Sigma$ with $\Sigma \cap \sL_{\leq 1}= \{ z_1, \ldots, z_n \}$. \vs
\item[(iii)]
Given finitely many distinct points of type (i) and collections of points of type (ii) in different earrings, we can choose corresponding good transversals that are pairwise disjoint.    
\end{enumerate} 
\end{lemma}
 
\begin{proof} 
Let $\tau_1, \ldots, \tau_k$ denote the outermost homoclinics of all the earrings in $\sL$. Consider the union $\hat{\sL}$ of the closed disks $\ov{\Delta}_{\tau_1}, \ldots, \ov{\Delta}_{\tau_k}$ together with all heteroclinics, the ray segment $R(]0,1])$, and all points in $\sL \cap J$. In other words, $\hat{\sL}$ is the ``filled in'' $\sL_{\leq 1}$. Evidently $\hat{\sL}$ is a full compact subset of $\CC$ containing $\sL_{\leq 1}$ with piecewise analytic boundary. Using the non-dynamical ``external rays'' of the uniformization $(\Chat \sm \ov{\DD}, \infty) \oset[-0.2ex]{\cong}{\longrightarrow} (\Chat \sm \hat{\sL}, \infty)$ we see that every $z \in \bd \hat{\sL}$ is the landing point of at least one ray in $\CC \sm \hat{\sL}$. If $z \in \bd \hat{\sL} \sm (J \cup R(1))$, each ray landing at $z$ meets the $\sL$-arc through $z$ orthogonally, so it can be extended ever so slightly past its $z$-end to become a good transversal with $\Sigma \cap \sL_{\leq 1} = \{ z \}$. This proves part (i) and part (ii) for $1$-chains. \vs

If $z_1, \ldots, z_n$ lie on an $n$-chain $\eta_1, \ldots, \eta_n$ for $n \geq 2$, take a good transversal $\Sigma$ with $\Sigma \cap \sL_{\leq 1} = \{ z_1 \}$ as above. It is then easy to extend $\Sigma$ smoothly all the way inside $\Delta_{\eta_n}$, crossing $\eta_j$ once transversally at $z_j$ for $2 \leq j \leq n$. This proves part (ii) for $n \geq 2$. \vs

Part (iii) follows from the fact that distinct ``external rays'' are disjoint.   
\end{proof}

\subsection{Linked and unlinked pairs}

Let $D$ be a round disk in $\CC$. Pairs $(a_1,a_2), (b_1,b_2)$ of distinct points on the boundary circle $\bd D$ are said to be {\bit linked} if $b_1$ and $b_2$ lie in different connected components of $\bd D \sm \{ a_1, a_2 \}$. Otherwise, $(a_1,a_2), (b_1,b_2)$ are called {\bit unlinked}. A collection of pairs on $\bd D$ is unlinked if every two pairs in the collection are unlinked. It is customary to  represent a pair $(a_1,a_2) \in \bd D$ by the hyperbolic geodesic in $D$ with endpoints at $a_1$ and $a_2$. An unlinked collection is then visualized as one whose representative geodesics are pairwise disjoint. \vs

If $(a_1,a_2), (b_1,b_2)$ on $\bd D$ are linked, any two paths in $D$ that connect $a_1$ to $a_2$ and $b_1$ to $b_2$ must intersect. This is an easy consequence of the Jordan curve theorem. The following lemma is a generalization of this fact. Recall that a point $z$ on the boundary of a simply connected domain $U \subsetneq \CC$ is {\bit uniaccessible} if $\bd U \sm \{ z \}$ is connected. Equivalently, if for any base point $z_0 \in U$ there is a unique up to homotopy arc in $U$ that connects $z_0$ to $z$.      

\begin{lemma}\label{Usim}
Suppose $(a_1,a_2), (b_1,b_2)$ on $\bd D$ are linked. Let $U$ be any simply connected domain with $D \subset U \subset \CC$ such that $a_1,a_2,b_1,b_2$ are uniaccessible points of $\bd U$. Then any two paths in $U$ that connect $a_1$ to $a_2$ and $b_1$ to $b_2$ must intersect.
\end{lemma}

The assumptions that $D \subset U$ and $a_1,a_2,b_1,b_2$ are uniaccessible are both necessary; compare \figref{linked}. 

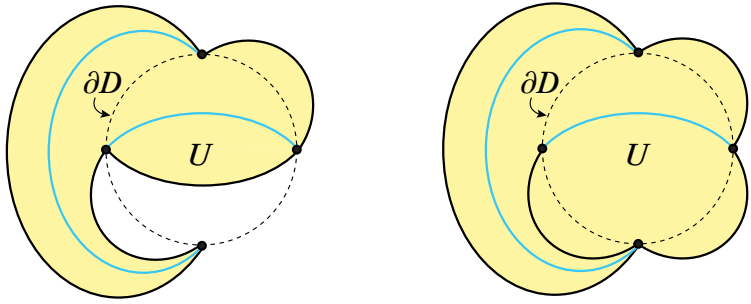
\begin{figure}[]
	\centering
	\begin{overpic}[width=0.7\textwidth]{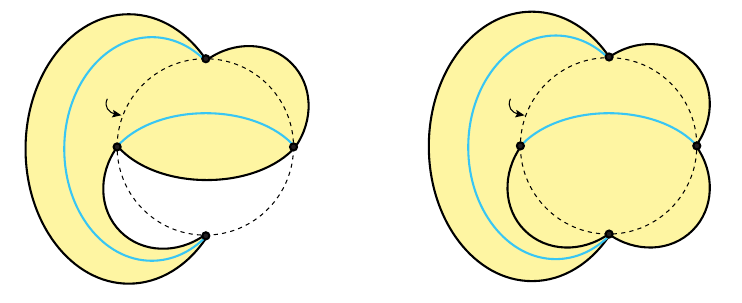}
\put (80.5,19) {$U$}
\put (26,19) {$U$}
\put (13,28) {$\bd D$}
\put (67.5,28) {$\bd D$}
	\end{overpic}
\caption{\footnotesize Examples of simply connected domains $U$ that violate the conclusion of \lemref{Usim}. In both cases there are non-intersecting paths that connect linked pairs on $\bd D \cap \bd U$.}
\label{linked}
\end{figure}

\begin{proof}
Let $w$ be the center of $D$ and take a conformal isomorphism $\phi: (U,w) \oset[-0.2ex]{\cong}{\longrightarrow} (\DD,0)$. By elementary conformal mapping theory the four radial lines in $U$ starting at $w$ and landing on $a_1,a_2,b_1,b_2$ map under $\phi$ to four disjoint paths in $\DD$ starting at $0$ and landing at distinct points $a'_1,a'_2,b'_1,b'_2$. The pairs $(a'_1,a'_2), (b'_1,b'_2)$ on the unit circle $\bd \DD$ are linked because $\phi$ preserves the cyclic order of the radial lines near $w$. Now since $a_1, a_2,b_1,b_2$ are uniaccessible, any two paths in $U$ that connect $a_1$ to $a_2$ and $b_1$ to $b_2$ map under $\phi$ to two paths in $\DD$ that connect the same pairs $(a'_1,a'_2)$ and $(b'_1,b'_2)$, and the result follows.  
\end{proof}
                
\subsection{The order and structure of homoclinic arcs}\label{oshom}

The proof of \thmref{B} will be based on a series of statements about the order of the homoclinics that are based at a given point in $\sL \cap J$ (Corollaries \ref{homuni}-\ref{earlast}). We derive these statements from the following topological result which will also be used repeatedly in \S \ref{BD}:   

\begin{theorem}\label{unlinked}
Fix $w \in \sL \cap J$ and $h \geq 1$. Let $\eta_0< \cdots <\eta_{h+1}$ be any collection of $\sL$-arcs such that $\eta_1, \ldots, \eta_h$ are homoclinics based at $w$ and $w^+(\eta_0)=w^-(\eta_{h+1})=w$ (so $\eta_0, \eta_{h+1}$ may or may not be homoclinics). Assume further that the homoclinics in this collection form a union of chains. Take a sufficiently small good disk $D$ centered at $w$ for the collection $\{ \eta_0, \ldots, \eta_{h+1} \}$ and let 
$$
v_j := w^+(\eta_j,D) \quad \text{and} \quad u_{j+1} := w^-(\eta_{j+1},D) \qquad (0 \leq j \leq h),
$$
so $v_0<u_1<v_1<\cdots<u_h<v_h<u_{h+1}$. Then the set of pairs 
$$
\{ (v_0,u_1), \ (v_1,u_2), \ \ldots, \ (v_h,u_{h+1}) \}
$$ 
on $\bd D$ is unlinked. 
\end{theorem}  

\begin{proof}
For each $0 \leq j \leq h+1$ choose a point $z_j \in \eta_j$, with $z_{h+1}<R(1)$. By \lemref{egt} we can take pairwise disjoint good transversals $\Sigma_1, \ldots, \Sigma_r$ so that $\bigcup_{i=1}^r \Sigma_i \cap \sL_{\leq 1} = \{ z_0, \ldots, z_{h+1} \}$. By choosing the good disk $D$ sufficiently small we can guarantee that the $\Sigma_i$ are disjoint from $\ov{D}$ and that
\begin{equation}\label{ccc} 
z_0<v_0<u_1<z_1<v_1<\cdots<u_h<z_h<v_h<u_{h+1}<z_{h+1}<R(1). 
\end{equation}
By transversality, for all large $n$ the ray $R_n$ meets $\bd D$ at nearby points 
$$
v_{n,j} := R_n(t_{n,j}) \quad \text{and} \quad u_{n,j+1} := R_n(s_{n,j+1}) \qquad (0 \leq j \leq h).
$$ 
Here 
$$
\langle t_{n,0} \rangle < \langle s_{n,1} \rangle < \langle t_{n,1} \rangle < \ldots < \langle s_{n,h} \rangle < \langle t_{n,h} \rangle < \langle s_{n,h+1} \rangle
$$ 
are the respective preimages of $v_0,u_1,v_1, \ldots,u_h,v_h,u_{h+1}$ under the homeomorphism $\Pi: \sS \to \sL \sm (J \cup \{ \infty \})$ of \S\ref{ssord}. Complete the construction at the two ends by choosing $u_0:=\Pi(\langle s_{n,0} \rangle) \in \eta_0$ and $v_{h+1}:=\Pi(\langle t_{n,h+1} \rangle) \in \eta_{h+1}$ such that $u_0<z_0$ and $z_{h+1}<v_{h+1}<R(1)$, and set $u_{n,0}:=R_n(s_{n,0}), v_{n,h+1}:=R_n(t_{n,h+1})$ (see \figref{R5}). \vs 

\begin{figure}[]
	\centering
	\begin{overpic}[width=\textwidth]{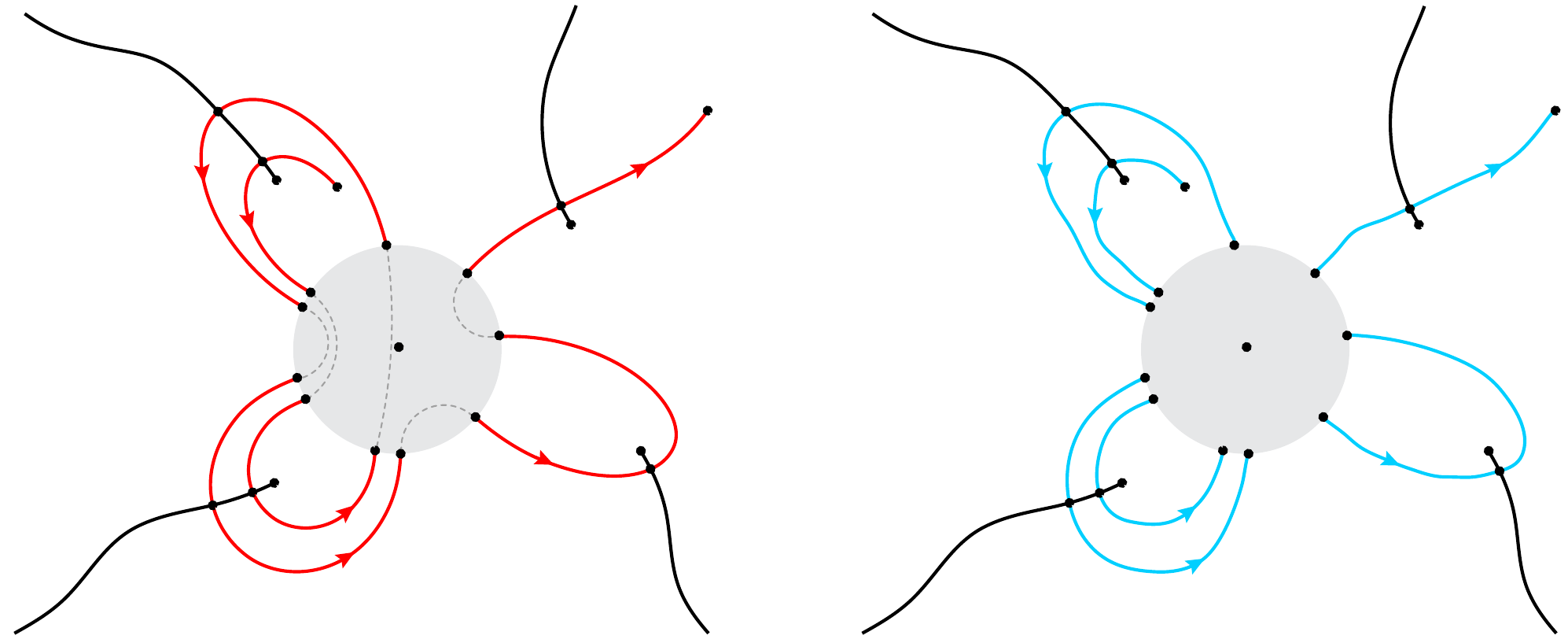}
\put (17.3,25) {\color{red}\small{$\tilde{\eta}_0$}}
\put (20,9.3) {\color{red}\small{$\tilde{\eta}_1$}}
\put (10,28) {\color{red}\small{$\tilde{\eta}_2$}}
\put (22,3) {\color{red}\small{$\tilde{\eta}_3$}}
\put (33,9) {\color{red}\small{$\tilde{\eta}_4$}}
\put (40,27.4) {\color{red}\small{$\tilde{\eta}_5$}}
\put (2,36) {\small $\Sigma_1$}
\put (1,3) {\small $\Sigma_2$}
\put (44.5,2) {\small $\Sigma_3$}
\put (32.5,38) {\small $\Sigma_4$}
\put (81,16) {\color{gray}\small{$D$}}
\put (25.7,17) {\footnotesize{$w$}}
\put (78.5,16.7) {\footnotesize{$w$}}
\put (16.6,31.2) {\footnotesize{$z_0$}}
\put (14.7,7.4) {\footnotesize{$z_1$}}
\put (13.3,34.8) {\footnotesize{$z_2$}}
\put (11.8,6.8) {\footnotesize{$z_3$}}
\put (42.3,10) {\footnotesize{$z_4$}}
\put (33.3,27.8) {\footnotesize{$z_5$}}
\put (21,27.2) {\tiny{$u_0$}}
\put (19,13.8) {\tiny{$u_1$}}
\put (25.2,25) {\tiny{$u_2$}}
\put (17.3,17.3) {\tiny{$u_3$}}
\put (30.5,14.4) {\tiny{$u_4$}}
\put (28,24) {\tiny{$u_5$}}
\put (19.9,22.7) {\tiny{$v_0$}}
\put (22,11.7) {\tiny{$v_1$}}
\put (18,19.8) {\tiny{$v_2$}}
\put (26.1,11.6) {\tiny{$v_3$}}
\put (31,20.2) {\tiny{$v_4$}}
\put (44.5,34.5) {\tiny{$v_5$}}
\put (74,27.5) {\tiny{$u_{n,0}$}}
\put (73,13.8) {\tiny{$u_{n,1}$}}
\put (79.2,25) {\tiny{$u_{n,2}$}}
\put (71.7,17.5) {\tiny{$u_{n,3}$}}
\put (84.3,14.8) {\tiny{$u_{n,4}$}}
\put (82.2,22.3) {\tiny{$u_{n,5}$}}
\put (73.2,22.8) {\tiny{$v_{n,0}$}}
\put (77,12.8) {\tiny{$v_{n,1}$}}
\put (72,19.9) {\tiny{$v_{n,2}$}}
\put (80.1,11.6) {\tiny{$v_{n,3}$}}
\put (85,20.2) {\tiny{$v_{n,4}$}}
\put (98.5,34.5) {\tiny{$v_{n,5}$}}

	\end{overpic}
\caption{\footnotesize Illustration of the proof of \thmref{unlinked}. Left: The subarcs $\tilde{\eta}_j$ of $\eta_j$ are shown in red and the good transversals $\Sigma_i$ in black. The union of red and black is the set $E$. Right: The arcs in blue are segments of the external ray $R_n$ for a large $n$ that uniformly approximate the $\tilde{\eta}_j$. The union of blue and black is the set $E_n$. Both $\CC \sm E$ and $\CC \sm E_n$ are simply connected.}
\label{R5}
\end{figure}
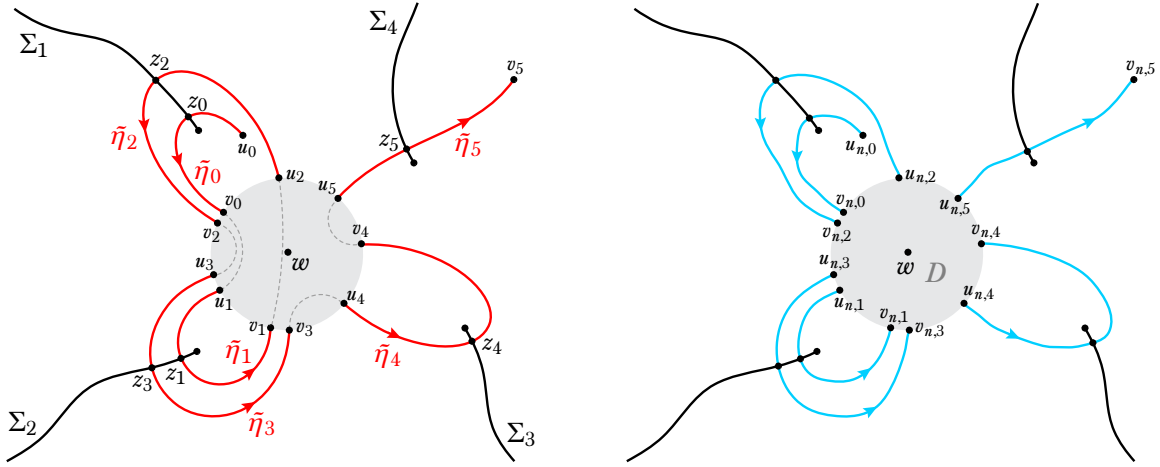

For $0 \leq j \leq h+1$ let $\tilde{\eta}_j$ denote the subarc of $\eta_j$ that joins $u_j$ to $v_j$. The closed set 
$$
E := \bigcup_{i=1}^r \Sigma_i \cup \bigcup_{j=0}^{h+1} \tilde{\eta}_j 
$$ 
has simply connected complement in $\CC$ (this is the reason why we introduced the good transversals $\Sigma_i$). For each $0 \leq j \leq h+1$ the ray segment $R_n([s_{n,j},t_{n,j}])$ that joins $u_{n,j}$ to $v_{n,j}$ tends to the subarc $\tilde{\eta}_j$ in $C^\infty$-topology as $n \to \infty$. Thus, for large $n$ the closed set 
$$
E_n := \bigcup_{i=1}^r \Sigma_i \cup \bigcup_{j=0}^{h+1} R_n([s_{n,j},t_{n,j}]) 
$$ 
is $C^\infty$-close to $E$. Since all the intersections in $E$ are transversal, it follows that the complement $\CC \sm E_n$ is also simply connected (see \figref{R5}). \vs

Now suppose the pairs $(v_j,u_{j+1})$ and $(v_k,u_{k+1})$ are linked for some $0 \leq j<k \leq h$. Then for all large $n$ the pairs $(v_{n,j},u_{n,j+1})$ and $(v_{n,k},u_{n,k+1})$ are linked as well. At least one of the open ray segments $R_n (]t_{n,j},s_{n,j+1}[)$ between $v_{n,j}$ and $u_{n,j+1}$ or $R_n (]t_{n,k},s_{n,k+1}[)$ between $v_{n,k}$ and $u_{n,j+k}$ must meet $E_n$; otherwise by \lemref{Usim} these ray segments would have to intersect, which is impossible. Since these ray segments are clearly disjoint from the ray segments in $E_n$, one of them must intersect $\bigcup_{i=1}^r \Sigma_i$. In other words, for all large $n$ the ray $R_n$ intersects $\bigcup_{i=1}^r \Sigma_i$ at some point $R_n(\lambda_n)$, where 
\begin{equation}\label{asat1}
t_{n,j}<\lambda_n<s_{n,j+1} \qquad \text{or} \qquad t_{n,k}<\lambda_n<s_{n,k+1}. 
\end{equation}
Any accumulation point $\tilde{z}$ of the sequence $\{ R_n(\lambda_n) \}$ must then belong to $\bigcup_{i=1}^r \Sigma_i \cap \sL_{\leq 1} = \{ z_0, \ldots, z_{h+1} \}$. But then \lemref{EE} together with \eqref{asat1} implies that $v_j \leq \tilde{z} \leq u_{j+1}$ or $v_k \leq \tilde{z} \leq u_{k+1}$, contradicting \eqref{ccc}.
\end{proof}

We now gather several corollaries of \thmref{unlinked}.

\begin{corollary}\label{homuni}
Two distinct $\sL$-arcs in a given parabolic basin must be homoclinics belonging to the same earring. In particular, every parabolic basin contains at most one earring of homoclinics. 
\end{corollary}
 
\begin{proof}
We already know from \corref{hetuni}(i) that a parabolic basin $B=P^{\circ q}(B)$ contains at most one heteroclinic. Thus, we must rule out a homoclinic/heteroclinic pair or a homoclinic/homoclinic pair in different earrings in $B$. Assume by way of contradiction that $B$ contains a homoclinic $\eta$ and an $\sL$-arc $\gamma$ not in the earring of $\eta$. Without loss of generality we can take $\eta$, and $\gamma$ if it is also a homoclinic, to be the outermost elements in their respective earrings. Set $w:=w^+(\gamma)=w^+(\eta)$ and let $\xi$ be the unique heteroclinic with $w^-(\xi)=w$, or $\xi=R$ if $w=w_0$. By \corref{hetuni}(ii), $\gamma<\eta<\xi$ if $\gamma$ is a heteroclinic. We may assume the same order even if $\gamma$ is a homoclinic (simply swap $\eta$ and $\gamma$ if necessary). \vs

Take a small good disk $D$ centered at $w$ for the collection $\{ \gamma, \eta, \xi \}$ and set  
$$
v_0:=w^+(\gamma,D), \quad u_1:=w^-(\eta,D), \quad v_1:=w^+(\eta,D), \quad u_2:=w^-(\xi,D).
$$
Since $\gamma$ and $\xi$ are not in the earring of $\eta$, neither of the points $v_0,u_2$ belongs to the arc $I_\eta \subset \bd D$ bounded by $u_1,v_1$ (property (P2) of good disks in \S \ref{gd}). Moreover, $\gamma, \eta$ are in the same basin so $v_0,v_1$ are asymptotically close to the same attracting direction, while $u_2$ is asymptotically close to a repelling direction. Thus, the pairs $(v_0,u_1), (v_1,u_2)$ must be linked. This contradicts \thmref{unlinked}. 
\end{proof} 

The next corollary shows that the intrinsic potential order on the set of homoclinics in an earring is compatible with the order coming from embedding in the plane:
 
\begin{corollary}\label{homord}
Suppose $\gamma, \eta$ are homoclinics in the same earring, with $\Delta_\eta \subset \Delta_\gamma$. Then, $\eta<\gamma$. 
\end{corollary}

\begin{proof}
Label the homoclinics in the earring $\eta_1, \eta_2, \eta_3, \ldots$ so that $\Delta_{\eta_1} \supset \Delta_{\eta_2} \supset \Delta_{\eta_3} \supset \cdots$. It suffices to show that $\eta_j < \eta_{j-1}$ for all $j$. Suppose there is a least index $n \geq 2$ such that the opposite order $\eta_{n-1}<\eta_n$ holds. Then $\eta_n$ has an immediate predecessor in the $n$-chain $\{ \eta_1, \ldots, \eta_n \}$ with respect to $<$, i.e., there is a unique $1 \leq k \leq n-1$ such that $\eta_k<\eta_n$ and no other $\eta_j$ comes between $\eta_k,\eta_n$. We consider two cases: \vs

{\it Case 1.} $\eta_n$ has an immediate successor in this $n$-chain, i.e., there is a unique $1 \leq \ell \leq n-1$ such that $\eta_n<\eta_\ell$ and no other $\eta_j$ comes between $\eta_n,\eta_\ell$. Note that $k > \ell$ by minimality of $n$. Let $w$ be the point where the earring is based at. Take a sufficiently small good disk $D$ centered at $w$ for the collection $\{ \eta_1, \ldots, \eta_n \}$ and set  
$$
v_0:=w^+(\eta_k,D), \quad u_1:=w^-(\eta_n,D), \quad v_1:=w^+(\eta_n,D), \quad u_2:=w^-(\eta_\ell,D).
$$   
Since $\Delta_{\eta_n} \subset \Delta_{\eta_k} \subset \Delta_{\eta_\ell}$, we have $I_{\eta_n} \subset I_{\eta_k} \subset I_{\eta_\ell}$, which shows that the pairs $(v_0,u_1), (v_1,u_2)$ must be linked. This contradicts \thmref{unlinked}. \vs 

{\it Case 2.} $\eta_n$ has no successor in this $n$-chain. Let $\xi$ be the unique heteroclinic with $w^-(\xi)=w$, or $\xi=R$ if $w=w_0$. By \corref{hetuni}(ii), $\eta_j<\xi$ for all $1 \leq j \leq n$.  Take a sufficiently small good disk $D$ centered at $w$ for the collection $\{ \eta_1, \ldots, \eta_n, \xi \}$ and set  
$$
v_0:=w^+(\eta_k,D), \quad u_1:=w^-(\eta_n,D), \quad v_1:=w^+(\eta_n,D), \quad u_2:=w^-(\xi,D).
$$   
Since $\Delta_{\eta_n} \subset \Delta_{\eta_k}$, we have $I_{\eta_n} \subset I_{\eta_k}$ but $u_2 \notin I_{\eta_k}$ (properties (P1) and (P2) of good disks in \S \ref{gd}). It follows that the pairs $(v_0,u_1),(v_1,u_2)$ are linked, which again contradicts \thmref{unlinked}. 
\end{proof}

\begin{corollary}\label{earlast}
Suppose there is an earring based at $w \in \sL \cap J$ containing at least two distinct homoclinics. Then there can be no heteroclinic $\gamma$ with $w^+(\gamma)=w$. 
\end{corollary}

\begin{proof}
Suppose such $\gamma$ exists. Let $\xi$ be the outermost homoclinic in the given earring and $\eta$ be the adjacent homoclinic inside $\gamma$. A combination of Corollaries \ref{hetuni}(ii) and \ref{homord} then shows that $\gamma<\eta<\xi$. Take a sufficiently small good disk $D$ centered at $w$ for the collection $\{ \gamma, \eta, \xi \}$ and set 
$$
v_0:=w^+(\gamma,D), \ \ u_1:=w^-(\eta,D), \ \ v_1:=w^+(\eta,D), \ \ u_2:=w^-(\xi,D), \ \ v_2:=w^+(\xi,D).
$$
We have $I_\eta \subset I_\xi$ but $v_1 \notin I_\xi$. It follows that the pairs $(v_0,u_1),(v_1,u_2)$ are linked, contradicting \thmref{unlinked}.   
\end{proof}

\begin{proof}[Proof of \thmref{B}]
By \thmref{A} the complement $\sL \sm \sL^*$ is either empty or consists of finitely many earrings attached to the points of $\sL \cap J$. By \corref{homuni} none of these earrings can share its basin with another earring or heteroclinic arc. By \corref{earlast} every earring with at least two homoclinics must be based at $w_N=\zeta_\infty$.   
\end{proof}

\subsection{Comment on the order of homoclinic arcs} 

As the final word of this section, let us comment on the relative order of the homoclinics in two earrings based at the same point. Suppose $\{ \eta_j \}$ and $\{ \xi_j \}$ are distinct earrings based at $w_N=\zeta_\infty$, labeled so that $\eta_{j+1}$ is inside $\eta_j$ and $\xi_{j+1}$ is inside $\xi_j$ for all $j$. By \corref{homord}, $\eta_{j+1} < \eta_j$ and $\xi_{j+1} < \xi_j$ for all $j$. Without loss of generality assume $\xi_1<\eta_1$. Then an inductive argument using \thmref{unlinked}, which we shall omit, shows that the two earrings must {\it order-interlace}:
$$
\cdots < \xi_3 < \eta_3 < \xi_2 < \eta_2 < \xi_1 < \eta_1.
$$
An explicit example of this phenomenon, communicated to us by H. Inou, is provided by suitable perturbations of the cubic $P(z)=z+z^3$ of the form $P_n(z)=\lambda_n z+z^3$, where $|\lambda_n|>1$ and $\lambda_n \to 1$ tangentially, as illustrated in \figref{double}. For large $n$ the fixed rays $R_{P_n,0}$ follow a double spiral towards their landing point at $0$. The Hausdorff limit $\sL=\lim_{n \to \infty} \ov{R_{P_n,0}}$ consists of $\ov{R_{P,0}}$ together with two order-interlacing earrings, each contained in one of the invariant basins of the parabolic fixed point at $0$. 

\begin{figure}[t]
\centering 
\begin{overpic}[width=\textwidth]{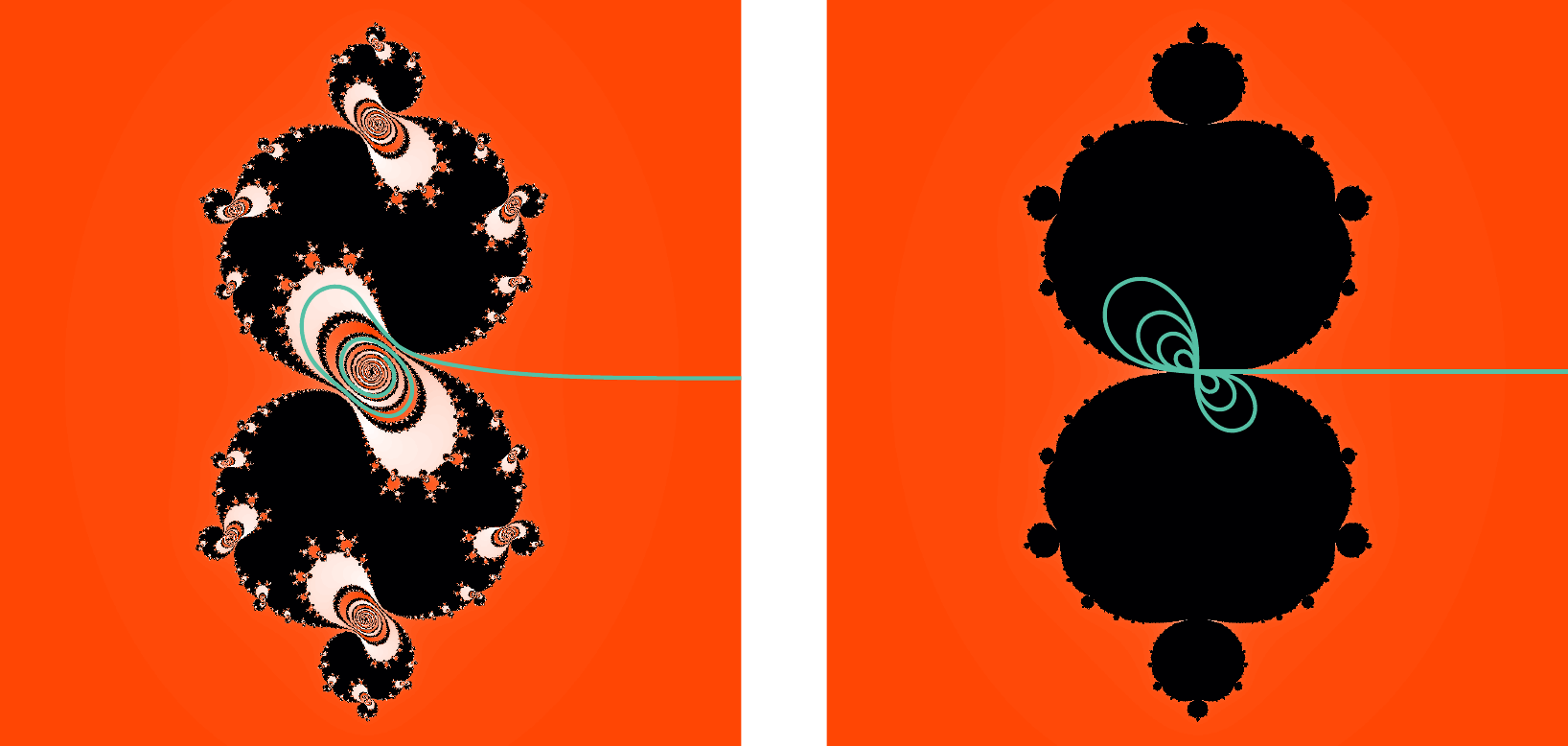}
\end{overpic}
\caption{\footnotesize Perturbations of the cubic $P(z)=z+z^3$ with a degenerate parabolic fixed point at $0$. The Hausdorff limit $\sL$ of the closed ray at angle $0$ contains two order-interlacing earrings based at the same point. The Hausdorff limit of the closed ray at angle $1/2$ contains a similar pair of earrings (not shown).}
\label{double}
\end{figure}

\section{Proof of \thmref{C}}\label{BD}

It is easy to see that the period of a point in $\sL \cap J$ can be a proper divisor of the ray period $q$. As the simplest example, suppose $\zeta_0$ is a repelling fixed point of $P$ of combinatorial rotation number $\neq 0$ and $R_{P,\theta}$ is a periodic $q$ ray landing at $\zeta_0$. Then any sequence $P_n \to P$ will produce a tame convergence $\ov{R_{P_n,\theta}} \to \ov{R_{P,\theta}}$ (\thmref{stab}). The same holds if $\zeta_0$ is a non-degenerate parabolic point and the sequence of perturbations is chosen such that their multiplier at $\zeta_0$ tends to the corresponding root of unity non-tangentially (work in progress; see the introduction). Figure \ref{pers} illustrates a more subtle example of this drop in period involving a wild convergence.

\begin{figure}[t]
\centering 
\begin{overpic}[width=\textwidth]{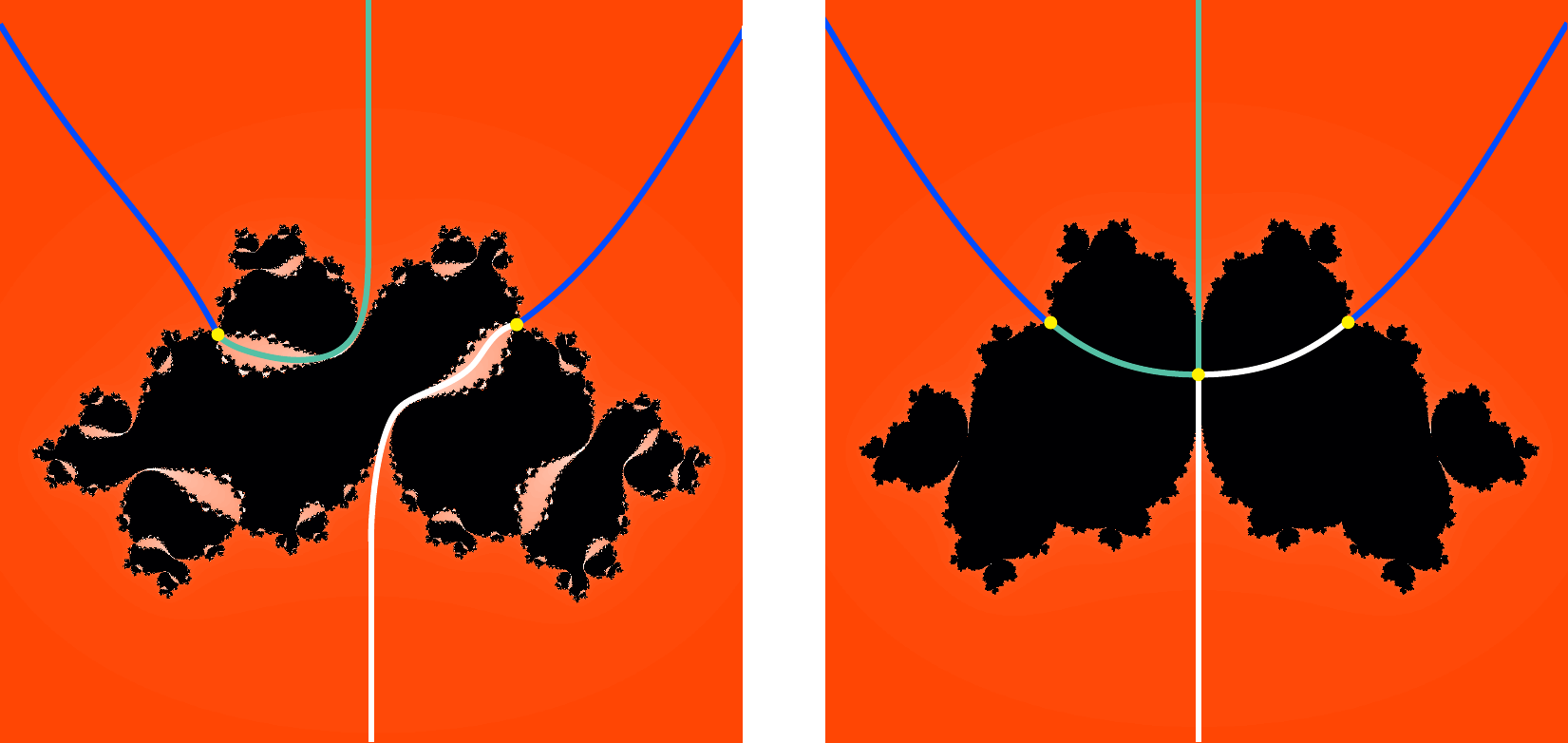}
\put (77.2,21.3) {\color{white}$w_0$}
\put (65.2,24.2) {\color{white}$w_1$}
\put (24,45) {\color{white}\footnotesize{$1/4$}}
\put (24.3,1) {\color{white}\footnotesize{$3/4$}}
\put (43,40) {\color{white}\footnotesize{$1/8$}}
\put (2,40) {\color{white}\footnotesize{$3/8$}}
\put (77,45) {\color{white}\footnotesize{$1/4$}}
\put (77,1) {\color{white}\footnotesize{$3/4$}}
\put (95,40) {\color{white}\footnotesize{$1/8$}}
\put (54,40) {\color{white}\footnotesize{$3/8$}}
\end{overpic}
\caption{\footnotesize Perturbations of the cubic $P(z)=-z+0.6\ii \, z^2+z^3$ with a parabolic fixed point at $0$. Here the Hausdorff limit $\sL$ of the closed ray at angle $1/4$ (shown in green) contains a unique heteroclinic connecting the fixed point $w_0=\zeta=0$ to the repelling period $2$ point $w_1=\zeta_\infty$. The Hausdorff limit of the closed ray at angle $3/4$ is the image $P(\sL)$ (shown in white). Notice the tame behavior of the period $2$ rays at angle $1/8,3/8$ (shown in blue).}
\label{pers}
\end{figure}

\subsection{The iterated images of $\sL$}

Since $P^{\circ q}$ acts homeomorphically on $\sL$, each restriction $P^{\circ i}: \sL \to P^{\circ i}(\sL)$ must be a homeomorphism. Here $P^{\circ i}(\sL)$ coincides with the Hausdorff limit of the sequence of periodic rays $P_n^{\circ i}(R_n)$. Let $\sL \cap J = \{ w_0 = \zeta, \ldots, w_N=\zeta_\infty \}$ and $\gamma_1, \ldots, \gamma_N$ be the heteroclinic arcs in $\sL$, with $\gamma_j$ joining $w_j$ to $w_{j-1}$ (we adopt the usual convention that if there are no heteroclinics then $N=0$ so $w_0=w_N$). Then $P^{\circ i}(\sL) \cap J$ consists of the points $P^{\circ i}(w_j)$, and the arcs $P^{\circ i}(\gamma_j)$ are the heteroclinics in $P^{\circ i}(\sL)$. Moreover, if $\eta$ is a homoclinic arc of $\sL$ based at $w_j$, then $P^{\circ i}(\eta)$ is a homoclinic arc of $P^{\circ i}(\sL)$ based at $P^{\circ i}(w_j)$. This shows that the Hausdorff limits $P^{\circ i}(\sL)$ have the same number and combinatorial structure of heteroclinics and earrings as $\sL$. The $N+1$ Julia points $P^{\circ i}(\sL) \cap J = \{ P^{\circ i}(w_0), \ldots, P^{\circ i}(w_N) \}$ appear in the same linear order as in $\sL \cap J = \{ w_0, \ldots, w_N \}$. It follows in particular that the points $w_0, \ldots, w_N$ have disjoint orbits under $P$. \vs  

Even though the $q$ rays $R, P(R), \ldots, P^{\circ q-1}(R)$ are always disjoint, the Hausdorff limits $\sL, P(\sL), \ldots, P^{\circ q-1}(\sL)$ may indeed intersect, as the example in \figref{pers} illustrates. We will show in \corref{i0} that such intersection can only occur at a {\it unique} point of the Julia set. Here is a much easier preliminary result:  \vs

\noindent
{\bf $(\dagger)$ Simple observation.} If $0 \leq i<j \leq q-1$, then $P^{\circ i}(\sL) \cap P^{\circ j}(\sL)$ is either empty or a subset of the Julia set. \vs 

\noindent
In fact, suppose $u \in P^{\circ i}(\sL) \cap P^{\circ j}(\sL) \cap \mathring{K}$ and take sequences $u_n=P_n^{\circ i}(R_n(s_n))$ and $u'_n = P_n^{\circ j}(R_n(s'_n))$ converging to $u$, where the potentials $s_n, s'_n$ necessarily tend to $0$. On the one hand $(\Om_n,u_n) \ct (V,u)$ for some disk $V \subset K$, so $\dist_{\Om_n}(u_n,u'_n) \to 0$ by \lemref{ct-easy}. On the other hand the rays $P_n^{\circ i}(R_n)$ and $P_n^{\circ j}(R_n)$ are disjoint, so $d^i\theta \neq d^j \theta \ (\modd \ZZ)$ and   
$$
\dist_{\Om_n}(u_n,u'_n) = \dist_{\Chat \sm \ov{\DD}}(\e^{d^i s_n + 2\pi \ii d^i \theta},\e^{d^j s'_n+2\pi \ii d^j \theta}) \to \infty.    
$$
  
From the above simple observation it is easy to conclude that $\sL, P(\sL), \ldots, P^{\circ q-1}(\sL)$ are disjoint if and only if $w_0, \ldots, w_N$ have exact period $q$.       

\begin{lemma}\label{period}
Either all $w_0, \ldots, w_N$ have period $q$, or there is a {\em unique} $0 \leq \ell \leq N$ such that the period of $w_\ell$ is a proper divisor of $q$.   
\end{lemma}

\begin{proof}
Let $\ell$ be the smallest index for which $w_\ell$ has period $p =q/k$ with $k>1$. If $\ell=N$ we are done, so let us assume $N \geq 1$ and $0 \leq \ell \leq N-1$. Split the spine $\sL^*$ into two arcs: $\Gamma$ from $w_\ell$ to $\infty$ and $\La$ from $w_\ell$ to $w_N$, so $\Gamma \cap \La= \{ w_\ell \}$. For $0 \leq i \leq k-1$ define $\Gamma^i:=P^{\circ ip}(\Gamma), \La^i:=P^{\circ ip}(\La)$. The minimality of $\ell$ implies that the $k$ arcs $\Gamma=\Gamma^0, \Gamma^1, \ldots, \Gamma^{k-1}$ are pairwise disjoint except at their end point $w_\ell$; otherwise $\Gamma^0$ would intersect some $\Gamma^i$ at a point other than $w_{\ell}$, which by the simple observation $(\dagger)$ would imply some $w_j$ with $j<\ell$ being in $\Gamma^0 \cap \Gamma^i$, hence having period $<q$. The arcs $\Gamma^0, \Gamma^1, \ldots, \Gamma^{k-1}$ are permuted cyclically under $P^{\circ p}$ in the manner determined by the combinatorial rotation number of $w_\ell$, which is necessarily of the form $r/k$ with $(r,k)=1$ since $\Gamma^0$ contains $R$ which under the action of $P^{\circ p}$ has period $k$. In particular, $(P^{\circ p})'(w_\ell)=\e^{2\pi \ii r/k}$. \vs

By the assumption $\ell \leq N-1$ there is a unique heteroclinic $\gamma_\ell \subset \Lambda=\Lambda^0$ with $w^+(\gamma_\ell)=w_\ell$. Setting $\gamma^i:=P^{\circ ip}(\gamma_\ell) \subset \Lambda^i$, it follows that $\gamma_\ell=\gamma^0, \gamma^1, \ldots, \gamma^{k-1}$ have a common end point $w_\ell$. By the same reasoning as above, none of these heteroclinics can intersect $\Gamma^0, \Gamma^1, \ldots, \Gamma^{k-1}$ anywhere other than $w_\ell$. As $P^{\circ p}$ permutes $\gamma^0, \gamma^1, \ldots, \gamma^{k-1}$ near $w_{\ell}$ cyclically with the same combinatorial rotation number $r/k$, each of the $k$ sectors of $\CC \sm (\Gamma^0 \cup \Gamma^1 \cup \ldots \cup \Gamma^{k-1})$ must contain exactly one of these heteroclinics. In particular, $\gamma^0, \gamma^1, \ldots, \gamma^{k-1}$ are pairwise disjoint except at their common end point $w_\ell$. It follows that each arc $\Lambda^i$ is contained in the same sector of $\CC \sm (\Gamma^0 \cup \Gamma^1 \cup \ldots \cup \Gamma^{k-1})$ as $\gamma^i$, hence $\Lambda^0, \Lambda^1, \ldots, \Lambda^{k-1}$ are also pairwise disjoint except at $w_\ell$. Thus, the points $w_{\ell+1}, \ldots, w_N \in \Lambda^0$ have period $k$ under $P^{\circ p}$, i.e., period $q$ under $P$.              
\end{proof}

\begin{lemma}\label{periodq}
Let $B=P^{\circ q}(B)$ be a parabolic basin of $P$ that meets $\sL$. Then $B$ has period $q$. In particular, every $\sL$-arc has period $q$. 
\end{lemma}

\begin{proof}
The result is clear if $B$ is a basin of some $w_j \in \sL \cap J$ with period $q$. Let us then assume there is a unique $w_{\ell}$ with period $p=q/k<q$ and $B$ is a basin of  $w_{\ell}$. In this case, by the proof of \lemref{period}, $B$ must be contained in one of the $k$ sectors of $\CC \sm (\Gamma^0 \cup \Gamma^1 \cup \ldots \cup \Gamma^{k-1})$, so the action of $P^{\circ p}$ on $B$ must have the same combinatorial rotation number $r/k$ as $w_\ell$. In particular, the period of $B$ under $P^{\circ p}$ must be $k$.    
\end{proof}

Here is a sharper statement about the period of $\sL$-arcs: 
  
\begin{lemma}\label{nosharing}
If $\gamma$ is an $\sL$-arc, none of the iterated images $P(\gamma), \ldots, P^{\circ q-1}(\gamma)$ can be an $\sL$-arc. 
\end{lemma}

\begin{proof}
The idea is similar to the one used to justify the simple observation $(\dagger)$. Suppose $\gamma'=P^{\circ j}(\gamma)$ is an $\sL$-arc for some $1 \leq j \leq q-1$. Take $u=\Pi(\langle s_n \rangle) \in \gamma$ and let $u'=P^{\circ j}(u) =\Pi(\langle s'_n \rangle) \in \gamma'$. Then  $P_n^{\circ j}(R_n(s_n)) \to u'$ and $R_n(s'_n) \to u'$. By \lemref{ct-easy}, the hyperbolic distance in the basin of infinity $\Om_n$ between $P_n^{\circ j}(R_n(s_n))$ and $R_n(s'_n)$ must tend to $0$ as $n \to \infty$. But this distance is the same as $\dist_{\Chat \sm \ov{\DD}}(\e^{d^j s_n + 2\pi \ii d^j \theta},\e^{s'_n+2\pi \ii \theta})$, which clearly tends to $\infty$ since $d^j \theta \neq \theta \ (\modd \ZZ)$. 
\end{proof}

\begin{corollary}[Intersections of the images of $\sL$]\label{i0}
Let $0\leq i<j \leq q-1$. \vs
\begin{enumerate}
\item[(i)]
If all $w_0, \ldots, w_N$ have period $q$, then $P^{\circ i}(\sL)$ and $P^{\circ j}(\sL)$ are disjoint. \vs
\item[(ii)]
If there is a unique $w_{\ell}$ with period $p<q$, then 
$P^{\circ i}(\sL)$ and $P^{\circ j}(\sL)$ are disjoint unless $i=j \ (\modd p)$ in which case $P^{\circ i}(\sL) \cap P^{\circ j}(\sL)$ is the single point $P^{\circ i}(w_{\ell})$. 
\end{enumerate}
\end{corollary}

\begin{proof}
We only need to treat case (ii) and rule out the possibility that $P^{\circ i}(\sL)$ and $P^{\circ j}(\sL)$ with $i=j \ (\modd p)$ might share a homoclinic arc $\gamma$ based at $P^{\circ i}(w_{\ell})$. But for any such $\gamma$ both $P^{\circ q-i}(\gamma)$ and $P^{\circ q-j}(\gamma)$ would be $\sL$-arcs, contrary to \lemref{nosharing}.   
\end{proof}

\subsection{$\sL$-arcs in the same cycle of basins}

Suppose $B=P^{\circ q}(B)$ is a parabolic basin of $w \in \sL \cap J$ that meets $\sL$. If $w$ has period $q$, \corref{i0} shows that none of the iterated images $P(B), \ldots, P^{\circ q-1}(B)$ can meet $\sL$. However, if $w$ has period $p<q$, then the union $P^{\circ p}(B) \cup \cdots \cup P^{\circ (q-p)}(B)$ can {\it a priori} meet $\sL$. Below we investigate this possibility in preparation for the proof of \thmref{C}. \vs     

\noindent
{\bf Standing assumptions.} For the remainder of this section up to the proof of \thmref{C}, we work under the following hypotheses: \vs

$\bullet$ $N \geq 1$ and there is a unique $0 \leq \ell \leq N-1$ for which $w_\ell$ has period $p=q/k<q$, so $w_\ell \neq w_N=\zeta_{\infty}$. Since the multiplier $(P^{\circ p})'(w_\ell)$ is a primitive $k$-th root of unity, there is a unique integer $1 \leq j \leq k-1$ for which $(P^{\circ jp})'(w_\ell)=\e^{2\pi \ii/k}$. We set 
\begin{alignat*}{3}
Q & :=P^{\circ jp}, & \qquad R^i & :=Q^{\circ i}(R), & \qquad \sL^i & :=Q^{\circ i}(\sL), \\
Q_n & :=P_n^{\circ jp}, & \qquad R^i_n &:=Q_n^{\circ i}(R_n). & & 
\end{alignat*}

$\bullet$ There is a cycle $B, Q(B), \ldots, Q^{\circ k-1}(B)$ of parabolic basins at $w_\ell$ containing at least one homoclinic in $\sL$. \vs

Sort all the homoclinic $\sL$-arcs in the cycle $B, Q(B), \ldots, Q^{\circ k-1}(B)$ as $\eta_1<\cdots<\eta_h$. Note that by \thmref{B} and the assumption $\ell \neq N$, each $\eta_j$ is the sole homoclinic arc in its earring. By \corref{homuni}, $\eta_1, \ldots, \eta_h$ belong to different parabolic basins in this cycle and in particular $1 \leq h \leq k$. Let $\eta_0$ be the unique heteroclinic in $\sL$ such that $w^+(\eta_0)=w_\ell$, and $\eta_{h+1}$ be the unique heteroclinic in $\sL$ such that $w^-(\eta_{h+1})=w_\ell$, or $\eta_{h+1}=R$ if $\ell=0$. Then $\eta_0<\eta_1<\cdots<\eta_h<\eta_{h+1}$. This puts us in the  situation of \thmref{unlinked}: If $D$ is a sufficiently small good disk centered at $w_\ell$ for the collection $\{ \eta_0, \ldots, \eta_{h+1} \}$, and if  
$$
v_j := w^+(\eta_j,D) \quad \text{and} \quad u_{j+1} := w^-(\eta_{j+1},D) \qquad (0 \leq j \leq h),
$$
then the set of pairs 
$$
\Theta^0 := \{ (v_0,u_1), \ (v_1,u_2), \ \ldots, \ (v_h,u_{h+1}) \}
$$ 
on $\bd D$ is unlinked. More generally, for each $0\leq i \leq k-1$ we can consider the $\sL^i$-arcs $\eta_j^i:=Q^{\circ i}(\eta_j)$ in the same cycle $B, Q(B), \ldots, Q^{\circ k-1}(B)$, and the points
$$
v^i_j := w^+(\eta^i_j,D) \quad \text{and} \quad u^i_{j+1} := w^-(\eta^i_{j+1},D) \qquad (0\leq j \leq h)
$$ 
(we may arrange the same $D$ to be a good disk for $\{ \eta_0^i, \ldots, \eta_{h+1}^i \}$ for every $i$). We then form the set
$$
\Theta^i := \{ (v^i_0,u^i_1), \ (v^i_1,u^i_2), \ \ldots, \ (v^i_h,u^i_{h+1}) \}
$$
of $h+1$ unlinked pairs. Observe that since $w_\ell$ is a fixed point of $Q$ with multiplier $\rho:=Q'(w_\ell)=\e^{2\pi \ii/k}$, each $\Theta^{i+1}$ is approximately the rotation $\rho \Theta^i$ with an error of the order of $o(\ve)$, where $\ve>0$ is the radius of $D$. Note also that by \lemref{nosharing} all the arcs $\eta^i_j$ are disjoint, hence all the points $u^i_j, v^i_j$ are distinct. \vs

The following is a generalization of \thmref{unlinked}:

\begin{theorem}\label{unlinked-k}
The union $\Theta^0 \cup \Theta^1 \cup \cdots \cup \Theta^{k-1}$ is unlinked. 
\end{theorem}

\begin{proof}
It suffices to prove that for every $1 \leq i \leq k-1$ the union $\Theta^0 \cup \Theta^i$ is unlinked. The argument is similar \thmref{unlinked}, so we will be brief on the identical details. To ease the notation a bit, we will denote all the objects corresponding to $\sL=\sL^0$ without a superscript $0$ and those corresponding to $\sL^i$ with a superscript $*$. Choose points $z_j \in \eta_j$ and $z^*_j \in \eta^*_j$ for $0 \leq j \leq h+1$, with $z_{h+1}<R(1)$ and $z^*_{h+1}<R^*(1)$. We can find pairwise disjoint good transversals $\Sigma_1, \ldots, \Sigma_r$ for $\sL \cup \sL^*$ such that 
\begin{equation}\label{sig}
\bigcup_{j=1}^r \Sigma_j \cap (\sL_{\leq 1} \cup \sL^*_{\leq 1}) = \{ z_0, \ldots, z_{h+1} , z^*_0, \ldots, z^*_{h+1} \}. 
\end{equation}
In fact, the union $\hat{\sL} \cup \hat{\sL}^*$ is connected and full, so in the construction of good transversals in the proof of \lemref{egt} we can use the ``external rays'' of the uniformization $(\Chat \sm \ov{\DD},\infty) \oset[-0.2ex]{\cong}{\longrightarrow} (\Chat \sm (\hat{\sL} \cup \hat{\sL}^*), \infty)$. Choosing $D$ sufficiently small guarantees that these transversals are disjoint from $\ov{D}$ and that the relations  
\begin{equation}\label{zz*}
\begin{aligned}
& z_0<v_0<u_1<z_1<v_1<\cdots<u_h<z_h<v_h<u_{h+1}<z_{h+1}<R(1) \vs \\
& z^*_0<v^*_0<u^*_1<z^*_1<v^*_1<\cdots<u^*_h<z^*_h<v^*_h<u^*_{h+1}<z^*_{h+1}<R^*(1)
\end{aligned}
\end{equation}
hold. By transversality, we can find the approximating sequences
\begin{align*}
v_{n,j} & := R_n(t_{n,j}) \to v_j &  u_{n,j+1} & := R_n(s_{n,j+1}) \to u_{j+1} \\
v^*_{n,j} & := R^*_n(t^*_{n,j}) \to v^*_j &  u^*_{n,j+1} & := R^*_n(s^*_{n,j+1}) \to u^*_{j+1}  
\end{align*}
on $\bd D$ for $0 \leq j \leq h$. Choose $u_0:=\Pi(\langle s_{n,0} \rangle) \in \eta_0$ and $v_{h+1}:=\Pi(\langle t_{n,h+1} \rangle) \in \eta_{h+1}$ such that $u_0<z_0$ and $z_{h+1}<v_{h+1}<R(1)$, and set $u_{n,0}:=R_n(s_{n,0}), v_{n,h+1}:=R_n(t_{n,h+1})$. Similarly, choose $u^*_0:=\Pi(\langle s^*_{n,0} \rangle) \in \eta^*_0$ and $v^*_{h+1}:=\Pi(\langle t^*_{n,h+1} \rangle) \in \eta^*_{h+1}$ such that $u^*_0<z^*_0$ and $z^*_{h+1}<v^*_{h+1}<R^*(1)$, and set $u^*_{n,0}:=R^*_n(s^*_{n,0}), v^*_{n,h+1}:=R^*_n(t^*_{n,h+1})$. \vs

For $0 \leq j \leq h+1$ let $\tilde{\eta}_j$ denote the subarc of $\eta_j$ that joins $u_j$ to $v_j$. Define the subarc $\tilde{\eta}^*_j$ of $\eta^*_j$ analogously. The closed set 
$$
E := \bigcup_{j=1}^r \Sigma_j \cup \bigcup_{j=0}^{h+1} \, \big( \tilde{\eta}_j \cup \tilde{\eta}^*_j \big)
$$
has simply connected complement in $\CC$. For each $0 \leq j \leq h+1$ we have $R_n([s_{n,j},t_{n,j}]) \to \tilde{\eta}_j$ and $R^*_n([s^*_{n,j},t^*_{n,j}]) \to \tilde{\eta}^*_j$ in $C^\infty$-topology as $n \to \infty$. Thus, for large $n$ the closed set 
$$
E_n := \bigcup_{j=1}^r \Sigma_j \cup \bigcup_{j=0}^{h+1} \, \big( R_n([s_{n,j},t_{n,j}]) \cup R^*_n([s^*_{n,j},t^*_{n,j}]) \big)
$$ 
is $C^\infty$-close to $E$. Since all the intersections in $E$ are transversal, the complement $\CC \sm E_n$ must also be simply connected. \vs

Now suppose there is a pair $(v_a,u_{a+1}) \in \Theta$ that is linked with a pair $(v^*_b,u^*_{b+1}) \in \Theta^*$. Then for all large $n$ the pairs $(v_{n,a},u_{n,a+1})$ and $(v^*_{n,b},u^*_{n,b+1})$ are linked as well. It follows that at least one of the open ray segments $R_n (]t_{n,a},s_{n,a+1}[)$ or $R^*_n (]t^*_{n,b},s^*_{n,i+b}[)$ must meet $E_n$, for otherwise by \lemref{Usim} these ray segments would have to intersect, which is impossible since $R_n$ and $R^*_n$ are disjoint. We conclude that either $R_n (]t_{n,a},s_{n,a+1}[)$ or $R^*_n (]t^*_{n,b},s^*_{n,b+1}[)$ must intersect $\bigcup_{j=1}^r \Sigma_j$ for infinitely many values of $n$. In the former case we obtain an accumulation point $\tilde{z} \in \bigcup_{j=1}^r \Sigma_j \cap \sL_{\leq 1}$ of a sequence $\{ R_n(\lambda_n) \}$, where $t_{n,a}<\lambda_n<s_{n,a+1}$. By \eqref{sig}, $\tilde{z} \in \{ z_0, \ldots, z_{h+1} \}$ while by \lemref{EE}, $v_a \leq \tilde{z} \leq u_{a+1}$. This contradicts \eqref{zz*}. In the latter case we obtain an accumulation point $\tilde{z} \in \bigcup_{j=1}^r \Sigma_j \cap \sL^*_{\leq 1}$ of a sequence $\{ R^*_n(\lambda_n) \}$, where $t^*_{n,b}<\lambda_n<s^*_{n,b+1}$. By \eqref{sig}, $\tilde{z} \in \{ z^*_0, \ldots, z^*_{h+1} \}$ while by \lemref{EE}, $v^*_b \leq \tilde{z} \leq u^*_{b+1}$. This, again, contradicts \eqref{zz*}.
\end{proof}

Recall that $\ve>0$ is the radius of the good disk $D$ centered at $w_\ell$ used to define the sets $\Theta^0, \ldots, \Theta^{k-1}$. It will be convenient to represent $\bd D$ in the additive model of the unit circle by identifying $w_\ell+\ve \e^{2\pi \ii t} \in \bd D$ with $t \in \TT:=\RR/\ZZ$. This identification involves rescaling by a factor of $1/\ve$, so it turns every $o(\ve)$ estimate on $\bd D$ to an $o(1)$ estimate in the additive model $\TT$ as $\ve \to 0$. To simplify the formulas that will follow, we write 
$$
x \aeq y \qquad \text{whenever}  \qquad x=y+o(1).
$$ 

By the {\bit distance} $\de(a,b)$ between $a,b \in \TT$ we mean the normalized Lebesgue measure of the shorter arc of $\TT$ between $a,b$. Choosing suitable representatives, it is clear that $\de(a,b)=|a-b| \leq 1/2$. \vs 

Let $\nu \geq 1$ be the degeneracy order of $w_\ell$ as a fixed point of $Q$. There are $\nu k$ attracting and $\nu k$ repelling directions of $w_\ell$ which intersect $\bd D \cong \TT$ at $2\nu k$ equally spaced alternating points that we mark as $\oplus$ for attracting and $\ominus$ for repelling. Thus, every $v^i_j$ is $o(1)$-close to a $\oplus$ and every $u^i_j$ is $o(1)$-close to a $\ominus$. This yields
\begin{align}
\delta(u^i_j,v^i_j) & \aeq \frac{1}{2\nu k} \label{L1} \\
\delta(v^i_j,u^i_{j+1}) & \aeq \text{an odd multiple of}  \ \frac{1}{2\nu k}. \label{L2}
\end{align}
Recalling that $\Theta^{i+1}$ is $o(1)$-close to $\rho \Theta^i$, where $\rho=Q'(w_\ell)=\e^{2\pi \ii/k}$, we also have
\begin{equation}\label{L5}
u^{i+1}_j \aeq u^i_j + \frac{1}{k} \qquad \text{and} \qquad 
v^{i+1}_j \aeq v^i_j + \frac{1}{k}. 
\end{equation}

\begin{lemma}\label{l<}
For any $0 \leq j \leq h$ the pair $(v^i_j,u^i_{j+1}) \in \Theta^i$ satisfies 
\begin{equation}\label{L3}
\delta(v^i_j,u^i_{j+1}) \leq \frac{1}{k} - \frac{1}{2\nu k} + o(1).
\end{equation} 
If $1 \leq j \leq h-1$, or if $j=0$ and $\eta^i_0$ belongs to the cycle $B, Q(B), \ldots, Q^{\circ k-1}(B)$, then
\begin{equation}\label{L4}
\delta(v^i_j,u^i_{j+1}) \aeq \frac{1}{k} - \frac{1}{2\nu k}.
\end{equation}
\end{lemma}  

\begin{proof}
The proof of \eqref{L3} is based on the simple observation that a pair of distance $>1/k$ and its rotated image by $1/k$ of a turn must be linked. There is nothing to prove if $k=2$, so suppose \eqref{L3} is false and $k \geq 3$. Then, by \eqref{L2}, we would have the lower bound $\delta(v^i_j,u^i_{j+1}) \geq 1/k + 1/(2\nu k) + o(1)$. In view of \eqref{L5}, this lower bound would imply that $(v^i_j,u^i_{j+1})$ and $(v^{i+1}_j,u^{i+1}_{j+1})$ are linked, contradicting \thmref{unlinked-k}. \vs

To prove \eqref{L4}, recall that the attracting directions corresponding to the cycle $B, Q(B), \ldots, Q^{\circ k-1}(B)$ form $k$ equally spaced $\oplus$ points on $\TT$. If $1 \leq j \leq h-1$, or if $j=0$ and $\eta^i_0$ belongs to the cycle $B, Q(B), \ldots, Q^{\circ k-1}(B)$, then $\eta^i_j$ and $\eta^i_{j+1}$ belong to different basins in this cycle, so $\delta(v^i_j,v^i_{j+1}) \geq 1/k +o(1)$. In view of \eqref{L1}, this gives the lower bound $\delta(v^i_j,u^i_{j+1}) \geq 1/k - 1/(2\nu k)+o(1)$. Combining this with the upper bound \eqref{L3}, we obtain \eqref{L4}. 
\end{proof}

The following is the main technical result of this section: 

\begin{theorem}\label{only1}
The cycle $B, Q(B), \ldots, Q^{\circ k-1}(B)$ meets only one $\sL$-arc, which is necessarily the homoclinic $\eta_1$ based at $w_\ell$. 
\end{theorem}

Thus, the sequence of $\sL$-arcs used to define $\Theta^0$ reduces to $\eta_0<\eta_1<\eta_2$ (so $h=1$), and the heteroclinic $\eta_0$ is not in the cycle $B, Q(B), \ldots, Q^{\circ k-1}(B)$. As a result, each of the $k$ basins in this cycle contains the homoclinic $\eta^i_1$ for a unique $0 \leq i \leq k-1$. In particular, another cycle of basins at $w_\ell$ is needed to accommodate the heteroclinic arcs $\eta_0, \eta^1_0, \ldots, \eta^{k-1}_0$, so the degeneracy order $\nu$ of $w_\ell$ must be at least $2$. 

\begin{proof}
We know that the cycle $B, Q(B), \ldots, Q^{\circ k-1}(B)$ contains the homoclinic $\eta_1$. But there is a possibility that this cycle contains the heteroclinic $\eta_0$ or a second homoclinic $\eta_2$. Below we rule out these scenarios: \vs

{\it Case 1.} Suppose the cycle $B, Q(B), \ldots, Q^{\circ k-1}(B)$ contains the heteroclinic $\eta_0$. Without loss of generality assume $\eta_0 \subset B$. By \corref{homuni} $\eta_1$ cannot be in $B$, so $\eta_1 \subset Q^{\circ i}(B)$ for some $1 \leq i \leq k-1$. It follows that $\eta_1$ is in the same parabolic basin as the heteroclinic $\eta^i_0$, so $\de(v_1,v^i_0) \aeq 0$ and therefore $\de(u_1,v^i_0) \aeq 1/(2\nu k)$ by \eqref{L1}. Since $\delta(v_0,u_1) \aeq 1/k-1/(2\nu k)$ by \eqref{L4}, it follows that $\de(v_0,v^i_0) \leq 1/k + o(1)$. On the other hand, \eqref{L5} shows that up to an $o(1)$ error the distance $\de(v_0,v^i_0)$ is a multiple of $1/k$. In fact, $\de(v_0,v^i_0) \aeq i/k$ if $i \leq k/2$ and $\de(v_0,v^i_0) \aeq (k-i)/k$ if $i > k/2$. It follows that $i=1$ or $i=k-1$. Without loss of generality assume that $i=1$ (the other case is completely similar). Then we have the following points in positive cyclic order on $\bd \DD \cong \TT$:  
$$
v_0 < u_1 \aeq v_0+\frac{1}{k}-\frac{1}{2\nu k} < v_1 \aeq v_0+\frac{1}{k} < v^1_0 \aeq v_0+\frac{1}{k}. 
$$
Using \eqref{L5}, we obtain a similar order for all $0 \leq i \leq k-1$:

\begin{figure}[]
	\centering
	\begin{overpic}[width=0.6\textwidth]{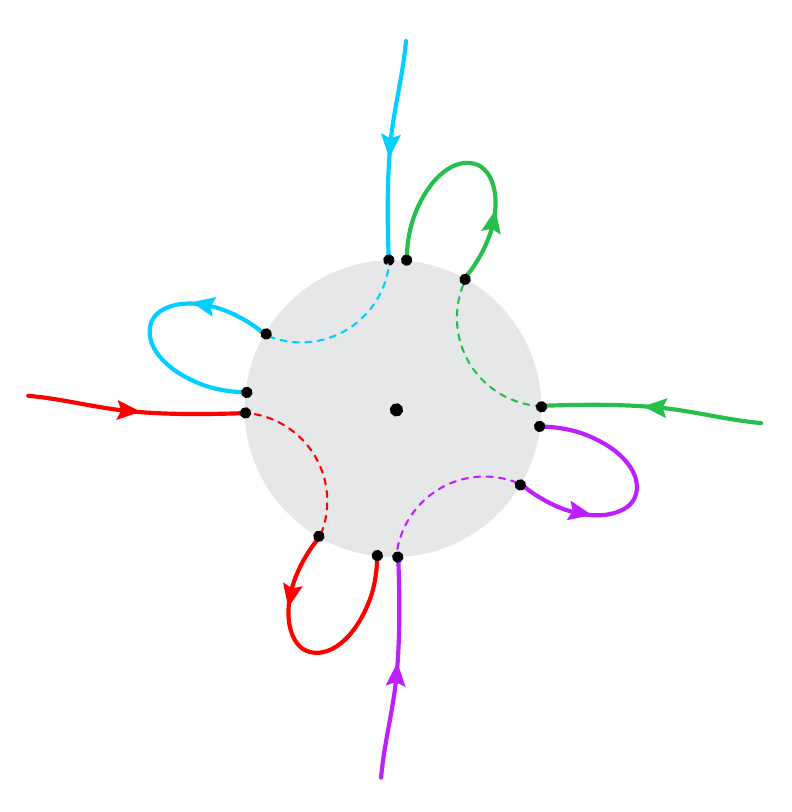}
\put (78,52.5) {\small{\color{mygreen}{$\eta^2_0$}}}
\put (62,75) {\small{\color{mygreen}{$\eta^2_1$}}}
\put (42,81) {\small{\color{myblue}{$\eta^3_0$}}}
\put (22,65.4) {\small{\color{myblue}{$\eta^3_1$}}}
\put (14,45) {\small{\color{red}{$\eta_0$}}}
\put (30.5,25) {\small{\color{red}{$\eta_1$}}}
\put (51,16) {\small{\color{mypurple}{$\eta^1_0$}}}
\put (72,32) {\small{\color{mypurple}{$\eta^1_1$}}}
\put (65.7,51.7) {\footnotesize $v^2_0$}
\put (58.2,64.2) {\footnotesize $u^2_1$}
\put (49.2,64.8) {\footnotesize $v^2_1$}
\put (43.8,67.5) {\footnotesize $v^3_0$}
\put (33,60.2) {\footnotesize $u^3_1$}
\put (31.8,51) {\footnotesize $v^3_1$}
\put (29,45.3) {\footnotesize $v_0$}
\put (35,34.3) {\footnotesize $u_1$}
\put (45.2,33.2) {\footnotesize $v_1$}
\put (50.6,31) {\footnotesize $v^1_0$}
\put (60.5,36.2) {\footnotesize $u^1_1$}
\put (62.4,46) {\footnotesize $v^1_1$}
\put (46.5,46) {\footnotesize $w_\ell$}
	\end{overpic}
\caption{\footnotesize Illustration of {\it Case 1} in the proof of \thmref{only1}, with $k=4$.}
\label{flower}
\end{figure}

\begin{equation}\label{L6}
v^i_0 < u^i_1 \aeq v^i_0+\frac{1}{k}-\frac{1}{2\nu k} < v^i_1 \aeq v^i_0+\frac{1}{k} < v^{i+1}_0 \aeq v^i_0+\frac{1}{k}
\end{equation}
(see \figref{flower}). Now consider the next pair $(v_1,u_2) \in \Theta^0$. By \eqref{L2} and \eqref{L3}, we have $u_2 \aeq v_1 - (2j+1)/(2\nu k)$ or $u_2 \aeq v_1 +(2j+1)/(2\nu k)$ for some $0 \leq j \leq \nu-1$. In the first case \eqref{L6} shows that $(v_1,u_2)$ and $(v_0,u_1)$ would be linked, contradicting \thmref{unlinked-k}. In the second case \eqref{L6} shows that $(v_1,u_2)$ and $(v^1_0,u^1_1)$ would be linked unless $j=2\nu-1$. This leaves only one possibility for $u_2$:
$$
u_2 \aeq v_1 + \frac{1}{k} -\frac{1}{2\nu k} \aeq v^1_0 + \frac{1}{k} -\frac{1}{2\nu k} \aeq u^1_1. 
$$
Observe that since $(v_1,u_2)$ and $(v^1_0,u^1_1)$ are unlinked by \thmref{unlinked-k}, $u_2$ belongs to the interval $I_{\eta^1_1}$ bounded by $u^1_1$ and $v^1_1$, so $\eta_2$ is contained in $\Delta_{\eta^1_1}$. In particular, $\eta_2$ must be a homoclinic, with $v_2 \in I_{\eta^1_1}$ and $v_2 \aeq v^1_1 \aeq v^2_0$. \vs

Now repeat the above argument with the next pair $(v_2,u_3) \in \Theta^0$ to conclude that the only possibility is $u_3 \aeq u^2_1, v_3 \aeq v^2_1$, with $\eta_3$ contained in $\Delta_{\eta^2_1}$, so $\eta_3$ must be a homoclinic. Continuing this process inductively, we finally reach $\eta_{h+1}$ which by the same argument must be contained in $\Delta_{\eta^h_1}$. This is a contradiction since $\eta_{h+1}$ is a heteroclinic. \vs

{\it Case 2.} Suppose the cycle $B, Q(B), \ldots, Q^{\circ k-1}(B)$ does not contain the heteroclinic $\eta_0$ but contains $\eta_1$ and a next homoclinic $\eta_2$. By \eqref{L2} we now have $\de(v_0,u_1) \aeq (2j+1)/(2\nu k)$ for some $0 \leq j \leq \nu-2$ (in particular $\nu \geq 2$). Without loss of generality, assume $u_1 \aeq v_0 + (2j+1)/(2\nu k)$. By \eqref{L1}, either $v_1 \aeq u_1 + 1/(2\nu k)$ or $v_1 \aeq u_1 - 1/(2\nu k)$. The latter is impossible, because it implies $v_0<v_1<u_1$ which would force $(v_0,u_1)$ and $(v_1,u_2)$ be linked since $\de(v_1,u_2) \aeq 1/k - 1/(2\nu k)$ is greater than $\de(v_0,u_1)$ by at least $1/(2\nu k)+o(1)$. Thus, we must have the following points in positive cyclic order:
$$
v_0 < u_1 \aeq v_0 + \frac{2j+1}{2\nu k} < v_1 \aeq v_0 + \frac{2j+2}{2\nu k} < v^1_0 \aeq v_0 + \frac{1}{k}.
$$
Using \eqref{L5}, we obtain a similar order for all $0 \leq i \leq k-1$:
$$
v^i_0 < u^i_1 \aeq v^i_0 + \frac{2j+1}{2\nu k} < v^i_1 \aeq v^i_0 + \frac{2j+2}{2\nu k} < v^{i+1}_0 \aeq v^i_0 + \frac{1}{k}.
$$
Now consider the next pair $(v_1,u_2) \in \Theta^0$. Since $\de(v_1,u_2) \aeq 1/k-1/(2\nu k)$ by \eqref{L4}, there are only two possibilities $u_2 \aeq v_1 \pm (1/k-1/(2\nu k))$. Let us address them separately: \vs 

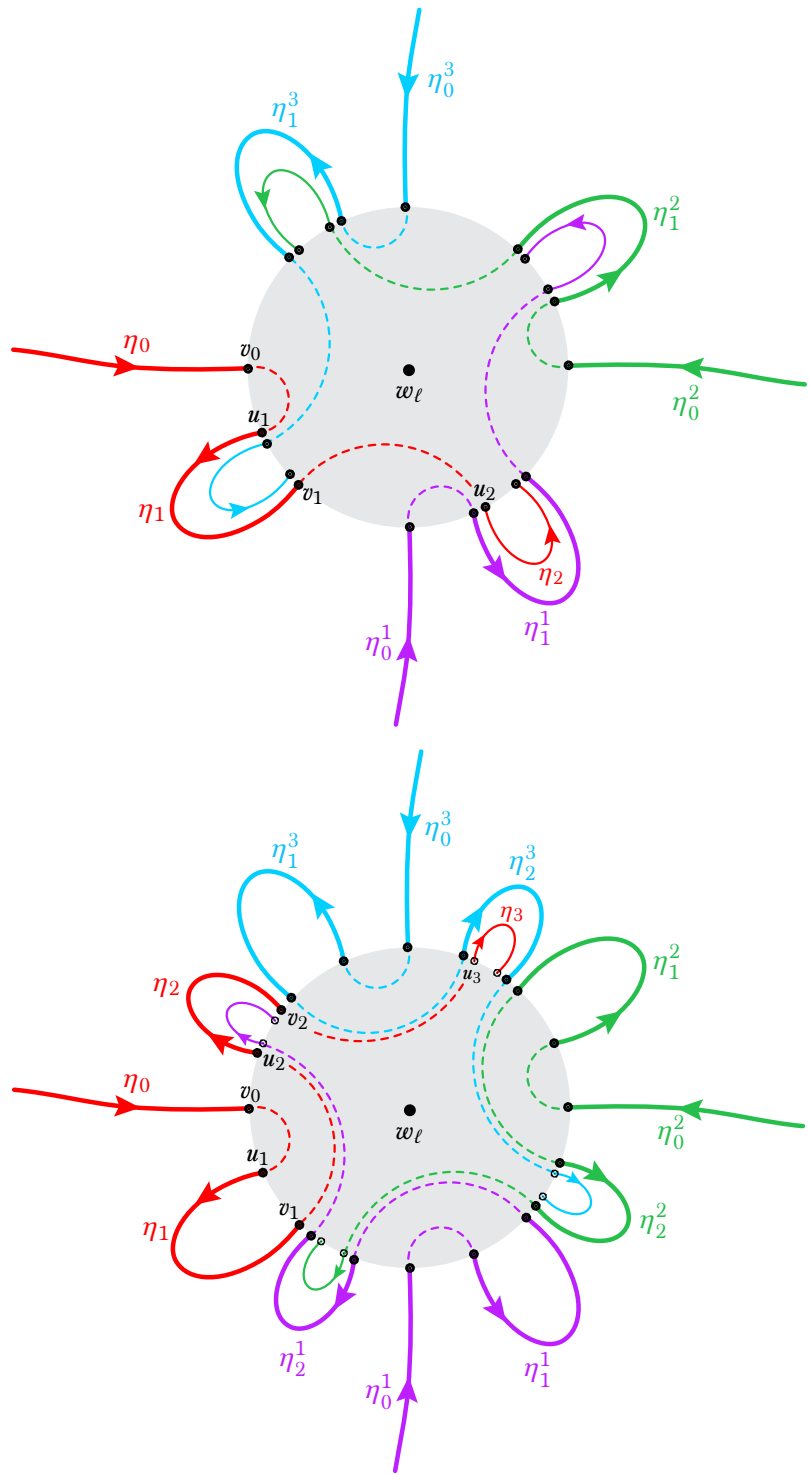
\begin{figure}[]
	\centering
	\begin{overpic}[width=0.7\textwidth]{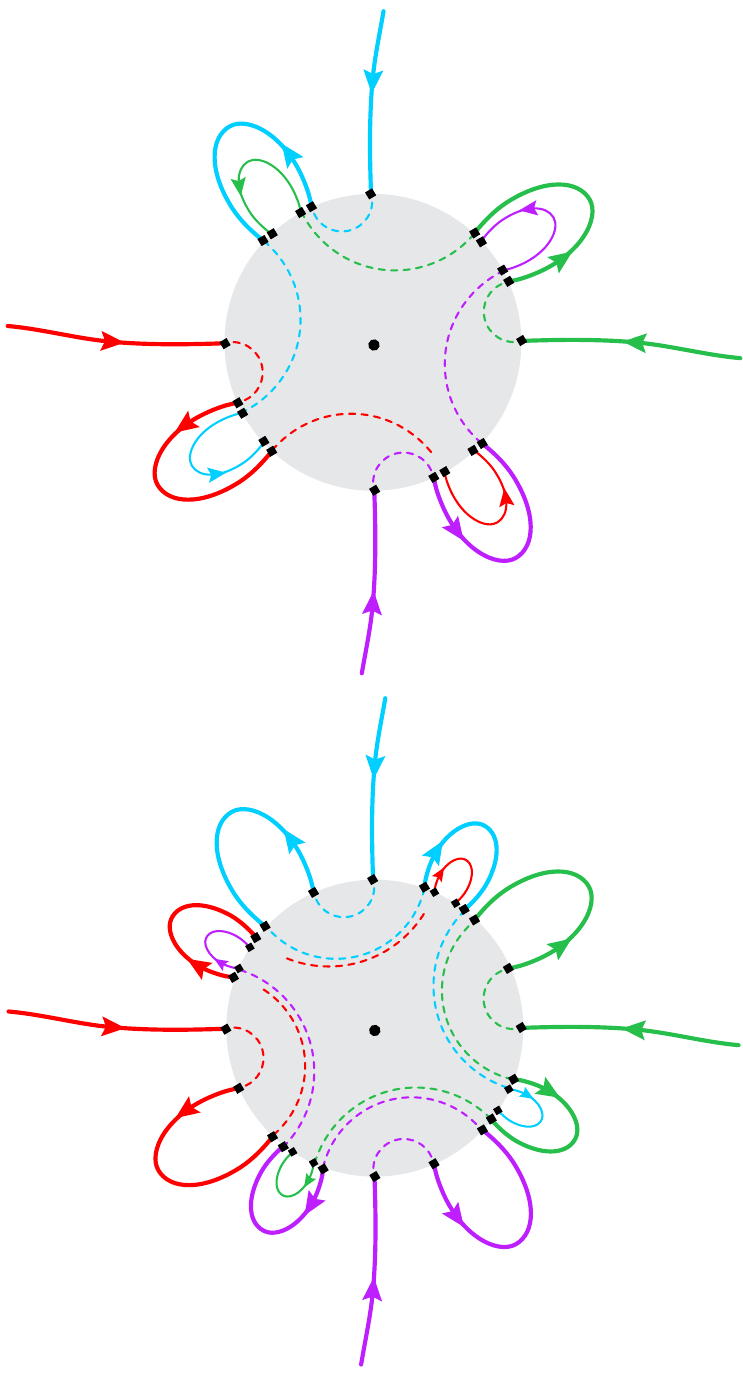}
\put (45,72.2) {\small{\color{mygreen}{$\eta^2_0$}}}
\put (43.7,85) {\small{\color{mygreen}{$\eta^2_1$}}}
\put (28.5,94) {\small{\color{myblue}{$\eta^3_0$}}}
\put (18,92) {\small{\color{myblue}{$\eta^3_1$}}}
\put (8,76.4) {\small{\color{red}{$\eta_0$}}}
\put (9,65) {\small{\color{red}{$\eta_1$}}}
\put (36,60.7) {\footnotesize{\color{red}{$\eta_2$}}}
\put (33.4,38) {\scriptsize{\color{red}{$\eta_3$}}}
\put (24.3,56) {\small{\color{mypurple}{$\eta^1_0$}}}
\put (35,57) {\small{\color{mypurple}{$\eta^1_1$}}}
\put (15.8,75.8) {\footnotesize $v_0$}
\put (16.3,71.4) {\footnotesize $u_1$}
\put (20,66.1) {\footnotesize $v_1$}
\put (31.5,66.4) {\footnotesize $u_2$}
\put (44,23) {\small{\color{mygreen}{$\eta^2_0$}}}
\put (43.7,34.5) {\small{\color{mygreen}{$\eta^2_1$}}}
\put (42.8,17) {\small{\color{mygreen}{$\eta^2_2$}}}
\put (28.3,43.5) {\small{\color{myblue}{$\eta^3_0$}}}
\put (18,42) {\small{\color{myblue}{$\eta^3_1$}}}
\put (34,41.2) {\small{\color{myblue}{$\eta^3_2$}}}
\put (8,26.5) {\small{\color{red}{$\eta_0$}}}
\put (9.3,16.5) {\small{\color{red}{$\eta_1$}}}
\put (10,33) {\small{\color{red}{$\eta_2$}}}
\put (24.3,5.5) {\small{\color{mypurple}{$\eta^1_0$}}}
\put (35,7) {\small{\color{mypurple}{$\eta^1_1$}}}
\put (18.5,7.7) {\small{\color{mypurple}{$\eta^1_2$}}}
\put (15.8,25.8) {\footnotesize $v_0$}
\put (16.2,21.5) {\footnotesize $u_1$}
\put (18.4,17.9) {\footnotesize $v_1$}
\put (17.3,27.9) {\footnotesize $u_2$}
\put (19,30.7) {\footnotesize $v_2$}
\put (30.8,33.8) {\tiny $u_3$}
\put (26.3,73) {\footnotesize $w_\ell$}
\put (26.3,23) {\footnotesize $w_\ell$}
	\end{overpic}
\caption{\footnotesize Illustration of {\it Case 2a} (top) and {\it Case 2b} (bottom) in the proof of \thmref{only1}, with $k=4$.}
\label{patterns}
\end{figure}

{\it Case 2a.} $u_2 \aeq v_1 + 1/k-1/(2\nu k)$. Then, by \eqref{L5} and \eqref{L1}, 
$$
u_2 \aeq v^1_1 - \frac{1}{2\nu k} \aeq u^1_1. 
$$
Since $(v_1,u_2)$ and $(v^1_0,u^1_1)$ are unlinked, $u_2$ belongs to the interval $I_{\eta^1_1}$ bounded by $u^1_1$ and $v^1_1$, so $\eta_2$ is contained in $\Delta_{\eta^1_1}$ (see \figref{patterns} top). Now repeat the argument with the next pair $(v_2,u_3) \in \Theta^0$ to conclude that the only possibility is $u_3 \aeq u^2_1, v_3 \aeq v^2_1$, with $\eta_3$ contained in $\Delta_{\eta^2_1}$, and therefore $\eta_3$ must be a homoclinic. Continuing this process inductively as in {\it Case 1}, we eventually arrive at the conclusion that $\eta_{h+1}$ is contained in $\Delta_{\eta^h_1}$, which is a contradiction since $\eta_{h+1}$ is a heteroclinic. \vs 

{\it Case 2b.} $u_2 \aeq v_1 - 1/k + 1/(2\nu k)$. Then, by \eqref{L5} and \eqref{L1}, 
$$
u_2 \aeq v^{k-1}_1 + \frac{1}{2\nu k}. 
$$
This means $\eta_2$ and $\eta^{k-1}_1$ are in the same basin but neither is inside the other (see \figref{patterns} bottom). Now repeat the argument with the next pair $(v_2,u_3) \in \Theta^0$ to conclude that the only possibility is 
$$
u_3 \aeq v_2 - \frac{1}{k} + \frac{1}{2\nu k} \aeq v^{k-1}_2 + \frac{1}{2\nu k} \aeq u^{k-1}_2.
$$ 
Moreover, since $(v_2,u_3)$ and $(v^{k-1}_1,u^{k-1}_2)$ are unlinked, $u_3$ belongs to the interval $I_{\eta^{k-1}_2}$ bounded by $u^{k-1}_2$ and $v^{k-1}_2$, so $\eta_3$ is contained in $\Delta_{\eta^{k-1}_2}$. In particular, $\eta_3$ must be a homoclinic. Continuing this process inductively, we finally reach $\eta_{h+1}$ which by the same argument must be contained in $\Delta_{\eta^{k-h+1}_2}$. This is a contradiction since $\eta_{h+1}$ is a heteroclinic.  
\end{proof} 

\begin{proof}[Proof of \thmref{C}]
Let us first treat the easier case $N=0$ where there are no heteroclinic arcs. If $w_0=w_N=\zeta_\infty$ is repelling, then $M=M^\#=0$ and there is nothing to prove. Otherwise $w_0$ is parabolic and $M^\#$ is at most the number of cycles of parabolic basins at $w_0$. By classical Fatou-Julia theory, every cycle of parabolic basins contains at least one critical point. Hence $M^\# \leq d-1$, proving \eqref{BB1}. To bound $M$, suppose $p=q/k$ is the period of $w_0$ and $\nu$ is the degeneracy order of $w_0$ as a fixed point of $P^{\circ p}$. Then there are $\nu$ cycles of parabolic basins at $w_0$, each of length $k$. Since $\nu \leq d-1$, we obtain $M \leq k \nu \leq d-1 + (k-1)\nu$, which proves \eqref{BB2}. \vs  

Now suppose $N \geq 1$. For $0 \leq j \leq N$, define 
\begin{align*}
M_j & := \text{number of earrings in} \ \sL \ \text{based at} \ w_j \\
M^\#_j & := \text{number of equivalence classes of earrings in} \ \sL \ \text{based at} \ w_j,
\end{align*}
so $M=\sum_{j=0}^N M_j$ and $M^\#=\sum_{j=0}^N M^\#_j$. Let $p_j=q/k_j$ be the period of $w_j$ (we know from \lemref{period} that $p_j=q \Leftrightarrow k_j=1$ for all $j$ with at most one exception). Let $\nu_j$ be the degeneracy order of $w_j$ as a fixed point of $P^{\circ p_j}$, with the convention that $\nu_N=0$ if $w_N$ is repelling. Then there are $\nu_j$ cycles of parabolic basins at $w_j$, each of length $k_j$. Let $B=P^{\circ q}(B)$ be a parabolic basin at $w_j$. By \thmref{only1}, if $0 \leq j \leq N-1$, the cycle $B, P(B), \ldots, P^{\circ q-1}(B)$ contains at most one heteroclinic or one earring in $\sL$, but not both. Moreover, if this cycle contains a heteroclinic, the Basic Structure Lemma shows that there are at least two critical points of $P$ in the union $B \cup P(B) \cup \cdots \cup P^{\circ q-1}(B)$. This shows  \vs
\begin{align*}
M_j \leq & \ \nu_j-1 & M^\#_j & = \ M_j \qquad \text{if} \ 0 \leq j \leq N-1, \\
M_N \leq & \ k_N \nu_N & M^\#_N & \leq \, \nu_N, 
\end{align*}
and
$$
\sum_{j=0}^{N-1} (\nu_j+1) + \nu_N \leq d-1, \qquad \text{so} \qquad \sum_{j=0}^{N-1} (\nu_j-1) + \nu_N \leq d-1-2N.
$$
It follows that 
$$
2N+M^\# = 2N + \sum_{j=0}^{N-1} M^\#_j + M^\#_N \leq 2N + \sum_{j=0}^{N-1} (\nu_j-1) + \nu_N \leq d-1, 
$$
which proves \eqref{BB1}. Similarly, 
$$
2N+M = 2N + \sum_{j=0}^{N-1} M_j + M_N \leq 2N + \sum_{j=0}^{N-1} (\nu_j-1) + k_N \nu_N \leq d-1 + (k_N-1) \nu_N, 
$$
which proves \eqref{BB2}.
\end{proof}

\section{Proof of \thmref{D}}\label{ex} 

In this section we prove \thmref{D}, that is, we construct real monic polynomials $P$ of any odd degree $\geq 3$ which have the maximum number of heteroclinics allowed by \thmref{C}. \vs     

First assume we have constructed a real monic polynomial $P$ of degree $d=2N+1$ with fixed points $w_N=0<\cdots<w_1<w_0$ such that $0$ is repelling and $w_j$ is parabolic with multiplier $1$ and $\resit(P,w_j)<0$ for every $0 \leq j \leq N-1$. Since $P$ has $2N+1$ fixed points in $\CC$ counting multiplicities, the fixed point set of $P$ is $\{ w_0, w_1, \ldots, w_N \}$ with $w_0, \ldots, w_{N-1}$ having multiplicity $2$. Using the fact that $P$ is monic, we can write  
$$
P_\ve(z)=P(z)+\ve=\ve+z+z(z-w_0)^2\cdots(z-w_{N-1})^2, 
$$ 
which shows $P_\ve(x) \geq \ve+x$ for $x>0$. It follows that the unique repelling fixed point $w_N(\ve)$ near $0$ is negative, while the two simple fixed points of $P_\ve$ bifurcating off of $w_j$ are non-real, hence form a complex-conjugate pair $w_j(\ve),\ov{w_j(\ve)}$ with multipliers $\lambda_j(\ve), \ov{\lambda_j(\ve)}$. We have 
\begin{align*}
\resit(P_\ve,w_j(\ve))+\resit(P_\ve,\ov{w_j(\ve)}) 
& = \frac{1}{2} - \frac{1}{1-\lambda_j(\ve)} + \frac{1}{2} - \frac{1}{1-\ov{\lambda_j(\ve)}} \\
& = 1 - 2 \myre \left( \frac{1}{1-\lambda_j(\ve)} \right).
\end{align*}      
As $\ve \to 0$, this quantity must converge to $\resit(P,w_j)$, which is negative by the assumption. Hence $\myre(1/(1-\lambda_j(\ve)))>1/2$ or $|\lambda_j(\ve)|<1$ for all sufficiently small $\ve>0$ (it is easy to see that $\lambda_j(\ve)$ must tend to $1$ asymptotically along a horocycle as $\ve \to 0$). It follows that all the $2N$ critical points of $P_\ve$ lie in the basins of attraction of $w_j(\ve), \ov{w_j(\ve)}$ for $0 \leq j \leq N-1$. In particular, $K_{P_\ve}$ and therefore $K_P$ is connected. \vs 

Now each of the $2N$ fixed rays of $P_\ve$ must land at a repelling or parabolic fixed point of $P_\ve$. Since the fixed points of $P_\ve$ other than $w_N(\ve)$ are all attracting, it follows that the fixed rays of $P_\ve$ (in particular $R_{P_\ve,0}$) all land at $w_N(\ve)$. Thus, as $\ve \to 0$, the closed ray $\ov{R_{P_\ve,0}} =[w_N(\ve),+\infty]$ converges in the Hausdorff metric to  
$$
[0,+\infty] = [w_N, w_{N-1}] \cup \cdots \cup [w_1,w_0] \cup [w_0,+\infty],
$$ 
with the last interval $[w_0,+\infty]$ being the closed ray $\ov{R_{P,0}}$. \vs

It remains to construct a polynomial $P$ with the aforementioned properties. It will be more convenient in the notations that follow to label the points $0<w_{N-1}<\cdots<w_0$ in increasing order by setting $x_j:=w_{N-j}$. Let $C>0$ and $0<x_1<\cdots<x_N$. Define 
\begin{align*}
Q(z) & := \prod_{j=1}^N (z-x_j) \\
P(z) & := z+Cz(Q(z))^2.
\end{align*}
Evidently $P$ is a real polynomial with fixed points at $0$ and the $x_j$, and $P'(0)>1$ and $P'(x_j)=1$. Our goal is to find suitable $C, x_1, \ldots, x_N$ such that $\resit(P,x_j)<0$ for all $1 \leq j \leq N$. Once this is accomplished, we can conjugate $P$ via a real dilation to a real monic polynomial which will have the desired properties since the r\'esidu it\'eratif is invariant under analytic change of coordinates. \vs

Each $x_j$ is a parabolic fixed point of multiplicity $2$, so the formula \eqref{resitcomp} gives  
$$
\resit(P,x_j) = 1 - \iota(P,x_j) = 1- \frac{2}{3} \frac{P'''(x_j)}{(P''(x_j))^2}. 
$$  
Thus, we need to arrange for the inequalities 
$$
P'''(x_j)> \frac{3}{2} (P''(x_j))^2 \qquad (1 \leq j \leq N).
$$
A straightforward calculation shows 
\begin{align*}
P''(x_j) & = 2C \, x_j \, (Q'(x_j))^2\\
P'''(x_j) & = 6C \, (Q'(x_j))^2 + 6C \, x_j \, Q'(x_j)Q''(x_j),
\end{align*}
so the above inequalities translates to 
\begin{equation}\label{yek}
(Q'(x_j))^2 + x_j \, Q'(x_j)Q''(x_j) > C \, x_j^2 \, (Q'(x_j))^4 \qquad (1 \leq j \leq N). 
\end{equation}
To make these inequalities more explicit, we notice that 
\begin{align*}
\frac{Q'(z)}{Q(z)} & = \sum_{i=1}^N \frac{1}{z-x_i}\\
\frac{Q''(z)}{Q(z)} - \left(\frac{Q'(z)}{Q(z)} \right)^2 & = \sum_{i=1}^N \frac{-1}{(z-x_i)^2}.
\end{align*}
It follows, after routine algebra, that 
\begin{align*}
Q'(x_j) & = \prod_{i \neq j} (x_j-x_i) \\
Q''(x_j) & = 2 Q'(x_j) \sum_{i \neq j} \frac{1}{x_j-x_i}.
\end{align*}
Setting 
$$
H_j:= \sum_{i \neq j} \frac{1}{x_j-x_i},
$$
we can now write the desired inequalities \eqref{yek} in the form
\begin{equation}\label{do}
1+2x_j H_j > C \, x_j^2 \, (Q'(x_j))^2 \qquad (1 \leq j \leq N).
\end{equation}
It suffices to find $x_1,\ldots,x_N$ so that the weaker inequalities 
\begin{equation}\label{se}
1+2x_j H_j>0 \qquad (1 \leq j \leq N)
\end{equation}
hold, for then \eqref{do} can be achieved by choosing $C>0$ sufficiently small. \vs

To obtain \eqref{se}, define $x_1, \ldots, x_N$ recursively by  
\begin{align*}
x_1 &:= 1, \\ 
x_j &:= x_{j-1}+2^j \qquad (2 \leq j \leq N).
\end{align*}
We have 
$$
H_1 = - \frac{1}{2^2} - \frac{1}{2^2+2^3} - \cdots - \frac{1}{2^2+2^3+\cdots+2^N} > -\frac{1}{2}, 
$$
so $1+2x_1 H_1 >0$ and \eqref{se} holds for $j=1$. We claim that $H_j>0$ for $2 \leq j \leq N$, so \eqref{se} holds for these values of $j$ as well. In fact, $H_N>0$ trivially since every term in its defining sum is positive. On the other hand, if $2 \leq j \leq N-1$, then  
$$
\sum_{i<j} \frac{1}{x_j-x_i} > \frac{1}{x_j-x_{j-1}}=\frac{1}{2^j}
$$
while 
$$
\sum_{i>j} \frac{1}{x_j-x_i} = - \frac{1}{2^{j+1}} - \cdots - \frac{1}{2^{j+1}+\cdots+2^N} > -\frac{1}{2^j}.
$$
Adding the two inequalities, we obtain $H_j>0$. This completes the construction of $P$ and the proof of \thmref{D}.

\section*{Appendix. $C^1$ extensions of $\sL$-arcs at their extremities}

As noted in \S \ref{hhintro}, even though $\sL$-arcs are real-analytic curves, they typically fail to have $C^1$ extensions at either of their endpoints. This is easy to see if the endpoint is repelling with multiplier $\lambda$ such that $\lambda^q \notin \ ]1,+\infty[$ (of course only the point $w_N=\zeta_\infty$ can possibly be repelling). Below we explain the typical failure of $C^1$ extensions in the case the endpoint is parabolic by showing that unless the image of the $\sL$-arc in the corresponding Fatou coordinate is a straight line, the tangent direction to the arc must have non-diminishing oscillations near the parabolic endpoint. This behavior is already prevalent in the quadratic family and can be observed in the pictures of the external rays landing at $z=0$ of $z \mapsto e^{2\pi \ii p/q}z+z^2$ for large $q$. Since the problem is local, we may formulate it in the more general setting of a regular invariant curve $\ga$ for an analytic map $f$ with a parabolic fixed point of multiplier $1$ at the origin. We will assume $\ga$ is forward-invariant and approaches $0$ through an attracting petal; the case of a backward-invariant $\ga$ in a repelling petal is completely similar. Without loss of generality we can work with a parametrization $\ga: [0,+\infty[ \to \CC$ which satisfies $f(\gamma(t))=\gamma(t+1)$ and $\lim_{t \to +\infty} \ga(t)=0$. \vs   

Here is a caricature of the proof in the case of a simple (i.e., multiplicity $2$) parabolic fixed point. If $f(z)=z+z^2+O(z^3)$, the Fatou coordinate $\Phi$ mapping an attracting petal to a right half-plane and conjugating $f$ to the unit translation $T$ has the asymptotic $\Phi(z)= -1/z \, (1+o(1))$ as $z \to 0$. The image $\tilde{\ga}:=\Phi(\ga)$ is a $T$-invariant curve which satisfies $T(\tilde{\ga}(t))=\tilde{\ga}(t+1)$, so $\tilde{\ga}(t)-t$ is $1$-periodic and can be represented by a Fourier series. Consider the simplest test case where $\tilde{\ga}(t)-t$ is not constant, e.g., $\tilde{\ga}(t) = t + \ii \sin(2\pi t)$. Pulling back under $\Phi$ and ignoring small error terms then gives $\ga(t) = -1/(t+\ii \sin(2\pi t))$. Changing the time parameter to $s:=-1/t$, we obtain the re-parametrized curve $\ga(s) = 1/(1/s+\ii \sin(2\pi/s))$. It easily follows that $\ga'(s) = 1 + 2\pi \ii \cos(2\pi/s) + O(s)$ which clearly has no continuous extension as $s \to 0^-$. \vs

Let us now give a rigorous proof. Suppose the parabolic fixed point $z=0$ has multiplicity $m=r+1 \geq 2$. As noted in \S \ref{indres}, we can always put the map in the normal form 
$$
f(z) = z + z^{r+1} + b z^{2r+1}+O(z^{2r+2})
$$ 
for some $b \in \CC$. The Fatou coordinate $\Phi$ of $f$ can be constructed in two steps as follows. First we use the {\it pre-Fatou coordinate} $z \mapsto w=-1/(rz^r)$ to conjugate $f$ to a near-translation of the form 
$$
g(w) = w + 1 + \Big( \frac{r+1}{2}-b \Big) \frac{1}{r w} + O(w^{-1-1/r})
$$ 
(observe that by \eqref{resitcomp} the quantity $(r+1)/2-b$ is $\resit(f,0)$). We then conjugate $g$ on a right half-plane to the unit translation $T$ by a Fatou coordinate of the form 
$$
\phi(w) = w - \Big( \frac{r+1}{2} -b \Big) \log w + c + O(w^{-1/r}).
$$
The composition $\Phi(z):=\phi(-1/(rz^r))$ will then be a Fatou coordinate for $f$. Consider a regular $f$-invariant curve $\ga$ and its $T$-invariant image $\tilde{\ga} := \Phi(\ga) = \phi(-1/(r\ga^r))$ as before. The unit tangent $\tilde{\alpha}(t):= \tilde{\ga}'(t)/|\tilde{\ga}'(t)|$ is $1$-periodic and non-constant provided that $\tilde{\ga}$ is not a straight line. Assuming this is the case, it follows that $\lim_{t \to +\infty} \tilde{\alpha}(t)$ does not exist. Let $\Psi$ denote the branch of the inverse of $\Phi$ mapping a right half-plane conformally to an attracting petal containing $\ga$, so $\ga=\Psi(\tilde{\ga})$. Let $\sigma$ be the $r$-th root of $-1$ corresponding to the attracting direction along which $\gamma$ approaches $0$. Then $\Psi'(z)/|\Psi'(z)| \to \sigma$ uniformly as $z$ tends to $\infty$ within any horizontal strip of bounded height containing $\tilde{\ga}$. Now, the unit tangent $\alpha(t):=\ga'(t)/|\ga'(t)|$ satisfies 
$$
\alpha(t)= \frac{\Psi'(\tilde{\ga}(t))}{|\Psi'(\tilde{\ga}(t))|} \  \tilde{\alpha}(t). 
$$
As $t \to +\infty$ the first factor on the right tends to $\sigma$ while the second factor oscillates periodically without converging to a limit. This proves $\lim_{t \to +\infty} \alpha(t)$ does not exist and shows that $\ga$ cannot extend to $0$ as a $C^1$ curve.

\end{document}